\documentclass{article}%
\usepackage{amsfonts}
\usepackage{graphicx}
\usepackage{amsmath}%
\usepackage{amssymb}
\providecommand{\U}[1]{\protect\rule{.1in}{.1in}}
\newtheorem{theorem}{Theorem}

\newtheorem{corollary}[theorem]{Corollary}

\newtheorem{lemma}[theorem]{Lemma}
\newtheorem{notation}[theorem]{Notation}

\newtheorem{proposition}[theorem]{Proposition}

\newenvironment{proof}[1][Proof]{\textbf{#1.} }{\ \rule{0.5em}{0.5em}}
\begin{document}

\title{Axiomatic Differential Geometry II-3\\-Its Developments-\\Chapter 3: The General Jacobi Identity}
\author{Hirokazu Nishimura\\Institute of Mathematics\\University of Tsukuba\\Tsukuba, Ibaraki, 305-8571, JAPAN}
\maketitle

\begin{abstract}
As the fourth paper of our series of papers concerned with axiomatic
differential geometry, this paper is devoted to the general Jacobi identity
supporting the Jacobi identity of vector fields. The general Jacobi identity
can be regarded as one of the few fundamental results belonging properly to smootheology.

\end{abstract}

\section{Introduction}

It is well known in traditional differential geometry that the totality of
vector fields on a smooth manifold forms a Lie algebra. The proof of this fact
is tremendously easy, because we can identity vector fields with derivations
within the particular category that orthodox differential geometers have
indulged in. An axiomatic treatment of differential geometry emancipates
differential geometers from this comfortable adherence to their favorite
category of smooth manifolds and forces them to confront the infinitesimal
structure per se barehanded.

The Jacobi identity occupies the central position in the structure of a Lie
algebra, and we stumbled upon the \textit{general} Jacobi identity supporting
the Jacobi identity of vector fields from behind the very vale of
infinitesimal structures within the framework of synthetic differential
geometry in the previous century, for which the reader is referred to
\cite{nishi-1}, \cite{nishi-2} and \cite{nishi-3}. This paper is devoted to
the general Jacobi identity within our axiomatics of differential geometry,
which will play a predominant role in a subsequent paper dealing with the
Fr\"{o}licher-Nijenhuis calculus.

Our axiomatic differential geometry is an attempt to grasp the infinitesimal
structure without fringes or frills. It seems that the term
''\textit{smootheology}'' or ''\textit{diffeology}'' is gaining momentum for
the study of such an infinitesimal structure. We think that the general Jacobi
identity is one of the few fundamental results indigenous to smootheology.
This infinitesimal structure lies at the very core of not only differential
geometry but also many pure or applied branches of mathematics. We assume that
the reader is familiar with our axiomatic framework of differential geometry
presented in \cite{nishi-5} and \cite{nishi-7}. Thus we are working within a
DG-category
\[
\left(  \mathcal{K},\mathbb{R},\mathbf{T},\alpha\right)
\]
in the sense of \cite{nishi-7}. We always assume that $M$\ is a microlinear
and Weil-exponentiable object in the category $\mathcal{K}$.

The general Jacobi identity will be dealt with in Section \ref{s3.2}, which
will be preceded by the more elementary treatment of the primordial general
Jacobi identity in Section \ref{s3.1}. The final section is devoted to the
derivation of the Jacobi identity of vector fields from the general Jacobi
identity, in which the reader is assumed to be familiar with \cite{nishi-6}.

\section{Simplicial Sets}

We need to fix notation and terminology for simplicial objects, which form an
important subclass of infinitesimal objects. \textit{Simplicial objects} are
infinitesimal objects of the form
\begin{align*}
&  D^{n}\{\mathfrak{p}\}\\
&  =\{(d_{1},...,d_{n})\in D^{n}\mid d_{i_{1}}...d_{i_{k}}=0\text{ }%
(\forall(i_{1},...,i_{k})\in\mathfrak{p)\}}%
\end{align*}
where $\mathfrak{p}$\ is a finite set of finite sequences $(i_{1},...,i_{k}%
)$\ of natural numbers between $1$ and $n$, including the endpoints, with
$i_{1}<...<i_{k}$. If $\mathfrak{p}$\ is empty, $D^{n}\{\mathfrak{p}\}$ is
$D^{n}$ itself. If $\mathfrak{p}$ consists of all the binary sequences, then
$D^{n}\{\mathfrak{p}\}$ represents $D(n)$ in the standard terminology of SDG.
Given two simplicial objects $D^{m}\{\mathfrak{p}\}$\ and $D^{n}%
\{\mathfrak{q}\}$, we define a simplicial object $D^{m}\{\mathfrak{p}\}\oplus
D^{n}\{\mathfrak{q}\}$ to be
\[
D^{m+n}\{\mathfrak{p}\oplus\mathfrak{q}\}
\]
where
\begin{align*}
&  \mathfrak{p}\oplus\mathfrak{q}\\
&  =\mathfrak{p}\cup\{(j_{1}+m,...,j_{k}+m)\mid(j_{1},...,j_{k})\in
\mathfrak{q}\}\\
\cup\{(i,j+m) &  \mid1\leq i\leq m,\ 1\leq j\leq n\}
\end{align*}
Since the operation $\oplus$\ is associative, we can combine any finite number
of simplicial objects by $\oplus$ without bothering about how to insert
parentheses. Given morphisms of simplicial objects $\Phi_{i}:D^{m_{i}%
}\{\mathfrak{p}_{i}\}\rightarrow D^{m}\{\mathfrak{p}\}\ (1\leq i\leq n)$,
there exists a unique morphism of simplicial objects $\Phi:D^{m_{1}%
}\{\mathfrak{p}_{1}\}\oplus...\oplus D^{m_{n}}\{\mathfrak{p}_{n}\}\rightarrow
D^{m}\{\mathfrak{p}\}$ whose restriction to $D^{m_{i}}\{\mathfrak{p}_{i}\}$
coincides with $\Phi_{i}$ for each $i$. We denote this $\Phi$\ by $\Phi
_{1}\oplus...\oplus\Phi_{n}$. We write $D(n)$ for $\{(d,...,d)\in D^{n}\mid
d_{i}d_{j}=0$ for any $i\neq j\}$.

\section{\label{s3.1}The Preliminary Identity}

The principal objective in this paper is to give the general Jacobi identity
and its proof. Our harder treatment of the general Jacobi identity in the
coming section is preceded by a simpler treatment of the primordial general
Jacobi identity in this section, because the latter is easy to grasp
intuitively so that it prepares the reader for the coming general Jacobi identity.

\begin{proposition}
\label{t3.1}The diagram
\[%
\begin{array}
[c]{ccccc}
&  & \mathrm{id}_{M}\otimes\mathcal{W}_{\varphi} &  & \\
& M\otimes\mathcal{W}_{D^{3}\{(1,3),(2,3)\}} & \rightarrow & M\otimes
\mathcal{W}_{D^{2}} & \\
\mathrm{id}_{M}\otimes\mathcal{W}_{\psi} & \downarrow &  & \downarrow &
\mathrm{id}_{M}\otimes\mathcal{W}_{i_{D(2)}^{D^{2}}}\\
& M\otimes\mathcal{W}_{D^{2}} & \rightarrow & M\otimes\mathcal{W}_{D(2)} & \\
&  & \mathrm{id}_{M}\otimes\mathcal{W}_{i_{D(2)}^{D^{2}}} &  &
\end{array}
\]
is a pullback diagram, where the assumptive mapping $\varphi:D^{2}\rightarrow
D^{3}\{(1,3),(2,3)\}$ is
\[
(d_{1},d_{2})\in D^{2}\mapsto(d_{1},d_{2},0)\in D^{3}\{(1,3),(2,3)\}
\]
while the assumptive mapping $\psi:D^{2}\rightarrow D^{3}\{(1,3),(2,3)\}$ is
\[
(d_{1},d_{2})\in D^{2}\mapsto(d_{1},d_{2},d_{1}d_{2})\in D^{3}\{(1,3),(2,3)\}
\]

\end{proposition}

\begin{proof}
This follows from the microlinearity of $M$ and the pullback diagram of Weil
algebras
\[%
\begin{array}
[c]{ccccc}
&  & \mathcal{W}_{\varphi} &  & \\
& \mathcal{W}_{D^{3}\{(1,3),(2,3)\}} & \rightarrow & \mathcal{W}_{D^{2}} & \\
\mathcal{W}_{\psi} & \downarrow &  & \downarrow & \mathcal{W}_{i_{D(2)}%
^{D^{2}}}\\
& \mathcal{W}_{D^{2}} & \rightarrow & \mathcal{W}_{D(2)} & \\
&  & \mathcal{W}_{i_{D(2)}^{D^{2}}} &  &
\end{array}
\]

\end{proof}

\begin{corollary}
We have
\[
M\otimes\mathcal{W}_{D^{3}\{(1,3),(2,3)\}}=\left(  M\otimes\mathcal{W}_{D^{2}%
}\right)  \times_{M\otimes\mathcal{W}_{D(2)}}\left(  M\otimes\mathcal{W}%
_{D^{2}}\right)
\]

\end{corollary}

\begin{notation}
We will write
\[
\zeta^{\overset{\cdot}{-}}:\left(  M\otimes\mathcal{W}_{D^{2}}\right)
\times_{M\otimes\mathcal{W}_{D(2)}}\left(  M\otimes\mathcal{W}_{D^{2}}\right)
\rightarrow M\otimes\mathcal{W}_{D}%
\]
for the morphism
\begin{align*}
& \left(  \mathrm{id}_{M}\otimes\mathcal{W}_{d\in D\mapsto(0,0,d)\in
D^{3}\{(1,3),(2,3)\}},\mathcal{W}_{i_{D(2)}^{D^{2}}}\circ\mathcal{W}_{\varphi
}\right) \\
& :\mathcal{W}_{D^{2}}\times_{\mathcal{W}_{D(2)}}\mathcal{W}_{D^{2}%
}=\mathcal{W}_{D^{2}}\times_{\mathcal{W}_{D(2)}}\mathcal{W}_{D^{2}}\rightarrow
M\otimes\mathcal{W}_{D}%
\end{align*}

\end{notation}

The following is the prototype for the general Jacobi identity.

\begin{theorem}
\label{t3.2}(\underline{The Primordial General Jacobi Identity}) The three
morphisms
\begin{align*}
\zeta^{\ast_{2}\overset{\cdot}{-}\ast_{1}}  & :\underset{1}{\left(
M\otimes\mathcal{W}_{D^{2}}\right)  }\times_{M\otimes\mathcal{W}_{D(2)}%
}\underset{2}{\left(  M\otimes\mathcal{W}_{D^{2}}\right)  }\times
_{M\otimes\mathcal{W}_{D(2)}}\underset{3}{\left(  M\otimes\mathcal{W}_{D^{2}%
}\right)  }\\
& \rightarrow\underset{1}{\left(  M\otimes\mathcal{W}_{D^{2}}\right)  }%
\times_{M\otimes\mathcal{W}_{D(2)}}\underset{2}{\left(  M\otimes
\mathcal{W}_{D^{2}}\right)  }\,\underrightarrow{\zeta^{\overset{\cdot}{-}}%
}\,\left(  M\otimes\mathcal{W}_{D}\right)  \times\left(  M\otimes
\mathcal{W}_{D(2)}\right)
\end{align*}
\begin{align*}
\zeta^{\ast_{3}\overset{\cdot}{-}\ast_{2}}  & :\underset{1}{\left(
M\otimes\mathcal{W}_{D^{2}}\right)  }\times_{M\otimes\mathcal{W}_{D(2)}%
}\underset{2}{\left(  M\otimes\mathcal{W}_{D^{2}}\right)  }\times
_{M\otimes\mathcal{W}_{D(2)}}\underset{3}{\left(  M\otimes\mathcal{W}_{D^{2}%
}\right)  }\\
& \rightarrow\underset{2}{\left(  M\otimes\mathcal{W}_{D^{2}}\right)  }%
\times_{M\otimes\mathcal{W}_{D(2)}}\underset{3}{\left(  M\otimes
\mathcal{W}_{D^{2}}\right)  }\,\underrightarrow{\zeta^{\overset{\cdot}{-}}%
}\,\left(  M\otimes\mathcal{W}_{D}\right)  \times\left(  M\otimes
\mathcal{W}_{D(2)}\right)
\end{align*}
\begin{align*}
\zeta^{\ast_{1}\overset{\cdot}{-}\ast_{3}}  & :\underset{1}{\left(
M\otimes\mathcal{W}_{D^{2}}\right)  }\times_{M\otimes\mathcal{W}_{D(2)}%
}\underset{2}{\left(  M\otimes\mathcal{W}_{D^{2}}\right)  }\times
_{M\otimes\mathcal{W}_{D(2)}}\underset{3}{\left(  M\otimes\mathcal{W}_{D^{2}%
}\right)  }\\
& \rightarrow\underset{3}{\left(  M\otimes\mathcal{W}_{D^{2}}\right)  }%
\times_{M\otimes\mathcal{W}_{D(2)}}\underset{1}{\left(  M\otimes
\mathcal{W}_{D^{2}}\right)  }\,\underrightarrow{\zeta^{\overset{\cdot}{-}}%
}\,\left(  M\otimes\mathcal{W}_{D}\right)  \times\left(  M\otimes
\mathcal{W}_{D(2)}\right)
\end{align*}
sum up only to vanish, where the numbers under $\left(  M\otimes
\mathcal{W}_{D^{2}}\right)  $ are given simply so as for the reader to relate
each occurrence of $\left(  M\otimes\mathcal{W}_{D^{2}}\right)  $\ to another,
and the unlabeled arrows are the canonical projections.
\end{theorem}

The proof of Theorem \ref{t3.2} is based completely upon the following theorem.

\begin{theorem}
\label{t3.3}The diagram
\[%
\begin{array}
[c]{ccccccc}
& \mathrm{id}_{M}\otimes\mathcal{W}_{i_{D(2)}^{D^{2}}} &  & M\otimes
\mathcal{W}_{D^{2}} &  & \mathrm{id}_{M}\otimes\mathcal{W}_{i_{D(2)}^{D^{2}}}
& \\
&  & \swarrow & \uparrow & \searrow &  & \\
& M\otimes\mathcal{W}_{D(2)} &  & M\otimes\mathcal{W}_{E} &  & M\otimes
\mathcal{W}_{D(2)} & \\
\mathrm{id}_{M}\otimes\mathcal{W}_{i_{D(2)}^{D^{2}}} & \uparrow & \swarrow &
& \searrow & \uparrow & \mathrm{id}_{M}\otimes\mathcal{W}_{i_{D(2)}^{D^{2}}}\\
& M\otimes\mathcal{W}_{D^{2}} &  &  &  & M\otimes\mathcal{W}_{D^{2}} & \\
&  & \searrow &  & \swarrow &  & \\
& \mathrm{id}_{M}\otimes\mathcal{W}_{i_{D(2)}^{D^{2}}} &  & M\otimes
\mathcal{W}_{D(2)} &  & \mathrm{id}_{M}\otimes\mathcal{W}_{i_{D(2)}^{D^{2}}}
&
\end{array}
\]
is a limit diagram, where the assumptive object $E$ is
\[
D^{4}\{(1,3),(2,3),(1,4),(2,4),(3,4)\}
\]
and the assumptive mapping $i_{D(2)}^{D^{2}}:D(2)\rightarrow D^{2}$ is
$(d_{1},d_{2})\in D(2)\mapsto(d_{1},d_{2})\in D^{2}$, while the three unnamed
arrows $M\otimes\mathcal{W}_{E}\rightarrow M\otimes\mathcal{W}_{D^{2}}$ are
$\mathrm{id}_{M}\otimes\mathcal{W}_{l_{i}}$ $(i=1,2,3)$ counterclockwise from
the top with the assumptive mappings $l_{i}:D^{2}\rightarrow E$ $(i=1,2,3)$
being
\begin{align*}
l_{1} &  :(d_{1},d_{2})\in D^{2}\mapsto(d_{1},d_{2},0,0)\in E\\
l_{2} &  :(d_{1},d_{2})\in D^{2}\mapsto(d_{1},d_{2},d_{1}d_{2},0)\in E\\
l_{3} &  :(d_{1},d_{2})\in D^{2}\mapsto(d_{1},d_{2},0,d_{1}d_{2})\in E
\end{align*}

\end{theorem}

\begin{corollary}
We have
\begin{align*}
& \left(  M\otimes\mathcal{W}_{D^{2}}\right)  \times_{M\otimes\mathcal{W}%
_{D(2)}}\left(  M\otimes\mathcal{W}_{D^{2}}\right)  \times_{M\otimes
\mathcal{W}_{D(2)}}\left(  M\otimes\mathcal{W}_{D^{2}}\right) \\
& =M\otimes\mathcal{W}_{E}%
\end{align*}

\end{corollary}

This theorem follows directly from the following lemma.

\begin{lemma}
\label{t3.4}The following diagram is a limit diagram of Weil algebras:
\[
\
\begin{array}
[c]{ccccccc}
& \mathcal{W}_{i_{D(2)}^{D^{2}}} &  & \mathcal{W}_{D^{2}} &  & \mathcal{W}%
_{i_{D(2)}^{D^{2}}} & \\
&  & \swarrow & \uparrow & \searrow &  & \\
& \mathcal{W}_{D(2)} &  & \mathcal{W}_{E} &  & \mathcal{W}_{D(2)} & \\
\mathcal{W}_{i_{D(2)}^{D^{2}}} & \uparrow & \swarrow &  & \searrow & \uparrow
& \mathcal{W}_{i_{D(2)}^{D^{2}}}\\
& \mathcal{W}_{D^{2}} &  &  &  & \mathcal{W}_{D^{2}} & \\
&  & \searrow &  & \swarrow &  & \\
& \mathcal{W}_{i_{D(2)}^{D^{2}}} &  & \mathcal{W}_{D(2)} &  & \mathcal{W}%
_{i_{D(2)}^{D^{2}}} &
\end{array}
\]

\end{lemma}

\begin{proof}
Let $\gamma_{1},\gamma_{2},\gamma_{3}\in\mathcal{W}_{D^{2}}$ and $\gamma
\in\mathcal{W}_{E}$ so that they are the polynomials with coefficients in $k
$\ of the following forms:
\begin{align*}
\gamma_{1}(X_{1},X_{2}) &  =a+a_{1}X_{1}+a_{2}X_{2}+a_{12}X_{1}X_{2}\\
\gamma_{2}(X_{1},X_{2}) &  =b+b_{1}X_{1}+b_{2}X_{2}+b_{12}X_{1}X_{2}\\
\gamma_{3}(X_{1},X_{2}) &  =c+c_{1}X_{1}+c_{2}X_{2}+c_{12}X_{1}X_{2}\\
\gamma(X_{1},X_{2},X_{3},X_{4}) &  =e+e_{1}X_{1}+e_{2}X_{2}+e_{12}X_{1}%
X_{2}+e_{3}X_{3}+e_{4}X_{4}%
\end{align*}
The condition that $\mathcal{W}_{i_{D(2)}^{D^{2}}}(\gamma_{1})=\mathcal{W}%
_{i_{D(2)}^{D^{2}}}(\gamma_{2})=\mathcal{W}_{i_{D(2)}^{D^{2}}}(\gamma_{3})$ is
equivalent to the following three conditions as a whole:
\begin{align*}
a &  =b=c\\
a_{1} &  =b_{1}=c_{1}\\
a_{2} &  =b_{2}=c_{2}%
\end{align*}
Therefore, in order that $\mathcal{W}_{l_{1}}(\gamma)=\gamma_{1}$,
$\mathcal{W}_{l_{2}}(\gamma)=\gamma_{2}$ and $\mathcal{W}_{l_{3}}%
(\gamma)=\gamma_{3}$ in this case, it is necessary and sufficient that the
polynomial $\gamma$ should be of the following form:
\[
\gamma(X_{1},X_{2},X_{3},X_{4})=a+a_{1}X_{1}+a_{2}X_{2}+a_{12}X_{1}%
X_{2}+(b_{12}-a_{12})X_{3}+(c_{12}-a_{12})X_{4}%
\]
This completes the proof.
\end{proof}

\begin{theorem}
\label{t3.5}The diagram
\[%
\begin{array}
[c]{ccccccc}
& \mathrm{id}_{M}\otimes\mathcal{W}_{\varphi} &  & M\otimes\mathcal{W}_{C} &
& \mathrm{id}_{M}\otimes\mathcal{W}_{\psi} & \\
&  & \swarrow & \uparrow & \searrow &  & \\
& M\otimes\mathcal{W}_{D^{2}} &  & M\otimes\mathcal{W}_{E} &  & M\otimes
\mathcal{W}_{D^{2}} & \\
\mathrm{id}_{M}\otimes\mathcal{W}_{\psi} & \uparrow & \swarrow &  & \searrow &
\uparrow & \mathrm{id}_{M}\otimes\mathcal{W}_{\varphi}\\
& M\otimes\mathcal{W}_{C} &  &  &  & M\otimes\mathcal{W}_{C} & \\
&  & \searrow &  & \swarrow &  & \\
& \mathrm{id}_{M}\otimes\mathcal{W}_{\varphi} &  & M\otimes\mathcal{W}_{D^{2}}
&  & \mathrm{id}_{M}\otimes\mathcal{W}_{\psi} &
\end{array}
\]
is a limit diagram, where $C$\ stands for
\[
D^{3}\{(1,3),(2,3)\}
\]
and the three unnamed morphisms go contraclockwise from the top as follows:
\begin{align*}
& \mathrm{id}_{M}\otimes\mathcal{W}_{(d_{1},d_{2},d_{3})\in D^{3}%
\{(1,3),(2,3)\}\mapsto(d_{1},d_{2},d_{3},0)\in E}:M\otimes\mathcal{W}%
_{E}\rightarrow M\otimes\mathcal{W}_{C}\\
& \mathrm{id}_{M}\otimes\mathcal{W}_{(d_{1},d_{2},d_{3})\in D^{3}%
\{(1,3),(2,3)\}\mapsto(d_{1},d_{2},d_{1}d_{2}-d_{3},d_{3})\in E}%
:M\otimes\mathcal{W}_{E}\rightarrow M\otimes\mathcal{W}_{C}\\
& \mathrm{id}_{M}\otimes\mathcal{W}_{(d_{1},d_{2},d_{3})\in D^{3}%
\{(1,3),(2,3)\}\mapsto(d_{1},d_{2},0,d_{1}d_{2}-d_{3})\in E}:M\otimes
\mathcal{W}_{E}\rightarrow M\otimes\mathcal{W}_{C}%
\end{align*}

\end{theorem}

This theorem follows directly from the following lemma.

\begin{lemma}
\label{t3.6}The diagram
\[%
\begin{array}
[c]{ccccccc}
& \mathcal{W}_{\psi} &  & \mathcal{W}_{C} &  & \mathcal{W}_{\varphi} & \\
&  & \swarrow & \uparrow & \searrow &  & \\
& \mathcal{W}_{D^{2}} &  & \mathcal{W}_{E} &  & \mathcal{W}_{D^{2}} & \\
\mathcal{W}_{\varphi} & \uparrow & \swarrow &  & \searrow & \uparrow &
\mathcal{W}_{\psi}\\
& \mathcal{W}_{C} &  &  &  & \mathcal{W}_{C} & \\
&  & \searrow &  & \swarrow &  & \\
& \mathcal{W}_{\psi} &  & \mathcal{W}_{D^{2}} &  & \mathcal{W}_{\varphi} &
\end{array}
\]
is a limit diagram, where the three unnamed morphisms go contraclockwise from
the top as follows:
\begin{align*}
& \mathcal{W}_{(d_{1},d_{2},d_{3})\in D^{3}\{(1,3),(2,3)\}\mapsto(d_{1}%
,d_{2},d_{3},0)\in E}:\mathcal{W}_{E}\rightarrow\mathcal{W}_{C}\\
& \mathcal{W}_{(d_{1},d_{2},d_{3})\in D^{3}\{(1,3),(2,3)\}\mapsto(d_{1}%
,d_{2},d_{1}d_{2}-d_{3},d_{3})\in E}:\mathcal{W}_{E}\rightarrow\mathcal{W}%
_{C}\\
& \mathcal{W}_{(d_{1},d_{2},d_{3})\in D^{3}\{(1,3),(2,3)\}\mapsto(d_{1}%
,d_{2},0,d_{1}d_{2}-d_{3})\in E}:\mathcal{W}_{E}\rightarrow\mathcal{W}_{C}%
\end{align*}

\end{lemma}

\begin{proof}
By the same token as in Lemma \ref{t3.4}.
\end{proof}

\begin{proof}
(of the primordial Jacobi identity). The morphism
\[
\zeta^{\ast_{2}\overset{\cdot}{-}\ast_{1}}:\underset{1}{\left(  M\otimes
\mathcal{W}_{D^{2}}\right)  }\times_{M\otimes\mathcal{W}_{D(2)}}\underset
{2}{\left(  M\otimes\mathcal{W}_{D^{2}}\right)  }\times_{M\otimes
\mathcal{W}_{D(2)}}\underset{3}{\left(  M\otimes\mathcal{W}_{D^{2}}\right)
}\rightarrow M\otimes\mathcal{W}_{D}%
\]
is the composition of
\begin{align*}
& \mathrm{id}_{M}\otimes\mathcal{W}_{(d_{1},d_{2},d_{3})\in D^{3}%
\{(1,3),(2,3)\}\mapsto(d_{1},d_{2},d_{3},0)\in E}\\
& :\underset{1}{\left(  M\otimes\mathcal{W}_{D^{2}}\right)  }\times
_{M\otimes\mathcal{W}_{D(2)}}\underset{2}{\left(  M\otimes\mathcal{W}_{D^{2}%
}\right)  }\times_{M\otimes\mathcal{W}_{D(2)}}\underset{3}{\left(
M\otimes\mathcal{W}_{D^{2}}\right)  }=M\otimes\mathcal{W}_{E}\rightarrow
M\otimes\mathcal{W}_{C}\\
& =\underset{1}{\left(  M\otimes\mathcal{W}_{D^{2}}\right)  }\times
_{M\otimes\mathcal{W}_{D(2)}}\underset{2}{\left(  M\otimes\mathcal{W}_{D^{2}%
}\right)  }%
\end{align*}
and
\[
\underset{1}{\zeta^{\overset{\cdot}{-}}:\left(  M\otimes\mathcal{W}_{D^{2}%
}\right)  }\times_{M\otimes\mathcal{W}_{D(2)}}\underset{2}{\left(
M\otimes\mathcal{W}_{D^{2}}\right)  }\rightarrow M\otimes\mathcal{W}_{D}%
\]
in succession, which is in turn equivalent to
\[
\mathrm{id}_{M}\otimes\mathcal{W}_{d\in D\mapsto\left(  0,0,d,0\right)  \in E}%
\]
The morphism
\[
\zeta^{\ast_{3}\overset{\cdot}{-}\ast_{2}}:\underset{1}{\left(  M\otimes
\mathcal{W}_{D^{2}}\right)  }\times_{M\otimes\mathcal{W}_{D(2)}}\underset
{2}{\left(  M\otimes\mathcal{W}_{D^{2}}\right)  }\times_{M\otimes
\mathcal{W}_{D(2)}}\underset{3}{\left(  M\otimes\mathcal{W}_{D^{2}}\right)
}\rightarrow M\otimes\mathcal{W}_{D}%
\]
is the composition of
\begin{align*}
& \mathrm{id}_{M}\otimes\mathcal{W}_{(d_{1},d_{2},d_{3})\in D^{3}%
\{(1,3),(2,3)\}\mapsto(d_{1},d_{2},d_{1}d_{2}-d_{3},d_{3})\in E}\\
& :\underset{1}{\left(  M\otimes\mathcal{W}_{D^{2}}\right)  }\times
_{M\otimes\mathcal{W}_{D(2)}}\underset{2}{\left(  M\otimes\mathcal{W}_{D^{2}%
}\right)  }\times_{M\otimes\mathcal{W}_{D(2)}}\underset{3}{\left(
M\otimes\mathcal{W}_{D^{2}}\right)  }=M\otimes\mathcal{W}_{E}\rightarrow
M\otimes\mathcal{W}_{C}\\
& =\underset{2}{\left(  M\otimes\mathcal{W}_{D^{2}}\right)  }\times
_{M\otimes\mathcal{W}_{D(2)}}\underset{3}{\left(  M\otimes\mathcal{W}_{D^{2}%
}\right)  }%
\end{align*}
and
\[
\underset{2}{\zeta^{\overset{\cdot}{-}}:\left(  M\otimes\mathcal{W}_{D^{2}%
}\right)  }\times_{M\otimes\mathcal{W}_{D(2)}}\underset{3}{\left(
M\otimes\mathcal{W}_{D^{2}}\right)  }\rightarrow M\otimes\mathcal{W}_{D}%
\]
in succession, which is in turn equivalent to
\[
\mathrm{id}_{M}\otimes\mathcal{W}_{d\in D\mapsto\left(  0,0,-d,d\right)  \in
E}%
\]
The morphism
\[
\zeta^{\ast_{1}\overset{\cdot}{-}\ast_{3}}:\underset{1}{\left(  M\otimes
\mathcal{W}_{D^{2}}\right)  }\times_{M\otimes\mathcal{W}_{D(2)}}\underset
{2}{\left(  M\otimes\mathcal{W}_{D^{2}}\right)  }\times_{M\otimes
\mathcal{W}_{D(2)}}\underset{3}{\left(  M\otimes\mathcal{W}_{D^{2}}\right)
}\rightarrow M\otimes\mathcal{W}_{D}%
\]
is the composition of
\begin{align*}
& \mathrm{id}_{M}\otimes\mathcal{W}_{(d_{1},d_{2},d_{3})\in D^{3}%
\{(1,3),(2,3)\}\mapsto(d_{1},d_{2},0,d_{1}d_{2}-d_{3})\in E}\\
& :\underset{1}{\left(  M\otimes\mathcal{W}_{D^{2}}\right)  }\times
_{M\otimes\mathcal{W}_{D(2)}}\underset{2}{\left(  M\otimes\mathcal{W}_{D^{2}%
}\right)  }\times_{M\otimes\mathcal{W}_{D(2)}}\underset{3}{\left(
M\otimes\mathcal{W}_{D^{2}}\right)  }=M\otimes\mathcal{W}_{E}\rightarrow
M\otimes\mathcal{W}_{C}\\
& =\underset{3}{\left(  M\otimes\mathcal{W}_{D^{2}}\right)  }\times
_{M\otimes\mathcal{W}_{D(2)}}\underset{1}{\left(  M\otimes\mathcal{W}_{D^{2}%
}\right)  }%
\end{align*}
and
\[
\underset{3}{\left(  M\otimes\mathcal{W}_{D^{2}}\right)  }\times
_{M\otimes\mathcal{W}_{D(2)}}\underset{1}{\left(  M\otimes\mathcal{W}_{D^{2}%
}\right)  }\rightarrow M\otimes\mathcal{W}_{D}%
\]
in succession, which is in turn equivalent to
\[
\mathrm{id}_{M}\otimes\mathcal{W}_{d\in D\mapsto\left(  0,0,0,-d\right)  \in
E}%
\]
Therefore
\begin{align*}
& \zeta^{\ast_{2}\overset{\cdot}{-}\ast_{1}}+\zeta^{\ast_{3}\overset{\cdot}%
{-}\ast_{2}}+\zeta^{\ast_{1}\overset{\cdot}{-}\ast_{3}}\\
& :\underset{1}{\left(  M\otimes\mathcal{W}_{D^{2}}\right)  }\times
_{M\otimes\mathcal{W}_{D(2)}}\underset{2}{\left(  M\otimes\mathcal{W}_{D^{2}%
}\right)  }\times_{M\otimes\mathcal{W}_{D(2)}}\underset{3}{\left(
M\otimes\mathcal{W}_{D^{2}}\right)  }\rightarrow M\otimes\mathcal{W}_{D}%
\end{align*}
is equivalent to
\begin{align*}
& \left(  \mathrm{id}_{M}\otimes\mathcal{W}_{d\in D\mapsto\left(
d,d,d\right)  \in D\left(  3\right)  }\right)  \circ\left(  \mathrm{id}%
_{M}\otimes\mathcal{W}_{\left(  d_{1},d_{2},d_{3}\right)  \in D\left(
3\right)  \mapsto\left(  0,0,d_{1}-d_{2},d_{2}-d_{3}\right)  \in E}\right) \\
& =\mathrm{id}_{M}\otimes\left(  \mathcal{W}_{d\in D\mapsto\left(
d,d,d\right)  \in D\left(  3\right)  }\circ\mathcal{W}_{\left(  d_{1}%
,d_{2},d_{3}\right)  \in D\left(  3\right)  \mapsto\left(  0,0,d_{1}%
-d_{2},d_{2}-d_{3}\right)  \in E}\right) \\
& =\mathrm{id}_{M}\otimes\mathcal{W}_{d\in D\mapsto\left(  0,0,0,0\right)  \in
E}%
\end{align*}
This completes the proof.
\end{proof}

\section{\label{s3.2}The Main Identity}

\begin{proposition}
\label{t3.7}The diagram
\[%
\begin{array}
[c]{ccccc}
&  & \mathrm{id}_{M}\otimes\mathcal{W}_{\varphi_{1}^{3}} &  & \\
& M\otimes\mathcal{W}_{D^{4}\{(2,4),(3,4)\}} & \rightarrow & M\otimes
\mathcal{W}_{D^{3}} & \\
\mathrm{id}_{M}\otimes\mathcal{W}_{\psi_{1}^{3}} & \downarrow &  & \downarrow
& \mathrm{id}_{M}\otimes\mathcal{W}_{i_{D^{3}\{(2,3)\}}^{D^{3}}}\\
& M\otimes\mathcal{W}_{D^{3}} & \rightarrow & M\otimes\mathcal{W}%
_{D^{3}\{(2,3)\}} & \\
&  & \mathrm{id}_{M}\otimes\mathcal{W}_{i_{D^{3}\{(2,3)\}}^{D^{3}}} &  &
\end{array}
\]
is a pullback diagram, where the assumptive mapping $\varphi_{1}^{3}%
:D^{3}\rightarrow D^{4}\{(2,4),(3,4)\}$ is
\[
(d_{1},d_{2},d_{3})\in D^{3}\mapsto(d_{1},d_{2},d_{3},0)\in D^{4}%
\{(2,4),(3,4)\}
\]
while the assumptive mapping $\psi_{1}^{3}:D^{3}\rightarrow D^{4}%
\{(2,4),(3,4)\}$ is
\[
(d_{1},d_{2},d_{3})\in D^{3}\mapsto(d_{1},d_{2},d_{3},d_{2}d_{3})\in
D^{4}\{(2,4),(3,4)\}
\]

\end{proposition}

\begin{proof}
This follows from the microlinearity of $M$ and the pullback diagram of Weil
algebras
\[%
\begin{array}
[c]{ccccc}
&  & \mathcal{W}_{\varphi_{1}^{3}} &  & \\
& \mathcal{W}_{D^{4}\{(2,4),(3,4)\}} & \rightarrow & \mathcal{W}_{D^{3}} & \\
\mathcal{W}_{\psi_{1}^{3}} & \downarrow &  & \downarrow & \mathcal{W}%
_{i_{D^{3}\{(2,3)\}}^{D^{3}}}\\
& \mathcal{W}_{D^{3}} & \rightarrow & \mathcal{W}_{D^{3}\{(2,3)\}} & \\
&  & \mathcal{W}_{i_{D^{3}\{(2,3)\}}^{D^{3}}} &  &
\end{array}
\]

\end{proof}

\begin{corollary}
We have
\begin{align*}
& \left(  M\otimes\mathcal{W}_{D^{3}}\right)  \times_{M\otimes\mathcal{W}%
_{D^{3}\{(2,3)\}}}\left(  M\otimes\mathcal{W}_{D^{3}}\right) \\
& =M\otimes\mathcal{W}_{D^{4}\{(2,4),(3,4)\}}%
\end{align*}
with the diagrams
\[%
\begin{array}
[c]{ccc}%
\underset{1}{\left(  M\otimes\mathcal{W}_{D^{3}}\right)  }\times
_{M\otimes\mathcal{W}_{D^{3}\{(2,3)\}}}\underset{2}{\left(  M\otimes
\mathcal{W}_{D^{3}}\right)  } & \rightarrow & \underset{1}{M\otimes
\mathcal{W}_{D^{3}}}\\
\parallel &  & \parallel\\
M\otimes\mathcal{W}_{D^{4}\{(2,4),(3,4)\}} & \rightarrow & \underset
{1}{M\otimes\mathcal{W}_{D^{3}}}\\
& \mathrm{id}_{M}\otimes\mathcal{W}_{\varphi_{1}^{3}} &
\end{array}
\]
and
\[%
\begin{array}
[c]{ccc}%
\underset{1}{\left(  M\otimes\mathcal{W}_{D^{3}}\right)  }\times
_{M\otimes\mathcal{W}_{D^{3}\{(2,3)\}}}\underset{2}{\left(  M\otimes
\mathcal{W}_{D^{3}}\right)  } & \rightarrow & \underset{2}{M\otimes
\mathcal{W}_{D^{3}}}\\
\parallel &  & \parallel\\
M\otimes\mathcal{W}_{D^{4}\{(2,4),(3,4)\}} & \rightarrow & \underset
{2}{M\otimes\mathcal{W}_{D^{3}}}\\
& \mathrm{id}_{M}\otimes\mathcal{W}_{\psi_{1}^{3}} &
\end{array}
\]
being commutative, where the unnamed arrows are canonical projections.
\end{corollary}

\begin{notation}
We will write
\[
\zeta^{\underset{1}{\overset{\cdot}{-}}}:\left(  M\otimes\mathcal{W}_{D^{3}%
}\right)  \times_{M\otimes\mathcal{W}_{D^{3}\{(2,3)\}}}\left(  M\otimes
\mathcal{W}_{D^{3}}\right)  \rightarrow M\otimes\mathcal{W}_{D^{2}}%
\]
for the morphism
\begin{align*}
& \mathrm{id}_{M}\otimes\mathcal{W}_{(d_{1},d_{2})\in D^{2}\mapsto
(d_{1},0,0,d_{2})\in D^{4}\{(2,4),(3,4)\}}\\
& :\left(  M\otimes\mathcal{W}_{D^{3}}\right)  \times_{M\otimes\mathcal{W}%
_{D^{3}\{(2,3)\}}}\left(  M\otimes\mathcal{W}_{D^{3}}\right) \\
& =M\otimes\mathcal{W}_{D^{4}\{(2,4),(3,4)\}}\\
& \rightarrow M\otimes\mathcal{W}_{D^{2}}%
\end{align*}

\end{notation}

\begin{proposition}
\label{t3.8}The diagram
\[%
\begin{array}
[c]{ccc}%
M\otimes\mathcal{W}_{D^{4}\{(1,4),(3,4)\}} & \underrightarrow{\mathrm{id}%
_{M}\otimes\mathcal{W}_{\varphi_{2}^{3}}} & M\otimes\mathcal{W}_{D^{3}}\\
\mathrm{id}_{M}\otimes\mathcal{W}_{\psi_{2}^{3}}\downarrow &  & \downarrow
\mathrm{id}_{M}\otimes\mathcal{W}_{i_{D^{3}\{(1,3)\}}^{D^{3}}}\\
M\otimes\mathcal{W}_{D^{3}} & \overrightarrow{\mathrm{id}_{M}\otimes
\mathcal{W}_{i_{D^{3}\{(1,3)\}}^{D^{3}}}} & M\otimes\mathcal{W}_{D^{3}%
\{(1,3)\}}%
\end{array}
\]
is a pullback diagram, where the assumptive mapping $\varphi_{2}^{3}%
:D^{3}\rightarrow D^{4}\{(1,4),(3,4)\}$ is
\[
(d_{1},d_{2},d_{3})\in D^{3}\mapsto(d_{1},d_{2},d_{3},0)\in D^{4}%
\{(1,4),(3,4)\}
\]
while the assumptive mapping $\psi_{2}^{3}:D^{3}\rightarrow D^{4}%
\{(1,4),(3,4)\}$ is
\[
(d_{1},d_{2},d_{3})\in D^{3}\mapsto(d_{1},d_{2},d_{3},d_{1}d_{3})\in
D^{4}\{(1,4),(3,4)\}
\]

\end{proposition}

\begin{proof}
This follows from the microlinearity of $M$ and the pullback diagram of Weil
algebras
\[%
\begin{array}
[c]{ccccc}
&  & \mathcal{W}_{\varphi_{2}^{3}} &  & \\
& \mathcal{W}_{D^{4}\{(1,4),(3,4)\}} & \rightarrow & \mathcal{W}_{D^{3}} & \\
\mathcal{W}_{\psi_{2}^{3}} & \downarrow &  & \downarrow & \mathcal{W}%
_{i_{D^{3}\{(1,3)\}}^{D^{3}}}\\
& \mathcal{W}_{D^{3}} & \rightarrow & \mathcal{W}_{D^{3}\{(1,3)\}} & \\
&  & \mathcal{W}_{i_{D^{3}\{(1,3)\}}^{D^{3}}} &  &
\end{array}
\]

\end{proof}

\begin{corollary}
We have
\begin{align*}
& \left(  M\otimes\mathcal{W}_{D^{3}}\right)  \times_{M\otimes\mathcal{W}%
_{D^{3}\{(1,3)\}}}\left(  M\otimes\mathcal{W}_{D^{3}}\right) \\
& =M\otimes\mathcal{W}_{D^{4}\{(1,4),(3,4)\}}%
\end{align*}
with the diagrams
\[%
\begin{array}
[c]{ccc}%
\underset{1}{\left(  M\otimes\mathcal{W}_{D^{3}}\right)  }\times
_{M\otimes\mathcal{W}_{D^{3}\{(1,3)\}}}\underset{2}{\left(  M\otimes
\mathcal{W}_{D^{3}}\right)  } & \rightarrow & \underset{1}{M\otimes
\mathcal{W}_{D^{3}}}\\
\parallel &  & \parallel\\
M\otimes\mathcal{W}_{D^{4}\{(1,4),(3,4)\}} & \rightarrow & \underset
{1}{M\otimes\mathcal{W}_{D^{3}}}\\
& \mathrm{id}_{M}\otimes\mathcal{W}_{\varphi_{2}^{3}} &
\end{array}
\]
and
\[%
\begin{array}
[c]{ccc}%
\underset{1}{\left(  M\otimes\mathcal{W}_{D^{3}}\right)  }\times
_{M\otimes\mathcal{W}_{D^{3}\{(1,3)\}}}\underset{2}{\left(  M\otimes
\mathcal{W}_{D^{3}}\right)  } & \rightarrow & \underset{2}{M\otimes
\mathcal{W}_{D^{3}}}\\
\parallel &  & \parallel\\
M\otimes\mathcal{W}_{D^{4}\{(1,4),(3,4)\}} & \rightarrow & \underset
{2}{M\otimes\mathcal{W}_{D^{3}}}\\
& \mathrm{id}_{M}\otimes\mathcal{W}_{\psi_{2}^{3}} &
\end{array}
\]
being commutative, where unnamed arrows are canonical projections.
\end{corollary}

\begin{notation}
We will write
\[
\zeta^{\underset{2}{\overset{\cdot}{-}}}:\left(  M\otimes\mathcal{W}_{D^{3}%
}\right)  \times_{M\otimes\mathcal{W}_{D^{3}\{(1,3)\}}}\left(  M\otimes
\mathcal{W}_{D^{3}}\right)  \rightarrow M\otimes\mathcal{W}_{D^{2}}%
\]
for the morphism
\begin{align*}
& \mathrm{id}_{M}\otimes\mathcal{W}_{(d_{1},d_{2})\in D^{2}\mapsto
(0,d_{1},0,d_{2})\in D^{4}\{(1,4),(3,4)\}}\\
& :\left(  M\otimes\mathcal{W}_{D^{3}}\right)  \times_{M\otimes\mathcal{W}%
_{D^{3}\{(1,3)\}}}\left(  M\otimes\mathcal{W}_{D^{3}}\right) \\
& =M\otimes\mathcal{W}_{D^{4}\{(1,4),(3,4)\}}\\
& \rightarrow M\otimes\mathcal{W}_{D^{2}}%
\end{align*}

\end{notation}

\begin{proposition}
\label{t3.9}The diagram
\[%
\begin{array}
[c]{ccccc}
&  & \mathrm{id}_{M}\otimes\mathcal{W}_{\varphi_{3}^{3}} &  & \\
& M\otimes\mathcal{W}_{D^{4}\{(1,4),(2,4)\}} & \rightarrow & M\otimes
\mathcal{W}_{D^{3}} & \\
\mathrm{id}_{M}\otimes\mathcal{W}_{\psi_{3}^{3}} & \downarrow &  & \downarrow
& \mathrm{id}_{M}\otimes\mathcal{W}_{i_{D^{3}\{(1,2)\}}^{D^{3}}}\\
& M\otimes\mathcal{W}_{D^{3}} & \rightarrow & M\otimes\mathcal{W}%
_{D^{3}\{(1,2)\}} & \\
&  & \mathrm{id}_{M}\otimes\mathcal{W}_{i_{D^{3}\{(1,2)\}}^{D^{3}}} &  &
\end{array}
\]
is a pullback diagram, where the assumptive mapping $\varphi_{3}^{3}%
:D^{3}\rightarrow D^{4}\{(1,4),(2,4)\}$ is
\[
(d_{1},d_{2},d_{3})\in D^{3}\mapsto(d_{1},d_{2},d_{3},0)\in D^{4}%
\{(1,4),(2,4)\}
\]
while the assumptive mapping $\psi_{3}^{3}:D^{3}\rightarrow D^{4}%
\{(1,4),(2,4)\}$ is
\[
(d_{1},d_{2},d_{3})\in D^{3}\mapsto(d_{1},d_{2},d_{3},d_{1}d_{2})\in
D^{4}\{(1,4),(2,4)\}
\]

\end{proposition}

\begin{proof}
This follows from the microlinearity of $M$ and the pullback diagram of Weil
algebras
\[%
\begin{array}
[c]{ccccc}
&  & \mathcal{W}_{\varphi_{3}^{3}} &  & \\
& \mathcal{W}_{D^{4}\{(1,4),(2,4)\}} & \rightarrow & \mathcal{W}_{D^{3}} & \\
\mathcal{W}_{\psi_{3}^{3}} & \downarrow &  & \downarrow & \mathcal{W}%
_{i_{D^{3}\{(1,2)\}}^{D^{3}}}\\
& \mathcal{W}_{D^{3}} & \rightarrow & \mathcal{W}_{D^{3}\{(1,2)\}} & \\
&  & \mathcal{W}_{i_{D^{3}\{(1,2)\}}^{D^{3}}} &  &
\end{array}
\]

\end{proof}

\begin{corollary}
We have
\begin{align*}
& \left(  M\otimes\mathcal{W}_{D^{3}}\right)  \times_{M\otimes\mathcal{W}%
_{D^{3}\{(1,2)\}}}\left(  M\otimes\mathcal{W}_{D^{3}}\right) \\
& =M\otimes\mathcal{W}_{D^{4}\{(1,4),(2,4)\}}%
\end{align*}
with the diagrams
\[%
\begin{array}
[c]{ccc}%
\underset{1}{\left(  M\otimes\mathcal{W}_{D^{3}}\right)  }\times
_{M\otimes\mathcal{W}_{D^{3}\{(1,2)\}}}\underset{2}{\left(  M\otimes
\mathcal{W}_{D^{3}}\right)  } & \rightarrow & \underset{1}{M\otimes
\mathcal{W}_{D^{3}}}\\
\parallel &  & \parallel\\
M\otimes\mathcal{W}_{D^{4}\{(1,4),(2,4)\}} & \rightarrow & \underset
{1}{M\otimes\mathcal{W}_{D^{3}}}\\
& \mathrm{id}_{M}\otimes\mathcal{W}_{\varphi_{3}^{3}} &
\end{array}
\]
and
\[%
\begin{array}
[c]{ccc}%
\underset{1}{\left(  M\otimes\mathcal{W}_{D^{3}}\right)  }\times
_{M\otimes\mathcal{W}_{D^{3}\{(1,2)\}}}\underset{2}{\left(  M\otimes
\mathcal{W}_{D^{3}}\right)  } & \rightarrow & \underset{2}{M\otimes
\mathcal{W}_{D^{3}}}\\
\parallel &  & \parallel\\
M\otimes\mathcal{W}_{D^{4}\{(1,4),(2,4)\}} & \rightarrow & \underset
{2}{M\otimes\mathcal{W}_{D^{3}}}\\
& \mathrm{id}_{M}\otimes\mathcal{W}_{\psi_{3}^{3}} &
\end{array}
\]
being commutative, where unnamed arrows are canonical projections.
\end{corollary}

\begin{notation}
We will write
\[
\zeta^{\underset{3}{\overset{\cdot}{-}}}:\left(  M\otimes\mathcal{W}_{D^{3}%
}\right)  \times_{M\otimes\mathcal{W}_{D^{3}\{(1,2)\}}}\left(  M\otimes
\mathcal{W}_{D^{3}}\right)  \rightarrow M\otimes\mathcal{W}_{D^{2}}%
\]
for the morphism
\begin{align*}
& \mathrm{id}_{M}\otimes\mathcal{W}_{(d_{1},d_{2})\in D^{2}\mapsto
(0,0,d_{1},d_{2})\in D^{4}\{(1,4),(3,4)\}}\\
& :\left(  M\otimes\mathcal{W}_{D^{3}}\right)  \times_{M\otimes\mathcal{W}%
_{D^{3}\{(1,2)\}}}\left(  M\otimes\mathcal{W}_{D^{3}}\right) \\
& =M\otimes\mathcal{W}_{D^{4}\{(1,4),(2,4)\}}\\
& \rightarrow M\otimes\mathcal{W}_{D^{2}}%
\end{align*}

\end{notation}

\begin{notation}
We will write $i_{14}^{1}$, \ $i_{24}^{2}$ and $i_{34}^{3}$\ for the
assumptive mappings
\[
\left(  d_{1},d_{2}\right)  \in D(2)\mapsto\left(  d_{1},0,0,d_{2}\right)  \in
D^{4}\{(2,4),(3,4)\}\text{,}%
\]
\[
\left(  d_{1},d_{2}\right)  \in D(2)\mapsto\left(  0,d_{1},0,d_{2}\right)  \in
D^{4}\{(1,4),(3,4)\}
\]
and
\[
\left(  d_{1},d_{2}\right)  \in D(2)\mapsto\left(  0,0,d_{1},d_{2}\right)  \in
D^{4}\{(1,4),(2,4)\}
\]
respectively.
\end{notation}

\begin{proposition}
\label{t3.11}The diagram
\[%
\begin{array}
[c]{ccccc}
&  & \mathrm{id}_{M}\otimes\mathcal{W}_{\eta_{1}^{1}} &  & \\
& M\otimes\mathcal{W}_{E[1]} & \rightarrow & M\otimes\mathcal{W}%
_{D^{4}\{(2,4),(3,4)\}} & \\
\mathrm{id}_{M}\otimes\mathcal{W}_{\eta_{2}^{1}} & \downarrow &  & \downarrow
& \mathrm{id}_{M}\otimes\mathcal{W}_{i_{14}^{1}}\\
& M\otimes\mathcal{W}_{D^{4}\{(2,4),(3,4)\}} & \rightarrow & M\otimes
\mathcal{W}_{D(2)} & \\
&  & \mathrm{id}_{M}\otimes\mathcal{W}_{i_{14}^{1}} &  &
\end{array}
\]
is a pullback, where the assumptive object $E[1]$ is
\begin{align*}
&  D^{7}\{(2,6),(3,6),(4,6),(5,6),(1,7),(2,7),(3,7),(4,7),(5,7),(6,7),(2,4),\\
&  (2,5),(3,4),(3,5)\}\text{,}%
\end{align*}
the assumptive mapping
\[
\eta_{1}^{1}:D^{4}\{(2,4),(3,4)\}\rightarrow E[1]
\]
is
\begin{align*}
(d_{1},d_{2},d_{3},d_{4}) &  \in D^{4}\{(2,4),(3,4)\}\mapsto\\
(d_{1},d_{2},d_{3},0,0,d_{4},0) &  \in E[1]\text{, }%
\end{align*}
and the assumptive mapping
\[
\eta_{2}^{1}:D^{4}\{(2,4),(3,4)\}\rightarrow E\left[  1\right]
\]
is
\begin{align*}
(d_{1},d_{2},d_{3},d_{4}) &  \in D^{4}\{(2,4),(3,4)\}\mapsto\\
(d_{1},0,0,d_{2},d_{3},d_{4},d_{1}d_{4}) &  \in E[1]\text{.}%
\end{align*}

\end{proposition}

\begin{proof}
This follows from the microlinearity of $M$ and the pullback diagram of Weil
algebras
\[%
\begin{array}
[c]{ccccc}
&  & \mathcal{W}_{\eta_{1}^{1}} &  & \\
& \mathcal{W}_{E[1]} & \rightarrow & \mathcal{W}_{D^{4}\{(2,4),(3,4)\}} & \\
\mathcal{W}_{\eta_{2}^{1}} & \downarrow &  & \downarrow & \mathcal{W}%
_{i_{14}^{1}}\\
& \mathcal{W}_{D^{4}\{(2,4),(3,4)\}} & \rightarrow & \mathcal{W}_{D(2)} & \\
&  & \mathcal{W}_{i_{14}^{1}} &  &
\end{array}
\]

\end{proof}

\begin{notation}
We will write $\iota_{1}^{1}$, $\iota_{2}^{1}$, $\iota_{3}^{1}$ and $\iota
_{4}^{1}$ for the assumptive mappings $\eta_{1}^{1}\circ\varphi_{1}^{3}$,
$\eta_{1}^{1}\circ\psi_{1}^{3}$, $\eta_{2}^{1}\circ\varphi_{1}^{3}\ $and
$\eta_{2}^{1}\circ\psi_{1}^{3}$ respectively. That is to say, we have
\begin{align*}
\iota_{1}^{1} &  :(d_{1},d_{2},d_{3})\in D^{3}\mapsto(d_{1},d_{2}%
,d_{3},0,0,0,0)\in E[1]\\
\iota_{2}^{1} &  :(d_{1},d_{2},d_{3})\in D^{3}\mapsto(d_{1},d_{2}%
,d_{3},0,0,d_{2}d_{3},0)\in E[1]\\
\iota_{3}^{1} &  :(d_{1},d_{2},d_{3})\in D^{3}\mapsto(d_{1},0,0,d_{2}%
,d_{3},0,0)\in E[1]\\
\iota_{4}^{1} &  :(d_{1},d_{2},d_{3})\in D^{3}\mapsto(d_{1},0,0,d_{2}%
,d_{3},d_{2}d_{3},d_{1}d_{2}d_{3})\in E[1]
\end{align*}

\end{notation}

\begin{corollary}
We have
\begin{align*}
& \left(
\begin{array}
[c]{c}%
\left(  M\otimes\mathcal{W}_{D^{3}}\right) \\
\times_{M\otimes\mathcal{W}_{D^{3}\{(2,3)\}}}\\
\left(  M\otimes\mathcal{W}_{D^{3}}\right)
\end{array}
\right)  \times_{M\otimes\mathcal{W}_{D(2)}}\left(
\begin{array}
[c]{c}%
\left(  M\otimes\mathcal{W}_{D^{3}}\right) \\
\times_{M\otimes\mathcal{W}_{D^{3}\{(2,3)\}}}\\
\left(  M\otimes\mathcal{W}_{D^{3}}\right)
\end{array}
\right) \\
& =M\otimes\mathcal{W}_{E[1]}%
\end{align*}
with the diagrams
\[%
\begin{array}
[c]{ccc}%
\left(
\begin{array}
[c]{c}%
\left(  \underset{1}{M\otimes\mathcal{W}_{D^{3}}}\right) \\
\times_{M\otimes\mathcal{W}_{D^{3}\{(2,3)\}}}\\
\left(  \underset{2}{M\otimes\mathcal{W}_{D^{3}}}\right)
\end{array}
\right)  \times_{M\otimes\mathcal{W}_{D(2)}}\left(
\begin{array}
[c]{c}%
\left(  \underset{3}{M\otimes\mathcal{W}_{D^{3}}}\right) \\
\times_{M\otimes\mathcal{W}_{D^{3}\{(2,3)\}}}\\
\left(  \underset{4}{M\otimes\mathcal{W}_{D^{3}}}\right)
\end{array}
\right)  & \rightarrow &
\begin{array}
[c]{c}%
\left(  \underset{1}{M\otimes\mathcal{W}_{D^{3}}}\right) \\
\times_{M\otimes\mathcal{W}_{D^{3}\{(2,3)\}}}\\
\left(  \underset{2}{M\otimes\mathcal{W}_{D^{3}}}\right)
\end{array}
\\
\parallel &  & \parallel\\
M\otimes\mathcal{W}_{E[1]} & \rightarrow &
\begin{array}
[c]{c}%
M\otimes\\
\mathcal{W}_{D^{4}\{(2,4),(3,4)\}}%
\end{array}
\\
& \mathrm{id}_{M}\otimes\mathcal{W}_{\eta_{1}^{1}} &
\end{array}
\]
and
\[%
\begin{array}
[c]{ccc}%
\left(
\begin{array}
[c]{c}%
\left(  \underset{1}{M\otimes\mathcal{W}_{D^{3}}}\right) \\
\times_{M\otimes\mathcal{W}_{D^{3}\{(2,3)\}}}\\
\left(  \underset{2}{M\otimes\mathcal{W}_{D^{3}}}\right)
\end{array}
\right)  \times_{M\otimes\mathcal{W}_{D(2)}}\left(
\begin{array}
[c]{c}%
\left(  \underset{3}{M\otimes\mathcal{W}_{D^{3}}}\right) \\
\times_{M\otimes\mathcal{W}_{D^{3}\{(2,3)\}}}\\
\left(  \underset{4}{M\otimes\mathcal{W}_{D^{3}}}\right)
\end{array}
\right)  & \rightarrow &
\begin{array}
[c]{c}%
\left(  \underset{3}{M\otimes\mathcal{W}_{D^{3}}}\right) \\
\times_{M\otimes\mathcal{W}_{D^{3}\{(2,3)\}}}\\
\left(  \underset{4}{M\otimes\mathcal{W}_{D^{3}}}\right)
\end{array}
\\
\parallel &  & \parallel\\
M\otimes\mathcal{W}_{E[1]} & \rightarrow &
\begin{array}
[c]{c}%
M\otimes\\
\mathcal{W}_{D^{4}\{(2,4),(3,4)\}}%
\end{array}
\\
& \mathrm{id}_{M}\otimes\mathcal{W}_{\eta_{2}^{1}} &
\end{array}
\]
being commutative, where unnamed arrows are canonical projections.
\end{corollary}

\begin{proposition}
\label{t3.12}The diagram
\[%
\begin{array}
[c]{ccccc}
&  & \mathrm{id}_{M}\otimes\mathcal{W}_{\eta_{1}^{2}} &  & \\
& M\otimes\mathcal{W}_{E[2]} & \rightarrow & M\otimes\mathcal{W}%
_{D^{4}\{(1,4),(3,4)\}} & \\
\mathrm{id}_{M}\otimes\mathcal{W}_{\eta_{2}^{2}} & \downarrow &  & \downarrow
& \mathrm{id}_{M}\otimes\mathcal{W}_{i_{24}^{2}}\\
& M\otimes\mathcal{W}_{D^{4}\{(1,4),(3,4)\}} & \rightarrow & M\otimes
\mathcal{W}_{D(2)} & \\
&  & \mathrm{id}_{M}\otimes\mathcal{W}_{i_{24}^{2}} &  &
\end{array}
\]
is a pullback, where the assumptive object $E[2]$ is
\begin{align*}
&  D^{7}\{(1,6),(3,6),(4,6),(5,6),(1,7),(2,7),(3,7),(4,7),(5,7),(6,7),(1,4),\\
&  (1,5),(3,4),(3,5)\}\text{,}%
\end{align*}
the assumptive mapping $\eta_{1}^{2}:D^{4}\{(1,4),(3,4)\}\}\rightarrow E[2]$
is
\begin{align*}
(d_{1},d_{2},d_{3},d_{4}) &  \in D^{4}\{(1,4),(3,4)\}\mapsto\\
(d_{1},d_{2},d_{3},0,0,d_{4},0) &  \in E[2]\text{,}%
\end{align*}
and the assumptive mapping $\eta_{2}^{2}:D^{4}\{(1,4),(3,4)\}\rightarrow E[2]$
is
\begin{align*}
(d_{1},d_{2},d_{3},d_{4}) &  \in D^{4}\{(1,4),(3,4)\}\mapsto\\
(0,d_{2},0,d_{1},d_{3},d_{4},d_{2}d_{4}) &  \in E[2]\text{.}%
\end{align*}

\end{proposition}

\begin{proof}
This follows from the microlinearity of $M$ and the pullback diagram of Weil
algebras
\[%
\begin{array}
[c]{ccccc}
&  & \mathcal{W}_{\eta_{1}^{2}} &  & \\
& \mathcal{W}_{E[2]} & \rightarrow & \mathcal{W}_{D^{4}\{(1,4),(3,4)\}} & \\
\mathcal{W}_{\eta_{2}^{2}} & \downarrow &  & \downarrow & \mathcal{W}%
_{i_{24}^{2}}\\
& \mathcal{W}_{D^{4}\{(1,4),(3,4)\}} & \rightarrow & \mathcal{W}_{D(2)} & \\
&  & \mathcal{W}_{i_{24}^{2}} &  &
\end{array}
\]

\end{proof}

\begin{notation}
We will write $\iota_{1}^{2}$, $\iota_{2}^{2}$, $\iota_{3}^{2}$ and $\iota
_{4}^{2}$ for the assumptive mappings $\eta_{1}^{2}\circ\varphi_{2}^{3}$,
$\eta_{1}^{2}\circ\psi_{2}^{3}$, $\eta_{2}^{2}\circ\varphi_{2}^{3}\ $and
$\eta_{2}^{2}\circ\psi_{2}^{3}$ respectively. That is to say, we have
\begin{align*}
\iota_{1}^{2} &  :(d_{1},d_{2},d_{3})\in D^{3}\mapsto(d_{1},d_{2}%
,d_{3},0,0,0,0)\in E[2]\\
\iota_{2}^{2} &  :(d_{1},d_{2},d_{3})\in D^{3}\mapsto(d_{1},d_{2}%
,d_{3},0,0,d_{2}d_{3},0)\in E[2]\\
\iota_{3}^{2} &  :(d_{1},d_{2},d_{3})\in D^{3}\mapsto(0,d_{2},0,d_{3}%
,d_{1},0,0)\in E[2]\\
\iota_{4}^{2} &  :(d_{1},d_{2},d_{3})\in D^{3}\mapsto(0,d_{2},0,d_{3}%
,d_{1},d_{1}d_{3},d_{1}d_{2}d_{3})\in E[2]
\end{align*}

\end{notation}

\begin{corollary}
We have
\begin{align*}
& \left(
\begin{array}
[c]{c}%
\left(  M\otimes\mathcal{W}_{D^{3}}\right) \\
\times_{M\otimes\mathcal{W}_{D^{3}\{(1,3)\}}}\\
\left(  M\otimes\mathcal{W}_{D^{3}}\right)
\end{array}
\right)  \times_{M\otimes\mathcal{W}_{D(2)}}\left(
\begin{array}
[c]{c}%
\left(  M\otimes\mathcal{W}_{D^{3}}\right) \\
\times_{M\otimes\mathcal{W}_{D^{3}\{(1,3)\}}}\\
\left(  M\otimes\mathcal{W}_{D^{3}}\right)
\end{array}
\right) \\
& =M\otimes\mathcal{W}_{E[2]}%
\end{align*}
with the diagrams
\[%
\begin{array}
[c]{ccc}%
\left(
\begin{array}
[c]{c}%
\left(  \underset{1}{M\otimes\mathcal{W}_{D^{3}}}\right) \\
\times_{M\otimes\mathcal{W}_{D^{3}\{(1,3)\}}}\\
\left(  \underset{2}{M\otimes\mathcal{W}_{D^{3}}}\right)
\end{array}
\right)  \times_{M\otimes\mathcal{W}_{D(2)}}\left(
\begin{array}
[c]{c}%
\left(  \underset{3}{M\otimes\mathcal{W}_{D^{3}}}\right) \\
\times_{M\otimes\mathcal{W}_{D^{3}\{(1,3)\}}}\\
\left(  \underset{4}{M\otimes\mathcal{W}_{D^{3}}}\right)
\end{array}
\right)  & \rightarrow &
\begin{array}
[c]{c}%
\left(  \underset{1}{M\otimes\mathcal{W}_{D^{3}}}\right) \\
\times_{M\otimes\mathcal{W}_{D^{3}\{(1,3)\}}}\\
\left(  \underset{2}{M\otimes\mathcal{W}_{D^{3}}}\right)
\end{array}
\\
\parallel &  & \parallel\\
M\otimes\mathcal{W}_{E[2]} & \rightarrow &
\begin{array}
[c]{c}%
M\otimes\\
\mathcal{W}_{D^{4}\{(1,4),(3,4)\}}%
\end{array}
\\
& \mathrm{id}_{M}\otimes\mathcal{W}_{\eta_{1}^{2}} &
\end{array}
\]
and
\[%
\begin{array}
[c]{ccc}%
\left(
\begin{array}
[c]{c}%
\left(  \underset{1}{M\otimes\mathcal{W}_{D^{3}}}\right) \\
\times_{M\otimes\mathcal{W}_{D^{3}\{(1,3)\}}}\\
\left(  \underset{2}{M\otimes\mathcal{W}_{D^{3}}}\right)
\end{array}
\right)  \times_{M\otimes\mathcal{W}_{D(2)}}\left(
\begin{array}
[c]{c}%
\left(  \underset{3}{M\otimes\mathcal{W}_{D^{3}}}\right) \\
\times_{M\otimes\mathcal{W}_{D^{3}\{(1,3)\}}}\\
\left(  \underset{4}{M\otimes\mathcal{W}_{D^{3}}}\right)
\end{array}
\right)  & \rightarrow &
\begin{array}
[c]{c}%
\left(  \underset{3}{M\otimes\mathcal{W}_{D^{3}}}\right) \\
\times_{M\otimes\mathcal{W}_{D^{3}\{(1,3)\}}}\\
\left(  \underset{4}{M\otimes\mathcal{W}_{D^{3}}}\right)
\end{array}
\\
\parallel &  & \parallel\\
M\otimes\mathcal{W}_{E[2]} & \rightarrow &
\begin{array}
[c]{c}%
M\otimes\\
\mathcal{W}_{D^{4}\{(1,4),(3,4)\}}%
\end{array}
\\
& \mathrm{id}_{M}\otimes\mathcal{W}_{\eta_{2}^{1}} &
\end{array}
\]
being commutative, where unnamed arrows are the canonical projections.
\end{corollary}

\begin{proposition}
\label{t3.13}The diagram
\[%
\begin{array}
[c]{ccccc}
&  & \mathrm{id}_{M}\otimes\mathcal{W}_{\eta_{1}^{3}} &  & \\
& M\otimes\mathcal{W}_{E[3]} & \rightarrow & M\otimes\mathcal{W}%
_{D^{4}\{(1,4),(2,4)\}} & \\
\mathrm{id}_{M}\otimes\mathcal{W}_{\eta_{2}^{3}} & \downarrow &  & \downarrow
& \mathrm{id}_{M}\otimes\mathcal{W}_{i_{34}^{3}}\\
& M\otimes\mathcal{W}_{D^{4}\{(1,4),(2,4)\}} & \rightarrow & M\otimes
\mathcal{W}_{D(2)} & \\
&  & \mathrm{id}_{M}\otimes\mathcal{W}_{i_{34}^{3}} &  &
\end{array}
\]
is a pullback, where the assumptive object $E[3]$ is
\begin{align*}
&  D^{7}\{(1,6),(2,6),(4,6),(5,6),(1,7),(2,7),(3,7),(4,7),(5,7),(6,7),(1,4),\\
&  (1,5),(2,4),(2,5)\}\}\text{,}%
\end{align*}
the assumptive mapping $\eta_{1}^{3}:D^{4}\{(1,4),(2,4)\}\}\rightarrow E[3]$
is
\begin{align*}
(d_{1},d_{2},d_{3},d_{4}) &  \in D^{4}\{(1,4),(2,4)\}\mapsto\\
(d_{1},d_{2},d_{3},0,0,d_{4},0) &  \in E[3]\text{,}%
\end{align*}
and the assumptive mapping $\eta_{2}^{3}:D^{4}\{(1,4),(3,4)\}\rightarrow E[3]$
is
\begin{align*}
(d_{1},d_{2},d_{3},d_{4}) &  \in D^{4}\{(1,4),(2,4)\}\mapsto\\
(0,0,d_{3},d_{1},d_{2},d_{4},d_{3}d_{4}) &  \in E[3]\text{.}%
\end{align*}

\end{proposition}

\begin{proof}
This follows from the microlinearity of $M$ and the pullback diagram of Weil
algebras
\[%
\begin{array}
[c]{ccccc}
&  & \mathcal{W}_{\eta_{1}^{3}} &  & \\
& \mathcal{W}_{E[3]} & \rightarrow & \mathcal{W}_{D^{4}\{(1,4),(2,4)\}} & \\
\mathcal{W}_{\eta_{2}^{3}} & \downarrow &  & \downarrow & \mathcal{W}%
_{i_{34}^{3}}\\
& \mathcal{W}_{D^{4}\{(1,4),(2,4)\}} & \rightarrow & \mathcal{W}_{D(2)} & \\
&  & \mathcal{W}_{i_{34}^{3}} &  &
\end{array}
\]

\end{proof}

\begin{notation}
We will write $\iota_{1}^{3}$, $\iota_{2}^{3}$, $\iota_{3}^{3}$ and $\iota
_{4}^{3}$ for the assumptive mappings $\eta_{1}^{3}\circ\varphi_{3}^{3}$,
$\eta_{1}^{3}\circ\psi_{3}^{3}$, $\eta_{2}^{3}\circ\varphi_{3}^{3}\ $and
$\eta_{2}^{3}\circ\psi_{3}^{3}$ respectively. That is to say, we have
\begin{align*}
\iota_{1}^{3} &  :(d_{1},d_{2},d_{3})\in D^{3}\mapsto(d_{1},d_{2}%
,d_{3},0,0,0,0)\in E[3]\\
\iota_{2}^{3} &  :(d_{1},d_{2},d_{3})\in D^{3}\mapsto(d_{1},d_{2}%
,d_{3},0,0,d_{1}d_{2},0)\in E[3]\\
\iota_{3}^{3} &  :(d_{1},d_{2},d_{3})\in D^{3}\mapsto(0,0,d_{3},d_{1}%
,d_{2},0,0)\in E[3]\\
\iota_{4}^{3} &  :(d_{1},d_{2},d_{3})\in D^{3}\mapsto(0,0,d_{3},d_{1}%
,d_{2},d_{1}d_{2},d_{1}d_{2}d_{3})\in E[3]
\end{align*}

\end{notation}

\begin{corollary}
We have
\begin{align*}
& \left(
\begin{array}
[c]{c}%
\left(  M\otimes\mathcal{W}_{D^{3}}\right) \\
\times_{M\otimes\mathcal{W}_{D^{3}\{(1,2)\}}}\\
\left(  M\otimes\mathcal{W}_{D^{3}}\right)
\end{array}
\right)  \times_{M\otimes\mathcal{W}_{D(2)}}\left(
\begin{array}
[c]{c}%
\left(  M\otimes\mathcal{W}_{D^{3}}\right) \\
\times_{M\otimes\mathcal{W}_{D^{3}\{(1,2)\}}}\\
\left(  M\otimes\mathcal{W}_{D^{3}}\right)
\end{array}
\right) \\
& =M\otimes\mathcal{W}_{E[3]}%
\end{align*}
with the diagrams
\[%
\begin{array}
[c]{ccc}%
\left(
\begin{array}
[c]{c}%
\left(  \underset{1}{M\otimes\mathcal{W}_{D^{3}}}\right) \\
\times_{M\otimes\mathcal{W}_{D^{3}\{(1,2)\}}}\\
\left(  \underset{2}{M\otimes\mathcal{W}_{D^{3}}}\right)
\end{array}
\right)  \times_{M\otimes\mathcal{W}_{D(2)}}\left(
\begin{array}
[c]{c}%
\left(  \underset{3}{M\otimes\mathcal{W}_{D^{3}}}\right) \\
\times_{M\otimes\mathcal{W}_{D^{3}\{(1,2)\}}}\\
\left(  \underset{4}{M\otimes\mathcal{W}_{D^{3}}}\right)
\end{array}
\right)  & \rightarrow &
\begin{array}
[c]{c}%
\left(  \underset{1}{M\otimes\mathcal{W}_{D^{3}}}\right) \\
\times_{M\otimes\mathcal{W}_{D^{3}\{(1,2)\}}}\\
\left(  \underset{2}{M\otimes\mathcal{W}_{D^{3}}}\right)
\end{array}
\\
\parallel &  & \parallel\\
M\otimes\mathcal{W}_{E[3]} & \rightarrow &
\begin{array}
[c]{c}%
M\otimes\\
\mathcal{W}_{D^{4}\{(1,4),(2,4)\}}%
\end{array}
\\
& \mathrm{id}_{M}\otimes\mathcal{W}_{\eta_{1}^{3}} &
\end{array}
\]
and
\[%
\begin{array}
[c]{ccc}%
\left(
\begin{array}
[c]{c}%
\left(  \underset{1}{M\otimes\mathcal{W}_{D^{3}}}\right) \\
\times_{M\otimes\mathcal{W}_{D^{3}\{(1,2)\}}}\\
\left(  \underset{2}{M\otimes\mathcal{W}_{D^{3}}}\right)
\end{array}
\right)  \times_{M\otimes\mathcal{W}_{D(2)}}\left(
\begin{array}
[c]{c}%
\left(  \underset{3}{M\otimes\mathcal{W}_{D^{3}}}\right) \\
\times_{M\otimes\mathcal{W}_{D^{3}\{(1,2)\}}}\\
\left(  \underset{4}{M\otimes\mathcal{W}_{D^{3}}}\right)
\end{array}
\right)  & \rightarrow &
\begin{array}
[c]{c}%
\left(  \underset{3}{M\otimes\mathcal{W}_{D^{3}}}\right) \\
\times_{M\otimes\mathcal{W}_{D^{3}\{(1,2)\}}}\\
\left(  \underset{4}{M\otimes\mathcal{W}_{D^{3}}}\right)
\end{array}
\\
\parallel &  & \parallel\\
M\otimes\mathcal{W}_{E[3]} & \rightarrow &
\begin{array}
[c]{c}%
M\otimes\\
\mathcal{W}_{D^{4}\{(1,4),(2,4)\}}%
\end{array}
\\
& \mathrm{id}_{M}\otimes\mathcal{W}_{\eta_{2}^{3}} &
\end{array}
\]
being commutative, where unnamed arrows are the canonical projections.
\end{corollary}

Now we come to the crucial step in the proof of the general Jacobi identity.

\begin{notation}
A limit of the diagram
\[%
\begin{array}
[c]{ccccc}
&  &
\begin{array}
[c]{c}%
\left(
\begin{array}
[c]{c}%
\underset{321}{\left(  M\otimes\mathcal{W}_{D^{3}}\right)  }\\
\times_{M\otimes\mathcal{W}_{D^{3}\{(2,3)\}}}\\
\underset{231}{\left(  M\otimes\mathcal{W}_{D^{3}}\right)  }%
\end{array}
\right) \\
\times_{M\otimes\mathcal{W}_{D(2)}}\\
\left(
\begin{array}
[c]{c}%
\underset{132}{\left(  M\otimes\mathcal{W}_{D^{3}}\right)  }\\
\times_{M\otimes\mathcal{W}_{D^{3}\{(2,3)\}}}\\
\underset{123}{\left(  M\otimes\mathcal{W}_{D^{3}}\right)  }%
\end{array}
\right)
\end{array}
&  & \\
& \swarrow &  & \searrow & \\%
\begin{array}
[c]{c}%
\underset{231}{\left(  M\otimes\mathcal{W}_{D^{3}}\right)  }\\
\times_{M}\\
\underset{132}{\left(  M\otimes\mathcal{W}_{D^{3}}\right)  }%
\end{array}
&  &  &  &
\begin{array}
[c]{c}%
\underset{123}{\left(  M\otimes\mathcal{W}_{D^{3}}\right)  }\\
\times_{M}\\
\underset{321}{\left(  M\otimes\mathcal{W}_{D^{3}}\right)  }%
\end{array}
\\
\uparrow &  &  &  & \uparrow\\%
\begin{array}
[c]{c}%
\left(
\begin{array}
[c]{c}%
\underset{132}{\left(  M\otimes\mathcal{W}_{D^{3}}\right)  }\\
\times_{M\otimes\mathcal{W}_{D^{3}\{(1,3)\}}}\\
\underset{312}{\left(  M\otimes\mathcal{W}_{D^{3}}\right)  }%
\end{array}
\right) \\
\times_{M\otimes\mathcal{W}_{D(2)}}\\
\left(
\begin{array}
[c]{c}%
\underset{213}{\left(  M\otimes\mathcal{W}_{D^{3}}\right)  }\\
\times_{M\otimes\mathcal{W}_{D^{3}\{(1,3)\}}}\\
\underset{231}{\left(  M\otimes\mathcal{W}_{D^{3}}\right)  }%
\end{array}
\right)
\end{array}
&  &  &  &
\begin{array}
[c]{c}%
\left(
\begin{array}
[c]{c}%
\underset{213}{\left(  M\otimes\mathcal{W}_{D^{3}}\right)  }\\
\times_{M\otimes\mathcal{W}_{D^{3}\{(1,2)\}}}\\
\underset{123}{\left(  M\otimes\mathcal{W}_{D^{3}}\right)  }%
\end{array}
\right) \\
\times_{M\otimes\mathcal{W}_{D(2)}}\\
\left(
\begin{array}
[c]{c}%
\underset{321}{\left(  M\otimes\mathcal{W}_{D^{3}}\right)  }\\
\times_{M\otimes\mathcal{W}_{D^{3}\{(1,2)\}}}\\
\underset{312}{\left(  M\otimes\mathcal{W}_{D^{3}}\right)  }%
\end{array}
\right)
\end{array}
\\
& \searrow &  & \swarrow & \\
&  &
\begin{array}
[c]{c}%
\underset{312}{\left(  M\otimes\mathcal{W}_{D^{3}}\right)  }\\
\times_{M}\\
\underset{213}{\left(  M\otimes\mathcal{W}_{D^{3}}\right)  }%
\end{array}
&  &
\end{array}
\]
with every arrow being the canonical projection is denoted by
\[
\left[
\begin{array}
[c]{ccc}
&
\begin{array}
[c]{c}%
\left(
\begin{array}
[c]{c}%
\underset{321}{\left(  M\otimes\mathcal{W}_{D^{3}}\right)  }\\
\times_{M\otimes\mathcal{W}_{D^{3}\{(2,3)\}}}\\
\underset{231}{\left(  M\otimes\mathcal{W}_{D^{3}}\right)  }%
\end{array}
\right) \\
\times_{M\otimes\mathcal{W}_{D(2)}}\\
\left(
\begin{array}
[c]{c}%
\underset{132}{\left(  M\otimes\mathcal{W}_{D^{3}}\right)  }\\
\times_{M\otimes\mathcal{W}_{D^{3}\{(2,3)\}}}\\
\underset{123}{\left(  M\otimes\mathcal{W}_{D^{3}}\right)  }%
\end{array}
\right)
\end{array}
& \\%
\begin{array}
[c]{c}%
\times_{M\otimes\mathcal{W}_{D^{3}\oplus D^{3}}}\\
\,\\
\,
\end{array}
&  &
\begin{array}
[c]{c}%
\times_{M\otimes\mathcal{W}_{D^{3}\oplus D^{3}}}\\
\,\\
\,
\end{array}
\\%
\begin{array}
[c]{c}%
\left(
\begin{array}
[c]{c}%
\underset{132}{\left(  M\otimes\mathcal{W}_{D^{3}}\right)  }\\
\times_{M\otimes\mathcal{W}_{D^{3}\{(1,3)\}}}\\
\underset{312}{\left(  M\otimes\mathcal{W}_{D^{3}}\right)  }%
\end{array}
\right) \\
\times_{M\otimes\mathcal{W}_{D(2)}}\\
\left(
\begin{array}
[c]{c}%
\underset{213}{\left(  M\otimes\mathcal{W}_{D^{3}}\right)  }\\
\times_{M\otimes\mathcal{W}_{D^{3}\{(1,3)\}}}\\
\underset{231}{\left(  M\otimes\mathcal{W}_{D^{3}}\right)  }%
\end{array}
\right)
\end{array}
& \times_{M\otimes\mathcal{W}_{D^{3}\oplus D^{3}}} &
\begin{array}
[c]{c}%
\left(
\begin{array}
[c]{c}%
\underset{213}{\left(  M\otimes\mathcal{W}_{D^{3}}\right)  }\\
\times_{M\otimes\mathcal{W}_{D^{3}\{(1,2)\}}}\\
\underset{123}{\left(  M\otimes\mathcal{W}_{D^{3}}\right)  }%
\end{array}
\right) \\
\times_{M\otimes\mathcal{W}_{D(2)}}\\
\left(
\begin{array}
[c]{c}%
\underset{321}{\left(  M\otimes\mathcal{W}_{D^{3}}\right)  }\\
\times_{M\otimes\mathcal{W}_{D^{3}\{(1,2)\}}}\\
\underset{312}{\left(  M\otimes\mathcal{W}_{D^{3}}\right)  }%
\end{array}
\right)
\end{array}
\end{array}
\right]
\]

\end{notation}

We can compute the above limit.

\begin{theorem}
\label{t3.14}The diagram
\[%
\begin{array}
[c]{ccccccc}
& \mathrm{id}_{M}\otimes\mathcal{W}_{h_{12}^{1}} &  & M\otimes\mathcal{W}%
_{E[1]} &  & \mathrm{id}_{M}\otimes\mathcal{W}_{h_{31}^{1}} & \\
&  & \swarrow & \uparrow & \searrow &  & \\
& M\otimes\mathcal{W}_{D^{3}\oplus D^{3}} &  & M\otimes\mathcal{W}_{G} &  &
M\otimes\mathcal{W}_{D^{3}\oplus D^{3}} & \\
\mathrm{id}_{M}\otimes\mathcal{W}_{h_{12}^{2}} & \uparrow & \swarrow &  &
\searrow & \uparrow & \mathrm{id}_{M}\otimes\mathcal{W}_{h_{31}^{3}}\\
& M\otimes\mathcal{W}_{E[2]} &  &  &  & M\otimes\mathcal{W}_{E[3]} & \\
&  & \searrow &  & \swarrow &  & \\
& \mathrm{id}_{M}\otimes\mathcal{W}_{h_{23}^{2}} &  & M\otimes\mathcal{W}%
_{D^{3}\oplus D^{3}} &  & \mathrm{id}_{M}\otimes\mathcal{W}_{h_{23}^{3}} &
\end{array}
\]
is a limit diagram with the three unnamed arrows being
\begin{align*}
\mathrm{id}_{M}\otimes\mathcal{W}_{k_{1}} &  :M\otimes\mathcal{W}%
_{G}\rightarrow M\otimes\mathcal{W}_{E[1]}\\
\mathrm{id}_{M}\otimes\mathcal{W}_{k_{2}} &  :M\otimes\mathcal{W}%
_{G}\rightarrow M\otimes\mathcal{W}_{E[2]}\\
\mathrm{id}_{M}\otimes\mathcal{W}_{k_{3}} &  :M\otimes\mathcal{W}%
_{G}\rightarrow M\otimes\mathcal{W}_{E[3]}%
\end{align*}
where the assumptive object $G$ is
\begin{align*}
&  D^{8}%
\{(2,4),(3,4),(1,5),(3,5),(1,6),(2,6),(4,5),(4,6),(5,6),(1,7),(2,7),(3,7),\\
&  (4,7),(5,7),(6,7),(1,8),(2,8),(3,8),(4,8),(5,8),(6,8),(7,8)\}\text{,}%
\end{align*}
the assumptive mapping $k_{1}:E[1]\rightarrow G$ is
\begin{align*}
(d_{1},d_{2},d_{3},d_{4},d_{5},d_{6},d_{7}) &  \in E[1]\mapsto\\
(d_{1},d_{2}+d_{4},d_{3}+d_{5},d_{6}-d_{2}d_{3}-d_{4}d_{5},-d_{1}d_{5}%
,d_{1}d_{4},d_{7}+d_{1}d_{2}d_{3},d_{1}d_{2}d_{3}) &  \in G\text{,}%
\end{align*}
the assumptive mapping $k_{2}:E[2]\rightarrow G$ is
\begin{align*}
(d_{1},d_{2},d_{3},d_{4},d_{5},d_{6},d_{7}) &  \in E[2]\mapsto\\
(d_{1}+d_{5},d_{2},d_{3}+d_{4},-d_{2}d_{3},d_{6}-d_{1}d_{3}-d_{4}d_{5}%
,d_{1}d_{2},d_{2}d_{4}d_{5},d_{7}) &  \in G\text{,}%
\end{align*}
the assumptive mapping $k_{3}:E[3]\rightarrow G$ is
\begin{align*}
(d_{1},d_{2},d_{3},d_{4},d_{5},d_{6},d_{7}) &  \in E[3]\mapsto\\
(d_{1}+d_{4},d_{2}+d_{5},d_{3},-d_{4}d_{5},-d_{1}d_{3},d_{6},-d_{7}%
,-d_{7}+d_{3}d_{4}d_{5}) &  \in G\text{,}%
\end{align*}
the assumptive mapping $h_{12}^{1}$ is
\[
\iota_{2}^{1}\oplus\iota_{3}^{1}\text{,}%
\]
the assumptive mapping $h_{12}^{2}$ is
\[
\iota_{4}^{2}\oplus\iota_{1}^{2}\text{,}%
\]
the assumptive mapping $h_{23}^{2}$ is
\[
\iota_{2}^{2}\oplus\iota_{3}^{2}\text{,}%
\]
the assumptive mapping $h_{23}^{3}$ is
\[
\iota_{4}^{3}\oplus\iota_{1}^{3}\text{,}%
\]
the assumptive mapping $h_{31}^{3}$ is
\[
\iota_{2}^{3}\oplus\iota_{3}^{3}\text{,}%
\]
and the assumptive mapping $h_{31}^{1}$ is
\[
\iota_{4}^{1}\oplus\iota_{1}^{2}\text{.}%
\]

\end{theorem}

\begin{corollary}
We have
\begin{align*}
& \left[
\begin{array}
[c]{ccc}
&
\begin{array}
[c]{c}%
\left(
\begin{array}
[c]{c}%
\underset{321}{\left(  M\otimes\mathcal{W}_{D^{3}}\right)  }\\
\times_{M\otimes\mathcal{W}_{D^{3}\{(2,3)\}}}\\
\underset{231}{\left(  M\otimes\mathcal{W}_{D^{3}}\right)  }%
\end{array}
\right) \\
\times_{M\otimes\mathcal{W}_{D(2)}}\\
\left(
\begin{array}
[c]{c}%
\underset{132}{\left(  M\otimes\mathcal{W}_{D^{3}}\right)  }\\
\times_{M\otimes\mathcal{W}_{D^{3}\{(2,3)\}}}\\
\underset{123}{\left(  M\otimes\mathcal{W}_{D^{3}}\right)  }%
\end{array}
\right)
\end{array}
& \\%
\begin{array}
[c]{c}%
\times_{M\otimes\mathcal{W}_{D^{3}\oplus D^{3}}}\\
\,\\
\,
\end{array}
&  &
\begin{array}
[c]{c}%
\times_{M\otimes\mathcal{W}_{D^{3}\oplus D^{3}}}\\
\,\\
\,
\end{array}
\\%
\begin{array}
[c]{c}%
\left(
\begin{array}
[c]{c}%
\underset{132}{\left(  M\otimes\mathcal{W}_{D^{3}}\right)  }\\
\times_{M\otimes\mathcal{W}_{D^{3}\{(1,3)\}}}\\
\underset{312}{\left(  M\otimes\mathcal{W}_{D^{3}}\right)  }%
\end{array}
\right) \\
\times_{M\otimes\mathcal{W}_{D(2)}}\\
\left(
\begin{array}
[c]{c}%
\underset{213}{\left(  M\otimes\mathcal{W}_{D^{3}}\right)  }\\
\times_{M\otimes\mathcal{W}_{D^{3}\{(1,3)\}}}\\
\underset{231}{\left(  M\otimes\mathcal{W}_{D^{3}}\right)  }%
\end{array}
\right)
\end{array}
& \times_{M\otimes\mathcal{W}_{D^{3}\oplus D^{3}}} &
\begin{array}
[c]{c}%
\left(
\begin{array}
[c]{c}%
\underset{213}{\left(  M\otimes\mathcal{W}_{D^{3}}\right)  }\\
\times_{M\otimes\mathcal{W}_{D^{3}\{(1,2)\}}}\\
\underset{123}{\left(  M\otimes\mathcal{W}_{D^{3}}\right)  }%
\end{array}
\right) \\
\times_{M\otimes\mathcal{W}_{D(2)}}\\
\left(
\begin{array}
[c]{c}%
\underset{321}{\left(  M\otimes\mathcal{W}_{D^{3}}\right)  }\\
\times_{M\otimes\mathcal{W}_{D^{3}\{(1,2)\}}}\\
\underset{312}{\left(  M\otimes\mathcal{W}_{D^{3}}\right)  }%
\end{array}
\right)
\end{array}
\end{array}
\right] \\
& =M\otimes\mathcal{W}_{G}%
\end{align*}
with the diagrams
\[%
\begin{array}
[c]{ccc}%
\triangle & \rightarrow &
\begin{array}
[c]{c}%
\left(
\begin{array}
[c]{c}%
\underset{132}{\left(  M\otimes\mathcal{W}_{D^{3}}\right)  }\\
\times_{M\otimes\mathcal{W}_{D^{3}\{(1,3)\}}}\\
\underset{312}{\left(  M\otimes\mathcal{W}_{D^{3}}\right)  }%
\end{array}
\right) \\
\times_{M\otimes\mathcal{W}_{D(2)}}\\
\left(
\begin{array}
[c]{c}%
\underset{213}{\left(  M\otimes\mathcal{W}_{D^{3}}\right)  }\\
\times_{M\otimes\mathcal{W}_{D^{3}\{(1,3)\}}}\\
\underset{231}{\left(  M\otimes\mathcal{W}_{D^{3}}\right)  }%
\end{array}
\right)
\end{array}
\\
\parallel &  & \parallel\\
M\otimes\mathcal{W}_{G} & \rightarrow & M\otimes\mathcal{W}_{E[2]}\\
& \mathrm{id}_{M}\otimes\mathcal{W}_{k_{2}} &
\end{array}
\]
and
\[%
\begin{array}
[c]{ccc}%
\triangle & \rightarrow &
\begin{array}
[c]{c}%
\left(
\begin{array}
[c]{c}%
\underset{213}{\left(  M\otimes\mathcal{W}_{D^{3}}\right)  }\\
\times_{M\otimes\mathcal{W}_{D^{3}\{(1,2)\}}}\\
\underset{123}{\left(  M\otimes\mathcal{W}_{D^{3}}\right)  }%
\end{array}
\right) \\
\times_{M\otimes\mathcal{W}_{D(2)}}\\
\left(
\begin{array}
[c]{c}%
\underset{321}{\left(  M\otimes\mathcal{W}_{D^{3}}\right)  }\\
\times_{M\otimes\mathcal{W}_{D^{3}\{(1,2)\}}}\\
\underset{312}{\left(  M\otimes\mathcal{W}_{D^{3}}\right)  }%
\end{array}
\right)
\end{array}
\\
\parallel &  & \parallel\\
M\otimes\mathcal{W}_{G} & \rightarrow & M\otimes\mathcal{W}_{E[3]}\\
& \mathrm{id}_{M}\otimes\mathcal{W}_{k_{3}} &
\end{array}
\]
being commutative, where $\triangle$\ in the above diagram is
\[
\left[
\begin{array}
[c]{ccc}
&
\begin{array}
[c]{c}%
\left(
\begin{array}
[c]{c}%
\underset{321}{\left(  M\otimes\mathcal{W}_{D^{3}}\right)  }\\
\times_{M\otimes\mathcal{W}_{D^{3}\{(2,3)\}}}\\
\underset{231}{\left(  M\otimes\mathcal{W}_{D^{3}}\right)  }%
\end{array}
\right) \\
\times_{M\otimes\mathcal{W}_{D(2)}}\\
\left(
\begin{array}
[c]{c}%
\underset{132}{\left(  M\otimes\mathcal{W}_{D^{3}}\right)  }\\
\times_{M\otimes\mathcal{W}_{D^{3}\{(2,3)\}}}\\
\underset{123}{\left(  M\otimes\mathcal{W}_{D^{3}}\right)  }%
\end{array}
\right)
\end{array}
& \\%
\begin{array}
[c]{c}%
\times_{M\otimes\mathcal{W}_{D^{3}\oplus D^{3}}}\\
\,\\
\,
\end{array}
&  &
\begin{array}
[c]{c}%
\times_{M\otimes\mathcal{W}_{D^{3}\oplus D^{3}}}\\
\,\\
\,
\end{array}
\\%
\begin{array}
[c]{c}%
\left(
\begin{array}
[c]{c}%
\underset{132}{\left(  M\otimes\mathcal{W}_{D^{3}}\right)  }\\
\times_{M\otimes\mathcal{W}_{D^{3}\{(1,3)\}}}\\
\underset{312}{\left(  M\otimes\mathcal{W}_{D^{3}}\right)  }%
\end{array}
\right) \\
\times_{M\otimes\mathcal{W}_{D(2)}}\\
\left(
\begin{array}
[c]{c}%
\underset{213}{\left(  M\otimes\mathcal{W}_{D^{3}}\right)  }\\
\times_{M\otimes\mathcal{W}_{D^{3}\{(1,3)\}}}\\
\underset{231}{\left(  M\otimes\mathcal{W}_{D^{3}}\right)  }%
\end{array}
\right)
\end{array}
& \times_{M\otimes\mathcal{W}_{D^{3}\oplus D^{3}}} &
\begin{array}
[c]{c}%
\left(
\begin{array}
[c]{c}%
\underset{213}{\left(  M\otimes\mathcal{W}_{D^{3}}\right)  }\\
\times_{M\otimes\mathcal{W}_{D^{3}\{(1,2)\}}}\\
\underset{123}{\left(  M\otimes\mathcal{W}_{D^{3}}\right)  }%
\end{array}
\right) \\
\times_{M\otimes\mathcal{W}_{D(2)}}\\
\left(
\begin{array}
[c]{c}%
\underset{321}{\left(  M\otimes\mathcal{W}_{D^{3}}\right)  }\\
\times_{M\otimes\mathcal{W}_{D^{3}\{(1,2)\}}}\\
\underset{312}{\left(  M\otimes\mathcal{W}_{D^{3}}\right)  }%
\end{array}
\right)
\end{array}
\end{array}
\right]
\]
and unnamed arrows are the canonical projections.
\end{corollary}

The proof of the above theorem follows directly from the following lemma.

\begin{lemma}
\label{t3.15}The following diagram is a limit diagram of Weil algebras:
\[%
\begin{array}
[c]{ccccccc}
& \mathcal{W}_{h_{12}^{1}} &  & \mathcal{W}_{E[1]} &  & \mathcal{W}%
_{h_{31}^{1}} & \\
&  & \swarrow & \uparrow & \searrow &  & \\
& \mathcal{W}_{D^{3}\oplus D^{3}} &  & \mathcal{W}_{G} &  & \mathcal{W}%
_{D^{3}\oplus D^{3}} & \\
\mathcal{W}_{h_{12}^{2}} & \uparrow & \swarrow &  & \searrow & \uparrow &
\mathcal{W}_{h_{31}^{3}}\\
& \mathcal{W}_{E[2]} &  &  &  & \mathcal{W}_{E[3]} & \\
&  & \searrow &  & \swarrow &  & \\
& \mathcal{W}_{h_{23}^{2}} &  & \mathcal{W}_{D^{3}\oplus D^{3}} &  &
\mathcal{W}_{h_{23}^{3}} &
\end{array}
\]

\end{lemma}

\begin{proof}
Let $\gamma_{1}\in\mathcal{W}_{E[1]}$, $\gamma_{2}\in\mathcal{W}_{E[2]}$,
$\gamma_{3}\in\mathcal{W}_{E[3]}$ and $\gamma\in\mathcal{W}_{G}$ so that they
are polynomials with coefficients in $k$\ of the following forms:
\begin{align*}
&  \gamma_{1}(X_{1},X_{2},X_{3},X_{4},X_{5},X_{6},X_{7})\\
&  =a^{1}+a_{1}^{1}X_{1}+a_{2}^{1}X_{2}+a_{3}^{1}X_{3}+a_{4}^{1}X_{4}%
+a_{5}^{1}X_{5}+a_{6}^{1}X_{6}+a_{7}^{1}X_{7}+a_{12}^{1}X_{1}X_{2}+a_{13}%
^{1}X_{1}X_{3}+\\
&  a_{14}^{1}X_{1}X_{4}+a_{15}^{1}X_{1}X_{5}+a_{16}^{1}X_{1}X_{6}+a_{23}%
^{1}X_{2}X_{3}+a_{45}^{1}X_{4}X_{5}+a_{123}^{1}X_{1}X_{2}X_{3}+a_{145}%
^{1}X_{1}X_{4}X_{5}%
\end{align*}
\begin{align*}
&  \gamma_{2}(X_{1},X_{2},X_{3},X_{4},X_{5},X_{6},X_{7})\\
&  =a^{2}+a_{1}^{2}X_{1}+a_{2}^{2}X_{2}+a_{3}^{2}X_{3}+a_{4}^{2}X_{4}%
+a_{5}^{2}X_{5}+a_{6}^{2}X_{6}+a_{7}^{2}X_{7}+a_{12}^{2}X_{1}X_{2}+a_{13}%
^{2}X_{1}X_{3}+\\
&  a_{23}^{2}X_{2}X_{3}+a_{24}^{2}X_{2}X_{4}+a_{25}^{2}X_{2}X_{5}+a_{26}%
^{2}X_{2}X_{6}+a_{45}^{2}X_{4}X_{5}+a_{123}^{2}X_{1}X_{2}X_{3}+a_{245}%
^{2}X_{2}X_{4}X_{5}%
\end{align*}
\begin{align*}
&  \gamma_{3}(X_{1},X_{2},X_{3},X_{4},X_{5},X_{6},X_{7})\\
&  =a^{3}+a_{1}^{3}X_{1}+a_{2}^{3}X_{2}+a_{3}^{3}X_{3}+a_{4}^{3}X_{4}%
+a_{5}^{3}X_{5}+a_{6}^{3}X_{6}+a_{7}^{3}X_{7}+a_{12}^{3}X_{1}X_{2}+a_{13}%
^{3}X_{1}X_{3}+\\
&  a_{23}^{3}X_{2}X_{3}+a_{34}^{3}X_{3}X_{4}+a_{35}^{3}X_{3}X_{5}+a_{36}%
^{3}X_{3}X_{6}+a_{45}^{3}X_{4}X_{5}+a_{123}^{3}X_{1}X_{2}X_{3}+a_{345}%
^{3}X_{3}X_{4}X_{5}%
\end{align*}
\begin{align*}
&  \gamma(X_{1},X_{2},X_{3},X_{4},X_{5},X_{6},X_{7},X_{8})\\
&  =b+b_{1}X_{1}+b_{2}X_{2}+b_{3}X_{3}+b_{4}X_{4}+b_{5}X_{5}+b_{6}X_{6}%
+b_{7}X_{7}+b_{8}X_{8}+b_{12}X_{1}X_{2}\\
&  +b_{13}X_{1}X_{3}+b_{14}X_{1}X_{4}+b_{23}X_{2}X_{3}+b_{25}X_{2}X_{5}%
+b_{36}X_{3}X_{6}%
\end{align*}
It is easy to see that
\begin{align*}
&  \mathcal{W}_{h_{12}^{1}}(\gamma_{1})(X_{1},X_{2},X_{3},X_{4},X_{5},X_{6})\\
&  =a^{1}+a_{1}^{1}X_{1}+a_{2}^{1}X_{2}+a_{3}^{1}X_{3}+a_{6}^{1}X_{2}%
X_{3}+a_{12}^{1}X_{1}X_{2}+a_{13}^{1}X_{1}X_{3}+a_{16}^{1}X_{1}X_{2}%
X_{3}+a_{23}^{1}X_{2}X_{3}\\
&  +a_{123}^{1}X_{1}X_{2}X_{3}+a_{1}^{1}X_{4}+a_{4}^{1}X_{5}+a_{5}^{1}%
X_{6}+a_{14}^{1}X_{4}X_{5}+a_{15}^{1}X_{4}X_{6}+a_{45}^{1}X_{5}X_{6}\\
&  +a_{145}^{1}X_{4}X_{5}X_{6}\\
&  =a^{1}+a_{1}^{1}X_{1}+a_{2}^{1}X_{2}+a_{3}^{1}X_{3}+a_{12}^{1}X_{1}%
X_{2}+a_{13}^{1}X_{1}X_{3}+(a_{6}^{1}+a_{23}^{1})X_{2}X_{3}\\
&  +(a_{16}^{1}+a_{123}^{1})X_{1}X_{2}X_{3}+a_{1}^{1}X_{4}+a_{4}^{1}%
X_{5}+a_{5}^{1}X_{6}+a_{14}^{1}X_{4}X_{5}+a_{15}^{1}X_{4}X_{6}+a_{45}^{1}%
X_{5}X_{6}\\
&  +a_{145}^{1}X_{4}X_{5}X_{6}%
\end{align*}
\begin{align*}
&  \mathcal{W}_{h_{12}^{2}}(\gamma_{2})(X_{1},X_{2},X_{3},X_{4},X_{5},X_{6})\\
&  =a^{2}+a_{2}^{2}X_{2}+a_{4}^{2}X_{3}+a_{5}^{2}X_{1}+a_{6}^{2}X_{1}%
X_{3}+a_{7}^{2}X_{1}X_{2}X_{3}+a_{24}^{2}X_{2}X_{3}+a_{25}^{2}X_{1}X_{2}\\
&  +a_{26}^{2}X_{1}X_{2}X_{3}+a_{45}^{2}X_{1}X_{3}+a_{245}^{2}X_{1}X_{2}%
X_{3}+a_{1}^{2}X_{4}+a_{2}^{2}X_{5}+a_{3}^{2}X_{6}+a_{12}^{2}X_{4}X_{5}\\
&  +a_{13}^{2}X_{4}X_{6}+a_{23}^{2}X_{5}X_{6}+a_{123}^{2}X_{4}X_{5}X_{6}\\
&  =a^{2}+a_{5}^{2}X_{1}+a_{2}^{2}X_{2}+a_{4}^{2}X_{3}+a_{25}^{2}X_{1}%
X_{2}+(a_{6}^{2}+a_{45}^{2})X_{1}X_{3}+a_{24}^{2}X_{2}X_{3}+\\
&  (a_{7}^{2}+a_{26}^{2}+a_{245}^{2})X_{1}X_{2}X_{3}+a_{1}^{2}X_{4}+a_{2}%
^{2}X_{5}+a_{3}^{2}X_{6}+a_{12}^{2}X_{4}X_{5}+a_{13}^{2}X_{4}X_{6}+a_{23}%
^{2}X_{5}X_{6}\\
&  +a_{123}^{2}X_{4}X_{5}X_{6}%
\end{align*}
Therefore the condition that $\mathcal{W}_{h_{12}^{1}}(\gamma_{1}%
)=\mathcal{W}_{h_{12}^{2}}(\gamma_{2})$ is equivalent to the following
conditions as a whole:
\begin{align}
a^{1} &  =a^{2}\label{5.1}\\
a_{1}^{1} &  =a_{5}^{2},\;a_{2}^{1}=a_{2}^{2},\;a_{3}^{1}=a_{4}^{2}%
,\;a_{1}^{1}=a_{1}^{2},\;a_{4}^{1}=a_{2}^{2},\;a_{5}^{1}=a_{3}^{2}%
\label{5.2}\\
a_{12}^{1} &  =a_{25}^{2},\;a_{13}^{1}=a_{6}^{2}+a_{45}^{2},\;a_{6}^{1}%
+a_{23}^{1}=a_{24}^{2},\;a_{14}^{1}=a_{12}^{2},\;a_{15}^{1}=a_{13}%
^{2},\;\nonumber\\
a_{45}^{1} &  =a_{23}^{2}\label{5.3}\\
a_{16}^{1}+a_{123}^{1} &  =a_{7}^{2}+a_{26}^{2}+a_{245}^{2},\;a_{145}%
^{1}=a_{123}^{2}\label{5.4}%
\end{align}
By the same token, the condition that $\mathcal{W}_{h_{23}^{2}}(\gamma
_{2})=\mathcal{W}_{h_{23}^{3}}(\gamma_{3})$ is equivalent to the following
conditions as a whole:
\begin{align}
a^{2} &  =a^{3}\label{5.5}\\
a_{2}^{2} &  =a_{5}^{3},\;a_{3}^{2}=a_{3}^{3},\;a_{1}^{2}=a_{4}^{3}%
,\;a_{2}^{2}=a_{2}^{3},\;a_{4}^{2}=a_{3}^{3},\;a_{5}^{2}=a_{1}^{3}%
\label{5.6}\\
a_{23}^{2} &  =a_{35}^{3},\;a_{12}^{2}=a_{6}^{3}+a_{45}^{3},\;a_{6}^{2}%
+a_{13}^{2}=a_{34}^{3},\;a_{24}^{2}=a_{23}^{3},\;a_{25}^{2}=a_{12}%
^{3},\;\nonumber\\
a_{45}^{2} &  =a_{13}^{3}\label{5.7}\\
a_{26}^{2}+a_{123}^{2} &  =a_{7}^{3}+a_{36}^{3}+a_{345}^{3},\;a_{245}%
^{2}=a_{123}^{3}\label{5.8}%
\end{align}
By the same token again, the condition that $\mathcal{W}_{h_{31}^{3}}%
(\gamma_{3})=\mathcal{W}_{h_{31}^{1}}(\gamma_{1})$ is equivalent to the
following conditions as a whole:
\begin{align}
a^{3} &  =a^{1}\label{5.9}\\
a_{3}^{3} &  =a_{5}^{1},\;a_{1}^{3}=a_{1}^{1},\;a_{2}^{3}=a_{4}^{1}%
,\;a_{3}^{3}=a_{3}^{1},\;a_{4}^{3}=a_{1}^{1},\;a_{5}^{3}=a_{2}^{1}%
\label{5.10}\\
a_{13}^{3} &  =a_{15}^{1},\;a_{23}^{3}=a_{6}^{1}+a_{45}^{1},\;a_{6}^{3}%
+a_{12}^{3}=a_{14}^{1},\;a_{34}^{3}=a_{13}^{1},\;a_{35}^{3}=a_{23}%
^{1},\;\nonumber\\
a_{45}^{3} &  =a_{12}^{1}\label{5.11}\\
a_{36}^{3}+a_{123}^{3} &  =a_{7}^{1}+a_{16}^{1}+a_{145}^{1},\;a_{345}%
^{3}=a_{123}^{1}\label{5.12}%
\end{align}
The three conditions (\ref{5.1}), (\ref{5.5}) and (\ref{5.9}) can be combined
into
\begin{equation}
a^{1}=a^{2}=a^{3}\label{5.13}%
\end{equation}
The three conditions (\ref{5.2}), (\ref{5.6}) and (\ref{5.10}) are to be
superseded by the following three conditions as a whole:
\begin{align}
a_{1}^{1} &  =a_{1}^{2}=a_{1}^{3}=a_{5}^{2}=a_{4}^{3}\label{5.14}\\
a_{2}^{1} &  =a_{2}^{2}=a_{2}^{3}=a_{4}^{1}=a_{5}^{3}\label{5.15}\\
a_{3}^{1} &  =a_{3}^{2}=a_{3}^{3}=a_{5}^{1}=a_{4}^{2}\label{5.16}%
\end{align}
The three conditions (\ref{5.3}), (\ref{5.7}) and (\ref{5.11}) are equivalent
to the following six conditions as a whole:
\begin{align}
a_{12}^{1} &  =a_{12}^{2}=a_{12}^{3}\label{5.17}\\
a_{13}^{1} &  =a_{13}^{2}=a_{13}^{3}\label{5.18}\\
a_{23}^{1} &  =a_{23}^{2}=a_{23}^{3}\label{5.19}\\
a_{14}^{1} &  =a_{12}^{1}+a_{6}^{3},\;a_{15}^{1}=a_{13}^{1}-a_{6}^{2}%
,\;a_{45}^{1}=a_{23}^{1}\label{5.20}\\
a_{24}^{2} &  =a_{23}^{2}+a_{6}^{1},\;a_{25}^{2}=a_{12}^{2}-a_{6}^{3}%
,\;a_{45}^{2}=a_{13}^{2}\label{5.21}\\
a_{34}^{3} &  =a_{13}^{3}+a_{6}^{2},\;a_{35}^{3}=a_{23}^{3}-a_{6}^{1}%
,\;a_{45}^{3}=a_{12}^{3}\label{5.22}%
\end{align}
The conditions (\ref{5.4}), (\ref{5.8}) and (\ref{5.12}) imply that
\begin{align}
&  a_{7}^{1}+a_{7}^{2}+a_{7}^{3}\nonumber\\
&  =(a_{36}^{3}+a_{123}^{3}-a_{16}^{1}-a_{145}^{1})+(a_{16}^{1}+a_{123}%
^{1}-a_{26}^{2}-a_{245}^{2})+(a_{26}^{2}+a_{123}^{2}-a_{36}^{3}-a_{345}%
^{3})\nonumber\\
&  =(a_{36}^{3}+a_{123}^{3}-a_{16}^{1}-a_{123}^{2})+(a_{16}^{1}+a_{123}%
^{1}-a_{26}^{2}-a_{123}^{3})+(a_{26}^{2}+a_{123}^{2}-a_{36}^{3}-a_{123}%
^{1})\nonumber\\
&  =0\label{5.23}%
\end{align}
Therefore the three conditions (\ref{5.4}), (\ref{5.8}) and (\ref{5.12}) are
to be replaced by the following five conditions as a whole:
\begin{align}
a_{145}^{1}-a_{123}^{1} &  =a_{7}^{3}+a_{36}^{3}-a_{26}^{2}\label{5.24}\\
a_{245}^{2}-a_{123}^{2} &  =a_{7}^{1}+a_{16}^{1}-a_{36}^{3}\label{5.25}\\
a_{345}^{3}-a_{123}^{3} &  =a_{7}^{2}+a_{26}^{2}-a_{16}^{1}\label{5.26}\\
a_{145}^{1} &  =a_{123}^{2},\;a_{245}^{2}=a_{123}^{3}\label{5.27}\\
a_{7}^{1}+a_{7}^{2}+a_{7}^{3} &  =0\label{5.28}%
\end{align}
Indeed, the condition that $a_{345}^{3}=a_{123}^{1}$ is derivable from the
above five conditions, as is to be demonstrated in the following:
\begin{align*}
&  a_{345}^{3}\\
&  =a_{123}^{3}+a_{7}^{2}+a_{26}^{2}-a_{16}^{1}\text{ \ \ [(\ref{5.26})]}\\
&  =a_{245}^{2}+a_{7}^{2}+a_{26}^{2}-a_{16}^{1}\text{ \ \ [(\ref{5.27})]}\\
&  =a_{123}^{2}+a_{7}^{1}-a_{36}^{3}+a_{7}^{2}+a_{26}^{2}\text{
\ \ [(\ref{5.25})]}\\
&  =a_{145}^{1}+a_{7}^{1}-a_{36}^{3}+a_{7}^{2}+a_{26}^{2}\text{
\ \ [(\ref{5.27})]}\\
&  =a_{123}^{1}+a_{7}^{1}+a_{7}^{2}+a_{7}^{3}\text{ \ \ [(\ref{5.24})]}\\
&  =a_{123}^{1}\text{ \ \ [(\ref{5.28})]}%
\end{align*}
Now it is not difficult to see that $\mathcal{W}_{h_{12}^{1}}(\gamma
_{1})=\mathcal{W}_{h_{12}^{2}}(\gamma_{2})$, $\mathcal{W}_{h_{23}^{2}}%
(\gamma_{2})=\mathcal{W}_{h_{23}^{3}}(\gamma_{3})$ and $\mathcal{W}%
_{h_{31}^{3}}(\gamma_{3})=\mathcal{W}_{h_{31}^{1}}(\gamma_{1})$ exactly when
there exists $\gamma\in\mathcal{W}_{G}$ with $\gamma_{i}=\mathcal{W}_{k_{i}%
}(\gamma)$ ($i=1,2,3$), in which $\gamma$ should uniquely be of the following
form:
\begin{align*}
&  \gamma(X_{1},X_{2},X_{3},X_{4},X_{5},X_{6},X_{7},X_{8})\\
&  =a^{1}+a_{1}^{1}X_{1}+a_{2}^{1}X_{2}+a_{3}^{1}X_{3}+a_{6}^{1}X_{4}%
+a_{6}^{2}X_{5}+a_{6}^{3}X_{6}+a_{7}^{1}X_{7}+a_{7}^{2}X_{8}+a_{12}^{1}%
X_{1}X_{2}\\
&  +a_{13}^{1}X_{1}X_{3}+a_{16}^{1}X_{1}X_{4}+(a_{23}^{2}+a_{6}^{1})X_{2}%
X_{3}+a_{26}^{2}X_{2}X_{5}+a_{36}^{3}X_{3}X_{6}%
\end{align*}
This completes the proof of the theorem.
\end{proof}

\begin{notation}
We will introduce three notations.

\begin{enumerate}
\item We will write
\begin{align*}
& \zeta^{(\ast_{123}\overset{\cdot}{\underset{1}{-}}\ast_{132})\overset{\cdot
}{-}(\ast_{231}\overset{\cdot}{\underset{1}{-}}\ast_{321})}:\\
& :\left[
\begin{array}
[c]{ccc}
&
\begin{array}
[c]{c}%
\left(
\begin{array}
[c]{c}%
\underset{321}{\left(  M\otimes\mathcal{W}_{D^{3}}\right)  }\\
\times_{M\otimes\mathcal{W}_{D^{3}\{(2,3)\}}}\\
\underset{231}{\left(  M\otimes\mathcal{W}_{D^{3}}\right)  }%
\end{array}
\right) \\
\times_{M\otimes\mathcal{W}_{D(2)}}\\
\left(
\begin{array}
[c]{c}%
\underset{132}{\left(  M\otimes\mathcal{W}_{D^{3}}\right)  }\\
\times_{M\otimes\mathcal{W}_{D^{3}\{(2,3)\}}}\\
\underset{123}{\left(  M\otimes\mathcal{W}_{D^{3}}\right)  }%
\end{array}
\right)
\end{array}
& \\%
\begin{array}
[c]{c}%
\times_{M\otimes\mathcal{W}_{D^{3}\oplus D^{3}}}\\
\,\\
\,
\end{array}
&  &
\begin{array}
[c]{c}%
\times_{M\otimes\mathcal{W}_{D^{3}\oplus D^{3}}}\\
\,\\
\,
\end{array}
\\%
\begin{array}
[c]{c}%
\left(
\begin{array}
[c]{c}%
\underset{132}{\left(  M\otimes\mathcal{W}_{D^{3}}\right)  }\\
\times_{M\otimes\mathcal{W}_{D^{3}\{(1,3)\}}}\\
\underset{312}{\left(  M\otimes\mathcal{W}_{D^{3}}\right)  }%
\end{array}
\right) \\
\times_{M\otimes\mathcal{W}_{D(2)}}\\
\left(
\begin{array}
[c]{c}%
\underset{213}{\left(  M\otimes\mathcal{W}_{D^{3}}\right)  }\\
\times_{M\otimes\mathcal{W}_{D^{3}\{(1,3)\}}}\\
\underset{231}{\left(  M\otimes\mathcal{W}_{D^{3}}\right)  }%
\end{array}
\right)
\end{array}
& \times_{M\otimes\mathcal{W}_{D^{3}\oplus D^{3}}} &
\begin{array}
[c]{c}%
\left(
\begin{array}
[c]{c}%
\underset{213}{\left(  M\otimes\mathcal{W}_{D^{3}}\right)  }\\
\times_{M\otimes\mathcal{W}_{D^{3}\{(1,2)\}}}\\
\underset{123}{\left(  M\otimes\mathcal{W}_{D^{3}}\right)  }%
\end{array}
\right) \\
\times_{M\otimes\mathcal{W}_{D(2)}}\\
\left(
\begin{array}
[c]{c}%
\underset{321}{\left(  M\otimes\mathcal{W}_{D^{3}}\right)  }\\
\times_{M\otimes\mathcal{W}_{D^{3}\{(1,2)\}}}\\
\underset{312}{\left(  M\otimes\mathcal{W}_{D^{3}}\right)  }%
\end{array}
\right)
\end{array}
\end{array}
\right] \\
& \rightarrow M\otimes\mathcal{W}_{D}%
\end{align*}
for the composition of morphisms
\begin{align*}
& \pi_{\left(  \left(  \underset{321}{\left(  M\otimes\mathcal{W}_{D^{3}%
}\right)  }\times_{M\otimes\mathcal{W}_{D^{3}\{(2,3)\}}}\underset{231}{\left(
M\otimes\mathcal{W}_{D^{3}}\right)  }\right)  \times_{M\otimes\mathcal{W}%
_{D(2)}}\left(  \underset{132}{\left(  M\otimes\mathcal{W}_{D^{3}}\right)
}\times_{M\otimes\mathcal{W}_{D^{3}\{(2,3)\}}}\underset{123}{\left(
M\otimes\mathcal{W}_{D^{3}}\right)  }\right)  \right)  }^{\triangle}:\\
& :\left[
\begin{array}
[c]{ccc}
&
\begin{array}
[c]{c}%
\left(
\begin{array}
[c]{c}%
\underset{321}{\left(  M\otimes\mathcal{W}_{D^{3}}\right)  }\\
\times_{M\otimes\mathcal{W}_{D^{3}\{(2,3)\}}}\\
\underset{231}{\left(  M\otimes\mathcal{W}_{D^{3}}\right)  }%
\end{array}
\right) \\
\times_{M\otimes\mathcal{W}_{D(2)}}\\
\left(
\begin{array}
[c]{c}%
\underset{132}{\left(  M\otimes\mathcal{W}_{D^{3}}\right)  }\\
\times_{M\otimes\mathcal{W}_{D^{3}\{(2,3)\}}}\\
\underset{123}{\left(  M\otimes\mathcal{W}_{D^{3}}\right)  }%
\end{array}
\right)
\end{array}
& \\%
\begin{array}
[c]{c}%
\times_{M\otimes\mathcal{W}_{D^{3}\oplus D^{3}}}\\
\,\\
\,
\end{array}
&  &
\begin{array}
[c]{c}%
\times_{M\otimes\mathcal{W}_{D^{3}\oplus D^{3}}}\\
\,\\
\,
\end{array}
\\%
\begin{array}
[c]{c}%
\left(
\begin{array}
[c]{c}%
\underset{132}{\left(  M\otimes\mathcal{W}_{D^{3}}\right)  }\\
\times_{M\otimes\mathcal{W}_{D^{3}\{(1,3)\}}}\\
\underset{312}{\left(  M\otimes\mathcal{W}_{D^{3}}\right)  }%
\end{array}
\right) \\
\times_{M\otimes\mathcal{W}_{D(2)}}\\
\left(
\begin{array}
[c]{c}%
\underset{213}{\left(  M\otimes\mathcal{W}_{D^{3}}\right)  }\\
\times_{M\otimes\mathcal{W}_{D^{3}\{(1,3)\}}}\\
\underset{231}{\left(  M\otimes\mathcal{W}_{D^{3}}\right)  }%
\end{array}
\right)
\end{array}
& \times_{M\otimes\mathcal{W}_{D^{3}\oplus D^{3}}} &
\begin{array}
[c]{c}%
\left(
\begin{array}
[c]{c}%
\underset{213}{\left(  M\otimes\mathcal{W}_{D^{3}}\right)  }\\
\times_{M\otimes\mathcal{W}_{D^{3}\{(1,2)\}}}\\
\underset{123}{\left(  M\otimes\mathcal{W}_{D^{3}}\right)  }%
\end{array}
\right) \\
\times_{M\otimes\mathcal{W}_{D(2)}}\\
\left(
\begin{array}
[c]{c}%
\underset{321}{\left(  M\otimes\mathcal{W}_{D^{3}}\right)  }\\
\times_{M\otimes\mathcal{W}_{D^{3}\{(1,2)\}}}\\
\underset{312}{\left(  M\otimes\mathcal{W}_{D^{3}}\right)  }%
\end{array}
\right)
\end{array}
\end{array}
\right] \\
& \rightarrow\left(
\begin{array}
[c]{c}%
\underset{321}{\left(  M\otimes\mathcal{W}_{D^{3}}\right)  }\\
\times_{M\otimes\mathcal{W}_{D^{3}\{(2,3)\}}}\\
\underset{231}{\left(  M\otimes\mathcal{W}_{D^{3}}\right)  }%
\end{array}
\right)  \times_{M\otimes\mathcal{W}_{D(2)}}\left(
\begin{array}
[c]{c}%
\underset{132}{\left(  M\otimes\mathcal{W}_{D^{3}}\right)  }\\
\times_{M\otimes\mathcal{W}_{D^{3}\{(2,3)\}}}\\
\underset{123}{\left(  M\otimes\mathcal{W}_{D^{3}}\right)  }%
\end{array}
\right)
\end{align*}
\begin{align*}
& \zeta^{\overset{\cdot}{\underset{1}{-}}}\times_{M\otimes\mathcal{W}_{D(2)}%
}\zeta^{\overset{\cdot}{\underset{1}{-}}}:\\
& :\left(
\begin{array}
[c]{c}%
\underset{321}{\left(  M\otimes\mathcal{W}_{D^{3}}\right)  }\\
\times_{M\otimes\mathcal{W}_{D^{3}\{(2,3)\}}}\\
\underset{231}{\left(  M\otimes\mathcal{W}_{D^{3}}\right)  }%
\end{array}
\right)  \times_{M\otimes\mathcal{W}_{D(2)}}\left(
\begin{array}
[c]{c}%
\underset{132}{\left(  M\otimes\mathcal{W}_{D^{3}}\right)  }\\
\times_{M\otimes\mathcal{W}_{D^{3}\{(2,3)\}}}\\
\underset{123}{\left(  M\otimes\mathcal{W}_{D^{3}}\right)  }%
\end{array}
\right) \\
& \rightarrow\left(  M\otimes\mathcal{W}_{D^{2}}\right)  \times_{M\otimes
\mathcal{W}_{D(2)}}\left(  M\otimes\mathcal{W}_{D^{2}}\right)
\end{align*}
\[
\zeta^{\overset{\cdot}{-}}:\left(  M\otimes\mathcal{W}_{D^{2}}\right)
\times_{M\otimes\mathcal{W}_{D(2)}}\left(  M\otimes\mathcal{W}_{D^{2}}\right)
\rightarrow M\otimes\mathcal{W}_{D}%
\]
in succession.

\item We will write the morphism
\begin{align*}
& \zeta^{(\ast_{231}\overset{\cdot}{\underset{2}{-}}\ast_{213})\overset{\cdot
}{-}(\ast_{312}\overset{\cdot}{\underset{2}{-}}\ast_{132})}:\\
& :\left[
\begin{array}
[c]{ccc}
&
\begin{array}
[c]{c}%
\left(
\begin{array}
[c]{c}%
\underset{321}{\left(  M\otimes\mathcal{W}_{D^{3}}\right)  }\\
\times_{M\otimes\mathcal{W}_{D^{3}\{(2,3)\}}}\\
\underset{231}{\left(  M\otimes\mathcal{W}_{D^{3}}\right)  }%
\end{array}
\right) \\
\times_{M\otimes\mathcal{W}_{D(2)}}\\
\left(
\begin{array}
[c]{c}%
\underset{132}{\left(  M\otimes\mathcal{W}_{D^{3}}\right)  }\\
\times_{M\otimes\mathcal{W}_{D^{3}\{(2,3)\}}}\\
\underset{123}{\left(  M\otimes\mathcal{W}_{D^{3}}\right)  }%
\end{array}
\right)
\end{array}
& \\%
\begin{array}
[c]{c}%
\times_{M\otimes\mathcal{W}_{D^{3}\oplus D^{3}}}\\
\,\\
\,
\end{array}
&  &
\begin{array}
[c]{c}%
\times_{M\otimes\mathcal{W}_{D^{3}\oplus D^{3}}}\\
\,\\
\,
\end{array}
\\%
\begin{array}
[c]{c}%
\left(
\begin{array}
[c]{c}%
\underset{132}{\left(  M\otimes\mathcal{W}_{D^{3}}\right)  }\\
\times_{M\otimes\mathcal{W}_{D^{3}\{(1,3)\}}}\\
\underset{312}{\left(  M\otimes\mathcal{W}_{D^{3}}\right)  }%
\end{array}
\right) \\
\times_{M\otimes\mathcal{W}_{D(2)}}\\
\left(
\begin{array}
[c]{c}%
\underset{213}{\left(  M\otimes\mathcal{W}_{D^{3}}\right)  }\\
\times_{M\otimes\mathcal{W}_{D^{3}\{(1,3)\}}}\\
\underset{231}{\left(  M\otimes\mathcal{W}_{D^{3}}\right)  }%
\end{array}
\right)
\end{array}
& \times_{M\otimes\mathcal{W}_{D^{3}\oplus D^{3}}} &
\begin{array}
[c]{c}%
\left(
\begin{array}
[c]{c}%
\underset{213}{\left(  M\otimes\mathcal{W}_{D^{3}}\right)  }\\
\times_{M\otimes\mathcal{W}_{D^{3}\{(1,2)\}}}\\
\underset{123}{\left(  M\otimes\mathcal{W}_{D^{3}}\right)  }%
\end{array}
\right) \\
\times_{M\otimes\mathcal{W}_{D(2)}}\\
\left(
\begin{array}
[c]{c}%
\underset{321}{\left(  M\otimes\mathcal{W}_{D^{3}}\right)  }\\
\times_{M\otimes\mathcal{W}_{D^{3}\{(1,2)\}}}\\
\underset{312}{\left(  M\otimes\mathcal{W}_{D^{3}}\right)  }%
\end{array}
\right)
\end{array}
\end{array}
\right] \\
& \rightarrow M\otimes\mathcal{W}_{D}%
\end{align*}
for the composition of morphisms
\begin{align*}
& \pi_{\left(  \left(  \underset{132}{\left(  M\otimes\mathcal{W}_{D^{3}%
}\right)  }\times_{M\otimes\mathcal{W}_{D^{3}\{(1,3)\}}}\underset{312}{\left(
M\otimes\mathcal{W}_{D^{3}}\right)  }\right)  \times_{M\otimes\mathcal{W}%
_{D(2)}}\left(  \underset{213}{\left(  M\otimes\mathcal{W}_{D^{3}}\right)
}\times_{M\otimes\mathcal{W}_{D^{3}\{(1,3)\}}}\underset{231}{\left(
M\otimes\mathcal{W}_{D^{3}}\right)  }\right)  \right)  }^{\triangle}:\\
& :\left[
\begin{array}
[c]{ccc}
&
\begin{array}
[c]{c}%
\left(
\begin{array}
[c]{c}%
\underset{321}{\left(  M\otimes\mathcal{W}_{D^{3}}\right)  }\\
\times_{M\otimes\mathcal{W}_{D^{3}\{(2,3)\}}}\\
\underset{231}{\left(  M\otimes\mathcal{W}_{D^{3}}\right)  }%
\end{array}
\right) \\
\times_{M\otimes\mathcal{W}_{D(2)}}\\
\left(
\begin{array}
[c]{c}%
\underset{132}{\left(  M\otimes\mathcal{W}_{D^{3}}\right)  }\\
\times_{M\otimes\mathcal{W}_{D^{3}\{(2,3)\}}}\\
\underset{123}{\left(  M\otimes\mathcal{W}_{D^{3}}\right)  }%
\end{array}
\right)
\end{array}
& \\%
\begin{array}
[c]{c}%
\times_{M\otimes\mathcal{W}_{D^{3}\oplus D^{3}}}\\
\,\\
\,
\end{array}
&  &
\begin{array}
[c]{c}%
\times_{M\otimes\mathcal{W}_{D^{3}\oplus D^{3}}}\\
\,\\
\,
\end{array}
\\%
\begin{array}
[c]{c}%
\left(
\begin{array}
[c]{c}%
\underset{132}{\left(  M\otimes\mathcal{W}_{D^{3}}\right)  }\\
\times_{M\otimes\mathcal{W}_{D^{3}\{(1,3)\}}}\\
\underset{312}{\left(  M\otimes\mathcal{W}_{D^{3}}\right)  }%
\end{array}
\right) \\
\times_{M\otimes\mathcal{W}_{D(2)}}\\
\left(
\begin{array}
[c]{c}%
\underset{213}{\left(  M\otimes\mathcal{W}_{D^{3}}\right)  }\\
\times_{M\otimes\mathcal{W}_{D^{3}\{(1,3)\}}}\\
\underset{231}{\left(  M\otimes\mathcal{W}_{D^{3}}\right)  }%
\end{array}
\right)
\end{array}
& \times_{M\otimes\mathcal{W}_{D^{3}\oplus D^{3}}} &
\begin{array}
[c]{c}%
\left(
\begin{array}
[c]{c}%
\underset{213}{\left(  M\otimes\mathcal{W}_{D^{3}}\right)  }\\
\times_{M\otimes\mathcal{W}_{D^{3}\{(1,2)\}}}\\
\underset{123}{\left(  M\otimes\mathcal{W}_{D^{3}}\right)  }%
\end{array}
\right) \\
\times_{M\otimes\mathcal{W}_{D(2)}}\\
\left(
\begin{array}
[c]{c}%
\underset{321}{\left(  M\otimes\mathcal{W}_{D^{3}}\right)  }\\
\times_{M\otimes\mathcal{W}_{D^{3}\{(1,2)\}}}\\
\underset{312}{\left(  M\otimes\mathcal{W}_{D^{3}}\right)  }%
\end{array}
\right)
\end{array}
\end{array}
\right] \\
& \rightarrow\left(
\begin{array}
[c]{c}%
\underset{132}{\left(  M\otimes\mathcal{W}_{D^{3}}\right)  }\\
\times_{M\otimes\mathcal{W}_{D^{3}\{(1,3)\}}}\\
\underset{312}{\left(  M\otimes\mathcal{W}_{D^{3}}\right)  }%
\end{array}
\right)  \times_{M\otimes\mathcal{W}_{D(2)}}\left(
\begin{array}
[c]{c}%
\underset{213}{\left(  M\otimes\mathcal{W}_{D^{3}}\right)  }\\
\times_{M\otimes\mathcal{W}_{D^{3}\{(1,3)\}}}\\
\underset{231}{\left(  M\otimes\mathcal{W}_{D^{3}}\right)  }%
\end{array}
\right)
\end{align*}
\begin{align*}
& \zeta^{\overset{\cdot}{\underset{2}{-}}}\times_{M\otimes\mathcal{W}_{D(2)}%
}\zeta^{\overset{\cdot}{\underset{2}{-}}}:\\
& :\left(
\begin{array}
[c]{c}%
\underset{132}{\left(  M\otimes\mathcal{W}_{D^{3}}\right)  }\\
\times_{M\otimes\mathcal{W}_{D^{3}\{(1,3)\}}}\\
\underset{312}{\left(  M\otimes\mathcal{W}_{D^{3}}\right)  }%
\end{array}
\right)  \times_{M\otimes\mathcal{W}_{D(2)}}\left(
\begin{array}
[c]{c}%
\underset{213}{\left(  M\otimes\mathcal{W}_{D^{3}}\right)  }\\
\times_{M\otimes\mathcal{W}_{D^{3}\{(1,3)\}}}\\
\underset{231}{\left(  M\otimes\mathcal{W}_{D^{3}}\right)  }%
\end{array}
\right) \\
& \rightarrow\left(  M\otimes\mathcal{W}_{D^{2}}\right)  \times_{M\otimes
\mathcal{W}_{D(2)}}\left(  M\otimes\mathcal{W}_{D^{2}}\right)
\end{align*}
\[
\zeta^{\overset{\cdot}{-}}:\left(  M\otimes\mathcal{W}_{D^{2}}\right)
\times_{M\otimes\mathcal{W}_{D(2)}}\left(  M\otimes\mathcal{W}_{D^{2}}\right)
\rightarrow M\otimes\mathcal{W}_{D}%
\]
in succession.

\item We will write the morphism
\begin{align*}
& \zeta^{(\ast_{312}\overset{\cdot}{\underset{3}{-}}\ast_{321})\overset{\cdot
}{-}(\ast_{123}\overset{\cdot}{\underset{3}{-}}\ast_{213})}:\\
& :\left[
\begin{array}
[c]{ccc}
&
\begin{array}
[c]{c}%
\left(
\begin{array}
[c]{c}%
\underset{321}{\left(  M\otimes\mathcal{W}_{D^{3}}\right)  }\\
\times_{M\otimes\mathcal{W}_{D^{3}\{(2,3)\}}}\\
\underset{231}{\left(  M\otimes\mathcal{W}_{D^{3}}\right)  }%
\end{array}
\right) \\
\times_{M\otimes\mathcal{W}_{D(2)}}\\
\left(
\begin{array}
[c]{c}%
\underset{132}{\left(  M\otimes\mathcal{W}_{D^{3}}\right)  }\\
\times_{M\otimes\mathcal{W}_{D^{3}\{(2,3)\}}}\\
\underset{123}{\left(  M\otimes\mathcal{W}_{D^{3}}\right)  }%
\end{array}
\right)
\end{array}
& \\%
\begin{array}
[c]{c}%
\times_{M\otimes\mathcal{W}_{D^{3}\oplus D^{3}}}\\
\,\\
\,
\end{array}
&  &
\begin{array}
[c]{c}%
\times_{M\otimes\mathcal{W}_{D^{3}\oplus D^{3}}}\\
\,\\
\,
\end{array}
\\%
\begin{array}
[c]{c}%
\left(
\begin{array}
[c]{c}%
\underset{132}{\left(  M\otimes\mathcal{W}_{D^{3}}\right)  }\\
\times_{M\otimes\mathcal{W}_{D^{3}\{(1,3)\}}}\\
\underset{312}{\left(  M\otimes\mathcal{W}_{D^{3}}\right)  }%
\end{array}
\right) \\
\times_{M\otimes\mathcal{W}_{D(2)}}\\
\left(
\begin{array}
[c]{c}%
\underset{213}{\left(  M\otimes\mathcal{W}_{D^{3}}\right)  }\\
\times_{M\otimes\mathcal{W}_{D^{3}\{(1,3)\}}}\\
\underset{231}{\left(  M\otimes\mathcal{W}_{D^{3}}\right)  }%
\end{array}
\right)
\end{array}
& \times_{M\otimes\mathcal{W}_{D^{3}\oplus D^{3}}} &
\begin{array}
[c]{c}%
\left(
\begin{array}
[c]{c}%
\underset{213}{\left(  M\otimes\mathcal{W}_{D^{3}}\right)  }\\
\times_{M\otimes\mathcal{W}_{D^{3}\{(1,2)\}}}\\
\underset{123}{\left(  M\otimes\mathcal{W}_{D^{3}}\right)  }%
\end{array}
\right) \\
\times_{M\otimes\mathcal{W}_{D(2)}}\\
\left(
\begin{array}
[c]{c}%
\underset{321}{\left(  M\otimes\mathcal{W}_{D^{3}}\right)  }\\
\times_{M\otimes\mathcal{W}_{D^{3}\{(1,2)\}}}\\
\underset{312}{\left(  M\otimes\mathcal{W}_{D^{3}}\right)  }%
\end{array}
\right)
\end{array}
\end{array}
\right] \\
& \rightarrow M\otimes\mathcal{W}_{D}%
\end{align*}
for the composition of morphisms
\begin{align*}
& \pi_{\left(  \left(  \underset{213}{\left(  M\otimes\mathcal{W}_{D^{3}%
}\right)  }\times_{M\otimes\mathcal{W}_{D^{3}\{(1,3)\}}}\underset{123}{\left(
M\otimes\mathcal{W}_{D^{3}}\right)  }\right)  \times_{M\otimes\mathcal{W}%
_{D(2)}}\left(  \underset{321}{\left(  M\otimes\mathcal{W}_{D^{3}}\right)
}\times_{M\otimes\mathcal{W}_{D^{3}\{(1,3)\}}}\underset{312}{\left(
M\otimes\mathcal{W}_{D^{3}}\right)  }\right)  \right)  }^{\triangle}:\\
& :\left[
\begin{array}
[c]{ccc}
&
\begin{array}
[c]{c}%
\left(
\begin{array}
[c]{c}%
\underset{321}{\left(  M\otimes\mathcal{W}_{D^{3}}\right)  }\\
\times_{M\otimes\mathcal{W}_{D^{3}\{(2,3)\}}}\\
\underset{231}{\left(  M\otimes\mathcal{W}_{D^{3}}\right)  }%
\end{array}
\right) \\
\times_{M\otimes\mathcal{W}_{D(2)}}\\
\left(
\begin{array}
[c]{c}%
\underset{132}{\left(  M\otimes\mathcal{W}_{D^{3}}\right)  }\\
\times_{M\otimes\mathcal{W}_{D^{3}\{(2,3)\}}}\\
\underset{123}{\left(  M\otimes\mathcal{W}_{D^{3}}\right)  }%
\end{array}
\right)
\end{array}
& \\%
\begin{array}
[c]{c}%
\times_{M\otimes\mathcal{W}_{D^{3}\oplus D^{3}}}\\
\,\\
\,
\end{array}
&  &
\begin{array}
[c]{c}%
\times_{M\otimes\mathcal{W}_{D^{3}\oplus D^{3}}}\\
\,\\
\,
\end{array}
\\%
\begin{array}
[c]{c}%
\left(
\begin{array}
[c]{c}%
\underset{132}{\left(  M\otimes\mathcal{W}_{D^{3}}\right)  }\\
\times_{M\otimes\mathcal{W}_{D^{3}\{(1,3)\}}}\\
\underset{312}{\left(  M\otimes\mathcal{W}_{D^{3}}\right)  }%
\end{array}
\right) \\
\times_{M\otimes\mathcal{W}_{D(2)}}\\
\left(
\begin{array}
[c]{c}%
\underset{213}{\left(  M\otimes\mathcal{W}_{D^{3}}\right)  }\\
\times_{M\otimes\mathcal{W}_{D^{3}\{(1,3)\}}}\\
\underset{231}{\left(  M\otimes\mathcal{W}_{D^{3}}\right)  }%
\end{array}
\right)
\end{array}
& \times_{M\otimes\mathcal{W}_{D^{3}\oplus D^{3}}} &
\begin{array}
[c]{c}%
\left(
\begin{array}
[c]{c}%
\underset{213}{\left(  M\otimes\mathcal{W}_{D^{3}}\right)  }\\
\times_{M\otimes\mathcal{W}_{D^{3}\{(1,2)\}}}\\
\underset{123}{\left(  M\otimes\mathcal{W}_{D^{3}}\right)  }%
\end{array}
\right) \\
\times_{M\otimes\mathcal{W}_{D(2)}}\\
\left(
\begin{array}
[c]{c}%
\underset{321}{\left(  M\otimes\mathcal{W}_{D^{3}}\right)  }\\
\times_{M\otimes\mathcal{W}_{D^{3}\{(1,2)\}}}\\
\underset{312}{\left(  M\otimes\mathcal{W}_{D^{3}}\right)  }%
\end{array}
\right)
\end{array}
\end{array}
\right] \\
& \rightarrow\left(
\begin{array}
[c]{c}%
\underset{213}{\left(  M\otimes\mathcal{W}_{D^{3}}\right)  }\\
\times_{M\otimes\mathcal{W}_{D^{3}\{(1,2)\}}}\\
\underset{123}{\left(  M\otimes\mathcal{W}_{D^{3}}\right)  }%
\end{array}
\right)  \times_{M\otimes\mathcal{W}_{D(2)}}\left(
\begin{array}
[c]{c}%
\underset{321}{\left(  M\otimes\mathcal{W}_{D^{3}}\right)  }\\
\times_{M\otimes\mathcal{W}_{D^{3}\{(1,2)\}}}\\
\underset{312}{\left(  M\otimes\mathcal{W}_{D^{3}}\right)  }%
\end{array}
\right)
\end{align*}
\begin{align*}
& \zeta^{\overset{\cdot}{\underset{3}{-}}}\times_{M\otimes\mathcal{W}_{D(2)}%
}\zeta^{\overset{\cdot}{\underset{3}{-}}}\\
& :\left(
\begin{array}
[c]{c}%
\underset{213}{\left(  M\otimes\mathcal{W}_{D^{3}}\right)  }\\
\times_{M\otimes\mathcal{W}_{D^{3}\{(1,2)\}}}\\
\underset{123}{\left(  M\otimes\mathcal{W}_{D^{3}}\right)  }%
\end{array}
\right)  \times_{M\otimes\mathcal{W}_{D(2)}}\left(
\begin{array}
[c]{c}%
\underset{321}{\left(  M\otimes\mathcal{W}_{D^{3}}\right)  }\\
\times_{M\otimes\mathcal{W}_{D^{3}\{(1,2)\}}}\\
\underset{312}{\left(  M\otimes\mathcal{W}_{D^{3}}\right)  }%
\end{array}
\right) \\
& \rightarrow\left(  M\otimes\mathcal{W}_{D^{2}}\right)  \times_{M\otimes
\mathcal{W}_{D(2)}}\left(  M\otimes\mathcal{W}_{D^{2}}\right)
\end{align*}
\[
\zeta^{\overset{\cdot}{-}}:\left(  M\otimes\mathcal{W}_{D^{2}}\right)
\times_{M\otimes\mathcal{W}_{D(2)}}\left(  M\otimes\mathcal{W}_{D^{2}}\right)
\rightarrow M\otimes\mathcal{W}_{D}%
\]
in succession.
\end{enumerate}
\end{notation}

\begin{theorem}
\label{t3.16}(\underline{The general Jacobi Identity}) The three morphisms
\begin{align*}
& \zeta^{(\ast_{123}\overset{\cdot}{\underset{1}{-}}\ast_{132})\overset{\cdot
}{-}(\ast_{231}\overset{\cdot}{\underset{1}{-}}\ast_{321})}:\\
& :\left[
\begin{array}
[c]{ccc}
&
\begin{array}
[c]{c}%
\left(
\begin{array}
[c]{c}%
\underset{321}{\left(  M\otimes\mathcal{W}_{D^{3}}\right)  }\\
\times_{M\otimes\mathcal{W}_{D^{3}\{(2,3)\}}}\\
\underset{231}{\left(  M\otimes\mathcal{W}_{D^{3}}\right)  }%
\end{array}
\right) \\
\times_{M\otimes\mathcal{W}_{D(2)}}\\
\left(
\begin{array}
[c]{c}%
\underset{132}{\left(  M\otimes\mathcal{W}_{D^{3}}\right)  }\\
\times_{M\otimes\mathcal{W}_{D^{3}\{(2,3)\}}}\\
\underset{123}{\left(  M\otimes\mathcal{W}_{D^{3}}\right)  }%
\end{array}
\right)
\end{array}
& \\%
\begin{array}
[c]{c}%
\times_{M\otimes\mathcal{W}_{D^{3}\oplus D^{3}}}\\
\,\\
\,
\end{array}
&  &
\begin{array}
[c]{c}%
\times_{M\otimes\mathcal{W}_{D^{3}\oplus D^{3}}}\\
\,\\
\,
\end{array}
\\%
\begin{array}
[c]{c}%
\left(
\begin{array}
[c]{c}%
\underset{132}{\left(  M\otimes\mathcal{W}_{D^{3}}\right)  }\\
\times_{M\otimes\mathcal{W}_{D^{3}\{(1,3)\}}}\\
\underset{312}{\left(  M\otimes\mathcal{W}_{D^{3}}\right)  }%
\end{array}
\right) \\
\times_{M\otimes\mathcal{W}_{D(2)}}\\
\left(
\begin{array}
[c]{c}%
\underset{213}{\left(  M\otimes\mathcal{W}_{D^{3}}\right)  }\\
\times_{M\otimes\mathcal{W}_{D^{3}\{(1,3)\}}}\\
\underset{231}{\left(  M\otimes\mathcal{W}_{D^{3}}\right)  }%
\end{array}
\right)
\end{array}
& \times_{M\otimes\mathcal{W}_{D^{3}\oplus D^{3}}} &
\begin{array}
[c]{c}%
\left(
\begin{array}
[c]{c}%
\underset{213}{\left(  M\otimes\mathcal{W}_{D^{3}}\right)  }\\
\times_{M\otimes\mathcal{W}_{D^{3}\{(1,2)\}}}\\
\underset{123}{\left(  M\otimes\mathcal{W}_{D^{3}}\right)  }%
\end{array}
\right) \\
\times_{M\otimes\mathcal{W}_{D(2)}}\\
\left(
\begin{array}
[c]{c}%
\underset{321}{\left(  M\otimes\mathcal{W}_{D^{3}}\right)  }\\
\times_{M\otimes\mathcal{W}_{D^{3}\{(1,2)\}}}\\
\underset{312}{\left(  M\otimes\mathcal{W}_{D^{3}}\right)  }%
\end{array}
\right)
\end{array}
\end{array}
\right] \\
& \rightarrow M\otimes\mathcal{W}_{D}%
\end{align*}
\begin{align*}
& \zeta^{(\ast_{231}\overset{\cdot}{\underset{2}{-}}\ast_{213})\overset{\cdot
}{-}(\ast_{312}\overset{\cdot}{\underset{2}{-}}\ast_{132})}:\\
& :\left[
\begin{array}
[c]{ccc}
&
\begin{array}
[c]{c}%
\left(
\begin{array}
[c]{c}%
\underset{321}{\left(  M\otimes\mathcal{W}_{D^{3}}\right)  }\\
\times_{M\otimes\mathcal{W}_{D^{3}\{(2,3)\}}}\\
\underset{231}{\left(  M\otimes\mathcal{W}_{D^{3}}\right)  }%
\end{array}
\right) \\
\times_{M\otimes\mathcal{W}_{D(2)}}\\
\left(
\begin{array}
[c]{c}%
\underset{132}{\left(  M\otimes\mathcal{W}_{D^{3}}\right)  }\\
\times_{M\otimes\mathcal{W}_{D^{3}\{(2,3)\}}}\\
\underset{123}{\left(  M\otimes\mathcal{W}_{D^{3}}\right)  }%
\end{array}
\right)
\end{array}
& \\%
\begin{array}
[c]{c}%
\times_{M\otimes\mathcal{W}_{D^{3}\oplus D^{3}}}\\
\,\\
\,
\end{array}
&  &
\begin{array}
[c]{c}%
\times_{M\otimes\mathcal{W}_{D^{3}\oplus D^{3}}}\\
\,\\
\,
\end{array}
\\%
\begin{array}
[c]{c}%
\left(
\begin{array}
[c]{c}%
\underset{132}{\left(  M\otimes\mathcal{W}_{D^{3}}\right)  }\\
\times_{M\otimes\mathcal{W}_{D^{3}\{(1,3)\}}}\\
\underset{312}{\left(  M\otimes\mathcal{W}_{D^{3}}\right)  }%
\end{array}
\right) \\
\times_{M\otimes\mathcal{W}_{D(2)}}\\
\left(
\begin{array}
[c]{c}%
\underset{213}{\left(  M\otimes\mathcal{W}_{D^{3}}\right)  }\\
\times_{M\otimes\mathcal{W}_{D^{3}\{(1,3)\}}}\\
\underset{231}{\left(  M\otimes\mathcal{W}_{D^{3}}\right)  }%
\end{array}
\right)
\end{array}
& \times_{M\otimes\mathcal{W}_{D^{3}\oplus D^{3}}} &
\begin{array}
[c]{c}%
\left(
\begin{array}
[c]{c}%
\underset{213}{\left(  M\otimes\mathcal{W}_{D^{3}}\right)  }\\
\times_{M\otimes\mathcal{W}_{D^{3}\{(1,2)\}}}\\
\underset{123}{\left(  M\otimes\mathcal{W}_{D^{3}}\right)  }%
\end{array}
\right) \\
\times_{M\otimes\mathcal{W}_{D(2)}}\\
\left(
\begin{array}
[c]{c}%
\underset{321}{\left(  M\otimes\mathcal{W}_{D^{3}}\right)  }\\
\times_{M\otimes\mathcal{W}_{D^{3}\{(1,2)\}}}\\
\underset{312}{\left(  M\otimes\mathcal{W}_{D^{3}}\right)  }%
\end{array}
\right)
\end{array}
\end{array}
\right] \\
& \rightarrow M\otimes\mathcal{W}_{D}%
\end{align*}
\begin{align*}
& \zeta^{(\ast_{312}\overset{\cdot}{\underset{3}{-}}\ast_{321})\overset{\cdot
}{-}(\ast_{123}\overset{\cdot}{\underset{3}{-}}\ast_{213})}:\\
& :\left[
\begin{array}
[c]{ccc}
&
\begin{array}
[c]{c}%
\left(
\begin{array}
[c]{c}%
\underset{321}{\left(  M\otimes\mathcal{W}_{D^{3}}\right)  }\\
\times_{M\otimes\mathcal{W}_{D^{3}\{(2,3)\}}}\\
\underset{231}{\left(  M\otimes\mathcal{W}_{D^{3}}\right)  }%
\end{array}
\right) \\
\times_{M\otimes\mathcal{W}_{D(2)}}\\
\left(
\begin{array}
[c]{c}%
\underset{132}{\left(  M\otimes\mathcal{W}_{D^{3}}\right)  }\\
\times_{M\otimes\mathcal{W}_{D^{3}\{(2,3)\}}}\\
\underset{123}{\left(  M\otimes\mathcal{W}_{D^{3}}\right)  }%
\end{array}
\right)
\end{array}
& \\%
\begin{array}
[c]{c}%
\times_{M\otimes\mathcal{W}_{D^{3}\oplus D^{3}}}\\
\,\\
\,
\end{array}
&  &
\begin{array}
[c]{c}%
\times_{M\otimes\mathcal{W}_{D^{3}\oplus D^{3}}}\\
\,\\
\,
\end{array}
\\%
\begin{array}
[c]{c}%
\left(
\begin{array}
[c]{c}%
\underset{132}{\left(  M\otimes\mathcal{W}_{D^{3}}\right)  }\\
\times_{M\otimes\mathcal{W}_{D^{3}\{(1,3)\}}}\\
\underset{312}{\left(  M\otimes\mathcal{W}_{D^{3}}\right)  }%
\end{array}
\right) \\
\times_{M\otimes\mathcal{W}_{D(2)}}\\
\left(
\begin{array}
[c]{c}%
\underset{213}{\left(  M\otimes\mathcal{W}_{D^{3}}\right)  }\\
\times_{M\otimes\mathcal{W}_{D^{3}\{(1,3)\}}}\\
\underset{231}{\left(  M\otimes\mathcal{W}_{D^{3}}\right)  }%
\end{array}
\right)
\end{array}
& \times_{M\otimes\mathcal{W}_{D^{3}\oplus D^{3}}} &
\begin{array}
[c]{c}%
\left(
\begin{array}
[c]{c}%
\underset{213}{\left(  M\otimes\mathcal{W}_{D^{3}}\right)  }\\
\times_{M\otimes\mathcal{W}_{D^{3}\{(1,2)\}}}\\
\underset{123}{\left(  M\otimes\mathcal{W}_{D^{3}}\right)  }%
\end{array}
\right) \\
\times_{M\otimes\mathcal{W}_{D(2)}}\\
\left(
\begin{array}
[c]{c}%
\underset{321}{\left(  M\otimes\mathcal{W}_{D^{3}}\right)  }\\
\times_{M\otimes\mathcal{W}_{D^{3}\{(1,2)\}}}\\
\underset{312}{\left(  M\otimes\mathcal{W}_{D^{3}}\right)  }%
\end{array}
\right)
\end{array}
\end{array}
\right] \\
& \rightarrow M\otimes\mathcal{W}_{D}%
\end{align*}
sum up only to vanish.
\end{theorem}

\begin{proof}
The proof is divided into four steps.

\begin{enumerate}
\item The morphism
\begin{align*}
& \zeta^{(\ast_{123}\overset{\cdot}{\underset{1}{-}}\ast_{132})\overset{\cdot
}{-}(\ast_{231}\overset{\cdot}{\underset{1}{-}}\ast_{321})}:\\
& :\left[
\begin{array}
[c]{ccc}
&
\begin{array}
[c]{c}%
\left(
\begin{array}
[c]{c}%
\underset{321}{\left(  M\otimes\mathcal{W}_{D^{3}}\right)  }\\
\times_{M\otimes\mathcal{W}_{D^{3}\{(2,3)\}}}\\
\underset{231}{\left(  M\otimes\mathcal{W}_{D^{3}}\right)  }%
\end{array}
\right) \\
\times_{M\otimes\mathcal{W}_{D(2)}}\\
\left(
\begin{array}
[c]{c}%
\underset{132}{\left(  M\otimes\mathcal{W}_{D^{3}}\right)  }\\
\times_{M\otimes\mathcal{W}_{D^{3}\{(2,3)\}}}\\
\underset{123}{\left(  M\otimes\mathcal{W}_{D^{3}}\right)  }%
\end{array}
\right)
\end{array}
& \\%
\begin{array}
[c]{c}%
\times_{M\otimes\mathcal{W}_{D^{3}\oplus D^{3}}}\\
\,\\
\,
\end{array}
&  &
\begin{array}
[c]{c}%
\times_{M\otimes\mathcal{W}_{D^{3}\oplus D^{3}}}\\
\,\\
\,
\end{array}
\\%
\begin{array}
[c]{c}%
\left(
\begin{array}
[c]{c}%
\underset{132}{\left(  M\otimes\mathcal{W}_{D^{3}}\right)  }\\
\times_{M\otimes\mathcal{W}_{D^{3}\{(1,3)\}}}\\
\underset{312}{\left(  M\otimes\mathcal{W}_{D^{3}}\right)  }%
\end{array}
\right) \\
\times_{M\otimes\mathcal{W}_{D(2)}}\\
\left(
\begin{array}
[c]{c}%
\underset{213}{\left(  M\otimes\mathcal{W}_{D^{3}}\right)  }\\
\times_{M\otimes\mathcal{W}_{D^{3}\{(1,3)\}}}\\
\underset{231}{\left(  M\otimes\mathcal{W}_{D^{3}}\right)  }%
\end{array}
\right)
\end{array}
& \times_{M\otimes\mathcal{W}_{D^{3}\oplus D^{3}}} &
\begin{array}
[c]{c}%
\left(
\begin{array}
[c]{c}%
\underset{213}{\left(  M\otimes\mathcal{W}_{D^{3}}\right)  }\\
\times_{M\otimes\mathcal{W}_{D^{3}\{(1,2)\}}}\\
\underset{123}{\left(  M\otimes\mathcal{W}_{D^{3}}\right)  }%
\end{array}
\right) \\
\times_{M\otimes\mathcal{W}_{D(2)}}\\
\left(
\begin{array}
[c]{c}%
\underset{321}{\left(  M\otimes\mathcal{W}_{D^{3}}\right)  }\\
\times_{M\otimes\mathcal{W}_{D^{3}\{(1,2)\}}}\\
\underset{312}{\left(  M\otimes\mathcal{W}_{D^{3}}\right)  }%
\end{array}
\right)
\end{array}
\end{array}
\right] \\
& \rightarrow M\otimes\mathcal{W}_{D}%
\end{align*}
is equivalent to the composition of
\[
\mathrm{id}_{M}\otimes k_{1}:M\otimes\mathcal{W}_{G}\rightarrow M\otimes
\mathcal{W}_{E\left[  1\right]  }%
\]
\[
\mathrm{id}_{M}\otimes\mathcal{W}_{(d_{1},d_{2},d_{3})\in D^{3}%
\{(1,3),(2,3)\}\mapsto(d_{1},0,0,0,0,d_{2},d_{3})\in E[1]}:M\otimes
\mathcal{W}_{E\left[  1\right]  }\rightarrow M\otimes\mathcal{W}%
_{D^{3}\{(1,3),(2,3)\}}%
\]
\[
\mathrm{id}_{M}\otimes\mathcal{W}_{d\in D\mapsto(0,0,d)\in D^{3}%
\{(1,3),(2,3)\}}:M\otimes\mathcal{W}_{D^{3}\{(1,3),(2,3)\}}\rightarrow
M\otimes\mathcal{W}_{D}%
\]
in succession, which results in
\[
\mathrm{id}_{M}\otimes\mathcal{W}_{d\in D\mapsto(0,0,0,0,0,0,d,0)\in
G}:M\otimes\mathcal{W}_{G}\rightarrow M\otimes\mathcal{W}_{D}%
\]

\item The morphism
\begin{align*}
& \zeta^{(\ast_{231}\overset{\cdot}{\underset{2}{-}}\ast_{213})\overset{\cdot
}{-}(\ast_{312}\overset{\cdot}{\underset{2}{-}}\ast_{132})}:\\
& :\left[
\begin{array}
[c]{ccc}
&
\begin{array}
[c]{c}%
\left(
\begin{array}
[c]{c}%
\underset{321}{\left(  M\otimes\mathcal{W}_{D^{3}}\right)  }\\
\times_{M\otimes\mathcal{W}_{D^{3}\{(2,3)\}}}\\
\underset{231}{\left(  M\otimes\mathcal{W}_{D^{3}}\right)  }%
\end{array}
\right) \\
\times_{M\otimes\mathcal{W}_{D(2)}}\\
\left(
\begin{array}
[c]{c}%
\underset{132}{\left(  M\otimes\mathcal{W}_{D^{3}}\right)  }\\
\times_{M\otimes\mathcal{W}_{D^{3}\{(2,3)\}}}\\
\underset{123}{\left(  M\otimes\mathcal{W}_{D^{3}}\right)  }%
\end{array}
\right)
\end{array}
& \\%
\begin{array}
[c]{c}%
\times_{M\otimes\mathcal{W}_{D^{3}\oplus D^{3}}}\\
\,\\
\,
\end{array}
&  &
\begin{array}
[c]{c}%
\times_{M\otimes\mathcal{W}_{D^{3}\oplus D^{3}}}\\
\,\\
\,
\end{array}
\\%
\begin{array}
[c]{c}%
\left(
\begin{array}
[c]{c}%
\underset{132}{\left(  M\otimes\mathcal{W}_{D^{3}}\right)  }\\
\times_{M\otimes\mathcal{W}_{D^{3}\{(1,3)\}}}\\
\underset{312}{\left(  M\otimes\mathcal{W}_{D^{3}}\right)  }%
\end{array}
\right) \\
\times_{M\otimes\mathcal{W}_{D(2)}}\\
\left(
\begin{array}
[c]{c}%
\underset{213}{\left(  M\otimes\mathcal{W}_{D^{3}}\right)  }\\
\times_{M\otimes\mathcal{W}_{D^{3}\{(1,3)\}}}\\
\underset{231}{\left(  M\otimes\mathcal{W}_{D^{3}}\right)  }%
\end{array}
\right)
\end{array}
& \times_{M\otimes\mathcal{W}_{D^{3}\oplus D^{3}}} &
\begin{array}
[c]{c}%
\left(
\begin{array}
[c]{c}%
\underset{213}{\left(  M\otimes\mathcal{W}_{D^{3}}\right)  }\\
\times_{M\otimes\mathcal{W}_{D^{3}\{(1,2)\}}}\\
\underset{123}{\left(  M\otimes\mathcal{W}_{D^{3}}\right)  }%
\end{array}
\right) \\
\times_{M\otimes\mathcal{W}_{D(2)}}\\
\left(
\begin{array}
[c]{c}%
\underset{321}{\left(  M\otimes\mathcal{W}_{D^{3}}\right)  }\\
\times_{M\otimes\mathcal{W}_{D^{3}\{(1,2)\}}}\\
\underset{312}{\left(  M\otimes\mathcal{W}_{D^{3}}\right)  }%
\end{array}
\right)
\end{array}
\end{array}
\right] \\
& \rightarrow M\otimes\mathcal{W}_{D}%
\end{align*}
is equivalent to the composition of
\[
\mathrm{id}_{M}\otimes k_{2}:M\otimes\mathcal{W}_{G}\rightarrow M\otimes
\mathcal{W}_{E\left[  2\right]  }%
\]
\[
\mathrm{id}_{M}\otimes\mathcal{W}_{(d_{1},d_{2},d_{3})\in D^{3}%
\{(1,3),(2,3)\}\mapsto(0,d_{1},0,0,0,d_{2},d_{3})\in E[2]}:M\otimes
\mathcal{W}_{E\left[  2\right]  }\rightarrow M\otimes\mathcal{W}%
_{D^{3}\{(1,3),(2,3)\}}%
\]
\[
\mathrm{id}_{M}\otimes\mathcal{W}_{d\in D\mapsto(0,0,d)\in D^{3}%
\{(1,3),(2,3)\}}:M\otimes\mathcal{W}_{D^{3}\{(1,3),(2,3)\}}\rightarrow
M\otimes\mathcal{W}_{D}%
\]
in succession, which results in
\[
\mathrm{id}_{M}\otimes\mathcal{W}_{d\in D\mapsto(0,0,0,0,0,0,0,d)\in
G}:M\otimes\mathcal{W}_{G}\rightarrow M\otimes\mathcal{W}_{D}%
\]

\item The morphism
\begin{align*}
& \zeta^{(\ast_{312}\overset{\cdot}{\underset{3}{-}}\ast_{321})\overset{\cdot
}{-}(\ast_{123}\overset{\cdot}{\underset{3}{-}}\ast_{213})}:\\
& :\left[
\begin{array}
[c]{ccc}
&
\begin{array}
[c]{c}%
\left(
\begin{array}
[c]{c}%
\underset{321}{\left(  M\otimes\mathcal{W}_{D^{3}}\right)  }\\
\times_{M\otimes\mathcal{W}_{D^{3}\{(2,3)\}}}\\
\underset{231}{\left(  M\otimes\mathcal{W}_{D^{3}}\right)  }%
\end{array}
\right) \\
\times_{M\otimes\mathcal{W}_{D(2)}}\\
\left(
\begin{array}
[c]{c}%
\underset{132}{\left(  M\otimes\mathcal{W}_{D^{3}}\right)  }\\
\times_{M\otimes\mathcal{W}_{D^{3}\{(2,3)\}}}\\
\underset{123}{\left(  M\otimes\mathcal{W}_{D^{3}}\right)  }%
\end{array}
\right)
\end{array}
& \\%
\begin{array}
[c]{c}%
\times_{M\otimes\mathcal{W}_{D^{3}\oplus D^{3}}}\\
\,\\
\,
\end{array}
&  &
\begin{array}
[c]{c}%
\times_{M\otimes\mathcal{W}_{D^{3}\oplus D^{3}}}\\
\,\\
\,
\end{array}
\\%
\begin{array}
[c]{c}%
\left(
\begin{array}
[c]{c}%
\underset{132}{\left(  M\otimes\mathcal{W}_{D^{3}}\right)  }\\
\times_{M\otimes\mathcal{W}_{D^{3}\{(1,3)\}}}\\
\underset{312}{\left(  M\otimes\mathcal{W}_{D^{3}}\right)  }%
\end{array}
\right) \\
\times_{M\otimes\mathcal{W}_{D(2)}}\\
\left(
\begin{array}
[c]{c}%
\underset{213}{\left(  M\otimes\mathcal{W}_{D^{3}}\right)  }\\
\times_{M\otimes\mathcal{W}_{D^{3}\{(1,3)\}}}\\
\underset{231}{\left(  M\otimes\mathcal{W}_{D^{3}}\right)  }%
\end{array}
\right)
\end{array}
& \times_{M\otimes\mathcal{W}_{D^{3}\oplus D^{3}}} &
\begin{array}
[c]{c}%
\left(
\begin{array}
[c]{c}%
\underset{213}{\left(  M\otimes\mathcal{W}_{D^{3}}\right)  }\\
\times_{M\otimes\mathcal{W}_{D^{3}\{(1,2)\}}}\\
\underset{123}{\left(  M\otimes\mathcal{W}_{D^{3}}\right)  }%
\end{array}
\right) \\
\times_{M\otimes\mathcal{W}_{D(2)}}\\
\left(
\begin{array}
[c]{c}%
\underset{321}{\left(  M\otimes\mathcal{W}_{D^{3}}\right)  }\\
\times_{M\otimes\mathcal{W}_{D^{3}\{(1,2)\}}}\\
\underset{312}{\left(  M\otimes\mathcal{W}_{D^{3}}\right)  }%
\end{array}
\right)
\end{array}
\end{array}
\right] \\
& \rightarrow M\otimes\mathcal{W}_{D}%
\end{align*}
is equivalent to the composition of
\[
\mathrm{id}_{M}\otimes k_{3}:M\otimes\mathcal{W}_{G}\rightarrow M\otimes
\mathcal{W}_{E\left[  3\right]  }%
\]
\[
\mathrm{id}_{M}\otimes\mathcal{W}_{(d_{1},d_{2},d_{3})\in D^{3}%
\{(1,3),(2,3)\}\mapsto(0,0,d_{1},0,0,d_{2},d_{3})\in E[3]}:M\otimes
\mathcal{W}_{E\left[  3\right]  }\rightarrow M\otimes\mathcal{W}%
_{D^{3}\{(1,3),(2,3)\}}%
\]
\[
\mathrm{id}_{M}\otimes\mathcal{W}_{d\in D\mapsto(0,0,d)\in D^{3}%
\{(1,3),(2,3)\}}:M\otimes\mathcal{W}_{D^{3}\{(1,3),(2,3)\}}\rightarrow
M\otimes\mathcal{W}_{D}%
\]
in succession, which results in
\[
\mathrm{id}_{M}\otimes\mathcal{W}_{d\in D\mapsto(0,0,0,0,0,0,-d,-d)\in
G}:M\otimes\mathcal{W}_{G}\rightarrow M\otimes\mathcal{W}_{D}%
\]

\item Therefore
\begin{align*}
& \zeta^{(\ast_{123}\overset{\cdot}{\underset{1}{-}}\ast_{132})\overset{\cdot
}{-}(\ast_{231}\overset{\cdot}{\underset{1}{-}}\ast_{321})}+\zeta^{(\ast
_{231}\overset{\cdot}{\underset{2}{-}}\ast_{213})\overset{\cdot}{-}(\ast
_{312}\overset{\cdot}{\underset{2}{-}}\ast_{132})}+\zeta^{(\ast_{312}%
\overset{\cdot}{\underset{3}{-}}\ast_{321})\overset{\cdot}{-}(\ast
_{123}\overset{\cdot}{\underset{3}{-}}\ast_{213})}:\\
& :\left[
\begin{array}
[c]{ccc}
&
\begin{array}
[c]{c}%
\left(
\begin{array}
[c]{c}%
\underset{321}{\left(  M\otimes\mathcal{W}_{D^{3}}\right)  }\\
\times_{M\otimes\mathcal{W}_{D^{3}\{(2,3)\}}}\\
\underset{231}{\left(  M\otimes\mathcal{W}_{D^{3}}\right)  }%
\end{array}
\right) \\
\times_{M\otimes\mathcal{W}_{D(2)}}\\
\left(
\begin{array}
[c]{c}%
\underset{132}{\left(  M\otimes\mathcal{W}_{D^{3}}\right)  }\\
\times_{M\otimes\mathcal{W}_{D^{3}\{(2,3)\}}}\\
\underset{123}{\left(  M\otimes\mathcal{W}_{D^{3}}\right)  }%
\end{array}
\right)
\end{array}
& \\%
\begin{array}
[c]{c}%
\times_{M\otimes\mathcal{W}_{D^{3}\oplus D^{3}}}\\
\,\\
\,
\end{array}
&  &
\begin{array}
[c]{c}%
\times_{M\otimes\mathcal{W}_{D^{3}\oplus D^{3}}}\\
\,\\
\,
\end{array}
\\%
\begin{array}
[c]{c}%
\left(
\begin{array}
[c]{c}%
\underset{132}{\left(  M\otimes\mathcal{W}_{D^{3}}\right)  }\\
\times_{M\otimes\mathcal{W}_{D^{3}\{(1,3)\}}}\\
\underset{312}{\left(  M\otimes\mathcal{W}_{D^{3}}\right)  }%
\end{array}
\right) \\
\times_{M\otimes\mathcal{W}_{D(2)}}\\
\left(
\begin{array}
[c]{c}%
\underset{213}{\left(  M\otimes\mathcal{W}_{D^{3}}\right)  }\\
\times_{M\otimes\mathcal{W}_{D^{3}\{(1,3)\}}}\\
\underset{231}{\left(  M\otimes\mathcal{W}_{D^{3}}\right)  }%
\end{array}
\right)
\end{array}
& \times_{M\otimes\mathcal{W}_{D^{3}\oplus D^{3}}} &
\begin{array}
[c]{c}%
\left(
\begin{array}
[c]{c}%
\underset{213}{\left(  M\otimes\mathcal{W}_{D^{3}}\right)  }\\
\times_{M\otimes\mathcal{W}_{D^{3}\{(1,2)\}}}\\
\underset{123}{\left(  M\otimes\mathcal{W}_{D^{3}}\right)  }%
\end{array}
\right) \\
\times_{M\otimes\mathcal{W}_{D(2)}}\\
\left(
\begin{array}
[c]{c}%
\underset{321}{\left(  M\otimes\mathcal{W}_{D^{3}}\right)  }\\
\times_{M\otimes\mathcal{W}_{D^{3}\{(1,2)\}}}\\
\underset{312}{\left(  M\otimes\mathcal{W}_{D^{3}}\right)  }%
\end{array}
\right)
\end{array}
\end{array}
\right] \\
& \rightarrow M\otimes\mathcal{W}_{D}%
\end{align*}
is equivalent to
\begin{align*}
& \left(  \mathrm{id}_{M}\otimes\mathcal{W}_{d\in D\mapsto\left(
d,d,d\right)  \in D\left(  3\right)  }\right)  \circ(\mathrm{id}_{M}%
\otimes\mathcal{W}_{(d_{1},d_{2},d_{3})\in D(3)\mapsto(0,0,0,0,0,0,d_{1}%
-d_{3},d_{2}-d_{3})\in G})\\
& =\mathrm{id}_{M}\otimes\left(  \mathcal{W}_{d\in D\mapsto\left(
d,d,d\right)  \in D\left(  3\right)  }\circ\mathcal{W}_{(d_{1},d_{2},d_{3})\in
D(3)\mapsto(0,0,0,0,0,0,d_{1}-d_{3},d_{2}-d_{3})\in G}\right) \\
& =\mathrm{id}_{M}\otimes\mathcal{W}_{d\in D\mapsto(0,0,0,0,0,0,0,0)\in G}%
\end{align*}
This completes the proof.
\end{enumerate}
\end{proof}

\section{\label{s3.3}From the General Jacobi Identity to the Jacobi Identity}

\begin{notation}
We write
\[
\left(  M^{M}\otimes\mathcal{W}_{D}\right)  _{\mathrm{id}_{M}}%
\]
for the pullback of
\[%
\begin{array}
[c]{ccc}%
\left(  M^{M}\otimes\mathcal{W}_{D}\right)  _{\mathrm{id}_{M}} & \rightarrow &
M^{M}\otimes\mathcal{W}_{D}\\
\downarrow &  & \downarrow\\
1 & \rightarrow & M^{M}%
\end{array}
\]
where the right arrow $M^{M}\otimes\mathcal{W}_{D}\rightarrow M^{M}$ is the
canonical projection, while the bottom arrow is the exponential transpose of
$\mathrm{id}_{M}:1\times M=M\rightarrow M$.
\end{notation}

\begin{theorem}
\label{t3.3.1}The composition of morphisms
\begin{equation}
\left(  \underset{1}{M^{M}\otimes\mathcal{W}_{D}}\right)  _{\mathrm{id}_{M}%
}\times\left(  \underset{2}{M^{M}\otimes\mathcal{W}_{D}}\right)
_{\mathrm{id}_{M}}\,\underrightarrow{\mathrm{Ass}_{M}^{1,1}}\,\left(
M^{M}\otimes\mathcal{W}_{D^{2}}\right)  _{\mathrm{id}_{M}}\label{3.3.1.1}%
\end{equation}
\[
M^{M}\otimes\mathcal{W}_{D^{2}}\rightarrow M^{M}\otimes\mathcal{W}_{D\left(
2\right)  }%
\]
in succession is equivalent to the composition of morphisms
\begin{align}
& \left(  \underset{1}{M^{M}\otimes\mathcal{W}_{D}}\right)  _{\mathrm{id}_{M}%
}\times\left(  \underset{2}{M^{M}\otimes\mathcal{W}_{D}}\right)
_{\mathrm{id}_{M}}\nonumber\\
& \rightarrow\left(  \underset{2}{M^{M}\otimes\mathcal{W}_{D}}\right)
_{\mathrm{id}_{M}}\times\left(  \underset{1}{M^{M}\otimes\mathcal{W}_{D}%
}\right)  _{\mathrm{id}_{M}}\nonumber\\
& \,\underrightarrow{\mathrm{Ass}_{M}^{1,1}}\,\left(  M^{M}\otimes
\mathcal{W}_{D^{2}}\right)  _{\mathrm{id}_{M}}\nonumber\\
& \underrightarrow{\mathrm{id}_{M^{M}}\otimes\mathcal{W}_{(d_{1},d_{2})\in
D^{2}\mapsto(d_{2},d_{1})\in D^{2}}}\,\left(  M^{M}\otimes\mathcal{W}_{D^{2}%
}\right)  _{\mathrm{id}_{M}}\label{3.3.1.2}%
\end{align}
\[
\left(  M^{M}\otimes\mathcal{W}_{D^{2}}\right)  _{\mathrm{id}_{M}}%
\rightarrow\left(  M^{M}\otimes\mathcal{W}_{D\left(  2\right)  }\right)
_{\mathrm{id}_{M}}%
\]
in succession, so that we have
\begin{align}
& \left(  \underset{1}{M^{M}\otimes\mathcal{W}_{D}}\right)  _{\mathrm{id}_{M}%
}\times\left(  \underset{2}{M^{M}\otimes\mathcal{W}_{D}}\right)
_{\mathrm{id}_{M}}\,\underrightarrow{\text{((\ref{3.3.1.2}),(\ref{3.3.1.1}))}%
}\nonumber\\
& \left(  M^{M}\otimes\mathcal{W}_{D^{2}}\right)  _{\mathrm{id}_{M}}%
\times_{M^{M}\otimes\mathcal{W}_{D\left(  2\right)  }}\left(  M^{M}%
\otimes\mathcal{W}_{D^{2}}\right)  _{\mathrm{id}_{M}}\,\underrightarrow
{\zeta^{\overset{\cdot}{-}}}\,\left(  M^{M}\otimes\mathcal{W}_{D}\right)
_{\mathrm{id}_{M}}\label{3.3.1.3}%
\end{align}
which is equivalent to
\begin{equation}
L_{M}:\left(  \underset{1}{M^{M}\otimes\mathcal{W}_{D}}\right)  _{\mathrm{id}%
_{M}}\times\left(  \underset{2}{M^{M}\otimes\mathcal{W}_{D}}\right)
_{\mathrm{id}_{M}}\rightarrow\left(  M^{M}\otimes\mathcal{W}_{D}\right)
_{\mathrm{id}_{M}}\label{3.3.1.4}%
\end{equation}

\end{theorem}

\begin{proof}
The nontrivial part of the statement is only the equivalence of (\ref{3.3.1.3}%
) and (\ref{3.3.1.4}), for which it is easy to modify the proof of Proposition
8 in \S 3.4 of \cite{lav}.
\end{proof}

The following proposition should be obvious.

\begin{proposition}
\label{t3.3.2}We have the following two statements.

\begin{enumerate}
\item The composition of morphisms
\begin{align}
& \left(  \left(  \underset{1}{M^{M}\otimes\mathcal{W}_{D^{2}}}\right)
_{\mathrm{id}_{M}}\times_{M^{M}\otimes\mathcal{W}_{D\left(  2\right)  }%
}\left(  \underset{2}{M^{M}\otimes\mathcal{W}_{D^{2}}}\right)  _{\mathrm{id}%
_{M}}\right)  \times\left(  M^{M}\otimes\mathcal{W}_{D}\right)  _{\mathrm{id}%
_{M}}\nonumber\\
& \rightarrow\left(  \underset{1}{M^{M}\otimes\mathcal{W}_{D^{2}}}\right)
_{\mathrm{id}_{M}}\times\left(  M^{M}\otimes\mathcal{W}_{D}\right)
_{\mathrm{id}_{M}}\,\underrightarrow{\mathrm{Ass}_{M}^{2,1}}\,\left(
M^{M}\otimes\mathcal{W}_{D^{3}}\right)  _{\mathrm{id}_{M}}\label{3.3.2.1}%
\end{align}
\[
\left(  M^{M}\otimes\mathcal{W}_{D^{3}}\right)  _{\mathrm{id}_{M}}%
\rightarrow\left(  M^{M}\otimes\mathcal{W}_{D^{3}\{(1,2)\}}\right)
_{\mathrm{id}_{M}}%
\]
in succession is equivalent to the composition of morphisms
\begin{align}
& \left(  \left(  \underset{1}{M^{M}\otimes\mathcal{W}_{D^{2}}}\right)
_{\mathrm{id}_{M}}\times_{M^{M}\otimes\mathcal{W}_{D\left(  2\right)  }%
}\left(  \underset{2}{M^{M}\otimes\mathcal{W}_{D^{2}}}\right)  _{\mathrm{id}%
_{M}}\right)  \times\left(  M^{M}\otimes\mathcal{W}_{D}\right)  _{\mathrm{id}%
_{M}}\nonumber\\
& \rightarrow\left(  \underset{2}{M^{M}\otimes\mathcal{W}_{D^{2}}}\right)
_{\mathrm{id}_{M}}\times\left(  M^{M}\otimes\mathcal{W}_{D}\right)
_{\mathrm{id}_{M}}\,\underrightarrow{\mathrm{Ass}_{M}^{2,1}}\,\left(
M^{M}\otimes\mathcal{W}_{D^{3}}\right)  _{\mathrm{id}_{M}}\label{3.3.2.2}%
\end{align}
\[
\left(  M^{M}\otimes\mathcal{W}_{D^{3}}\right)  _{\mathrm{id}_{M}}%
\rightarrow\left(  M^{M}\otimes\mathcal{W}_{D^{3}\{(1,2)\}}\right)
_{\mathrm{id}_{M}}%
\]
in succession, so that we have the morphism
\begin{align*}
& \left(  \left(  \underset{1}{M^{M}\otimes\mathcal{W}_{D^{2}}}\right)
_{\mathrm{id}_{M}}\times_{M^{M}\otimes\mathcal{W}_{D\left(  2\right)  }%
}\left(  \underset{2}{M^{M}\otimes\mathcal{W}_{D^{2}}}\right)  _{\mathrm{id}%
_{M}}\right)  \times\left(  M^{M}\otimes\mathcal{W}_{D}\right)  _{\mathrm{id}%
_{M}}\\
& \,\underrightarrow{\text{((\ref{3.3.2.1}),(\ref{3.3.2.2}))}}\,\left(
M^{M}\otimes\mathcal{W}_{D^{3}}\right)  _{\mathrm{id}_{M}}\times_{M^{M}%
\otimes\mathcal{W}_{D^{3}\{(1,2)\}}}\left(  M^{M}\otimes\mathcal{W}_{D^{3}%
}\right)  _{\mathrm{id}_{M}}\\
& \underrightarrow{\zeta^{\underset{3}{\overset{\cdot}{-}}}}\,\left(
M^{M}\otimes\mathcal{W}_{D^{2}}\right)  _{\mathrm{id}_{M}}%
\end{align*}
which is equivalent to the morphism
\begin{align*}
& \left(  \left(  \underset{1}{M^{M}\otimes\mathcal{W}_{D^{2}}}\right)
_{\mathrm{id}_{M}}\times_{M^{M}\otimes\mathcal{W}_{D\left(  2\right)  }%
}\left(  \underset{2}{M^{M}\otimes\mathcal{W}_{D^{2}}}\right)  _{\mathrm{id}%
_{M}}\right)  \times\left(  M^{M}\otimes\mathcal{W}_{D}\right)  _{\mathrm{id}%
_{M}}\\
& \underrightarrow{\zeta^{\overset{\cdot}{-}}\times\mathrm{id}_{M^{M}%
\otimes\mathcal{W}_{D}}}\,\left(  M^{M}\otimes\mathcal{W}_{D}\right)
_{\mathrm{id}_{M}}\times\left(  M^{M}\otimes\mathcal{W}_{D}\right)
_{\mathrm{id}_{M}}\,\underrightarrow{\mathrm{Ass}_{M}^{1,1}}\,\left(
M^{M}\otimes\mathcal{W}_{D^{2}}\right)  _{\mathrm{id}_{M}}%
\end{align*}

\item The composition of morphisms
\begin{align}
& \left(  M^{M}\otimes\mathcal{W}_{D}\right)  _{\mathrm{id}_{M}}\times\left(
\left(  \underset{1}{M^{M}\otimes\mathcal{W}_{D^{2}}}\right)  _{\mathrm{id}%
_{M}}\times_{M^{M}\otimes\mathcal{W}_{D\left(  2\right)  }}\left(
\underset{2}{M^{M}\otimes\mathcal{W}_{D^{2}}}\right)  _{\mathrm{id}_{M}%
}\right) \nonumber\\
& \rightarrow\left(  M^{M}\otimes\mathcal{W}_{D}\right)  _{\mathrm{id}_{M}%
}\times\left(  \underset{1}{M^{M}\otimes\mathcal{W}_{D^{2}}}\right)
_{\mathrm{id}_{M}}\,\underrightarrow{\mathrm{Ass}_{M}^{1,2}}\,\left(
M^{M}\otimes\mathcal{W}_{D^{3}}\right)  _{\mathrm{id}_{M}}\label{3.3.2.3}%
\end{align}
\[
\left(  M^{M}\otimes\mathcal{W}_{D^{3}}\right)  _{\mathrm{id}_{M}}%
\rightarrow\left(  M^{M}\otimes\mathcal{W}_{D^{3}\{(2,3)\}}\right)
_{\mathrm{id}_{M}}%
\]
in succession is equivalent to the composition of morphisms
\begin{align}
& \left(  M^{M}\otimes\mathcal{W}_{D}\right)  _{\mathrm{id}_{M}}\times\left(
\left(  \underset{1}{M^{M}\otimes\mathcal{W}_{D^{2}}}\right)  _{\mathrm{id}%
_{M}}\times_{M^{M}\otimes\mathcal{W}_{D\left(  2\right)  }}\left(
\underset{2}{M^{M}\otimes\mathcal{W}_{D^{2}}}\right)  _{\mathrm{id}_{M}%
}\right) \nonumber\\
& \rightarrow\left(  M^{M}\otimes\mathcal{W}_{D}\right)  _{\mathrm{id}_{M}%
}\times\left(  \underset{2}{M^{M}\otimes\mathcal{W}_{D^{2}}}\right)
_{\mathrm{id}_{M}}\,\underrightarrow{\mathrm{Ass}_{M}^{1,2}}\,\left(
M^{M}\otimes\mathcal{W}_{D^{3}}\right)  _{\mathrm{id}_{M}}\label{3.3.2.4}%
\end{align}
\[
\left(  M^{M}\otimes\mathcal{W}_{D^{3}}\right)  _{\mathrm{id}_{M}}%
\rightarrow\left(  M^{M}\otimes\mathcal{W}_{D^{3}\{(2,3)\}}\right)
_{\mathrm{id}_{M}}%
\]
in succession, so that we have the morphism
\begin{align*}
& \left(  M^{M}\otimes\mathcal{W}_{D}\right)  _{\mathrm{id}_{M}}\times\left(
\left(  \underset{1}{M^{M}\otimes\mathcal{W}_{D^{2}}}\right)  _{\mathrm{id}%
_{M}}\times_{M^{M}\otimes\mathcal{W}_{D\left(  2\right)  }}\left(
\underset{2}{M^{M}\otimes\mathcal{W}_{D^{2}}}\right)  _{\mathrm{id}_{M}%
}\right) \\
& \,\underrightarrow{\text{((\ref{3.3.2.3}),(\ref{3.3.2.4}))}}\,\left(
M^{M}\otimes\mathcal{W}_{D^{3}}\right)  _{\mathrm{id}_{M}}\times_{M^{M}%
\otimes\mathcal{W}_{D^{3}\{(2,3)\}}}\left(  M^{M}\otimes\mathcal{W}_{D^{3}%
}\right)  _{\mathrm{id}_{M}}\\
& \underrightarrow{\zeta^{\underset{1}{\overset{\cdot}{-}}}}\,\left(
M^{M}\otimes\mathcal{W}_{D^{2}}\right)  _{\mathrm{id}_{M}}%
\end{align*}
which is equivalent to the morphism
\begin{align*}
& \left(  M^{M}\otimes\mathcal{W}_{D}\right)  _{\mathrm{id}_{M}}\times\left(
\left(  \underset{1}{M^{M}\otimes\mathcal{W}_{D^{2}}}\right)  _{\mathrm{id}%
_{M}}\times_{M^{M}\otimes\mathcal{W}_{D\left(  2\right)  }}\left(
\underset{2}{M^{M}\otimes\mathcal{W}_{D^{2}}}\right)  _{\mathrm{id}_{M}%
}\right) \\
& \underrightarrow{\zeta^{\overset{\cdot}{-}}\times\mathrm{id}_{M^{M}%
\otimes\mathcal{W}_{D}}}\times\,\left(  M^{M}\otimes\mathcal{W}_{D}\right)
_{\mathrm{id}_{M}}\times\left(  M^{M}\otimes\mathcal{W}_{D}\right)
_{\mathrm{id}_{M}}\,\underrightarrow{\mathrm{Ass}_{M}^{1,1}}\,\left(
M^{M}\otimes\mathcal{W}_{D^{2}}\right)  _{\mathrm{id}_{M}}%
\end{align*}

\end{enumerate}
\end{proposition}

\begin{notation}
We introduce the following fifteen morphisms:

\begin{enumerate}
\item
\begin{align*}
\chi^{\ast_{1}}  & :\left(  \underset{1}{M^{M}\otimes\mathcal{W}_{D}}\right)
_{\mathrm{id}_{M}}\times\left(  \underset{2}{M^{M}\otimes\mathcal{W}_{D}%
}\right)  _{\mathrm{id}_{M}}\times\left(  \underset{3}{M^{M}\otimes
\mathcal{W}_{D}}\right)  _{\mathrm{id}_{M}}\rightarrow\\
& \left(  \underset{1}{M^{M}\otimes\mathcal{W}_{D}}\right)  _{\mathrm{id}_{M}}%
\end{align*}
as the canonical projection.

\item
\begin{align*}
\chi^{\ast_{2}}  & :\left(  \underset{1}{M^{M}\otimes\mathcal{W}_{D}}\right)
_{\mathrm{id}_{M}}\times\left(  \underset{2}{M^{M}\otimes\mathcal{W}_{D}%
}\right)  _{\mathrm{id}_{M}}\times\left(  \underset{3}{M^{M}\otimes
\mathcal{W}_{D}}\right)  _{\mathrm{id}_{M}}\rightarrow\\
& \left(  \underset{2}{M^{M}\otimes\mathcal{W}_{D}}\right)  _{\mathrm{id}_{M}}%
\end{align*}
as the canonical projection.

\item
\begin{align*}
\chi^{\ast_{3}}  & :\left(  \underset{1}{M^{M}\otimes\mathcal{W}_{D}}\right)
_{\mathrm{id}_{M}}\times\left(  \underset{2}{M^{M}\otimes\mathcal{W}_{D}%
}\right)  _{\mathrm{id}_{M}}\times\left(  \underset{3}{M^{M}\otimes
\mathcal{W}_{D}}\right)  _{\mathrm{id}_{M}}\rightarrow\\
& \left(  \underset{3}{M^{M}\otimes\mathcal{W}_{D}}\right)  _{\mathrm{id}_{M}}%
\end{align*}
as the canonical projection.

\item
\begin{align*}
\chi^{\ast_{12}}  & :\left(  \underset{1}{M^{M}\otimes\mathcal{W}_{D}}\right)
_{\mathrm{id}_{M}}\times\left(  \underset{2}{M^{M}\otimes\mathcal{W}_{D}%
}\right)  _{\mathrm{id}_{M}}\times\left(  \underset{3}{M^{M}\otimes
\mathcal{W}_{D}}\right)  _{\mathrm{id}_{M}}\\
& \rightarrow\left(  M^{M}\otimes\mathcal{W}_{D^{2}}\right)  _{\mathrm{id}%
_{M}}%
\end{align*}
as
\begin{align*}
& \left(  \underset{1}{M^{M}\otimes\mathcal{W}_{D}}\right)  _{\mathrm{id}_{M}%
}\times\left(  \underset{2}{M^{M}\otimes\mathcal{W}_{D}}\right)
_{\mathrm{id}_{M}}\times\left(  \underset{3}{M^{M}\otimes\mathcal{W}_{D}%
}\right)  _{\mathrm{id}_{M}}\\
& \rightarrow\left(  \underset{1}{M^{M}\otimes\mathcal{W}_{D}}\right)
_{\mathrm{id}_{M}}\times\left(  \underset{2}{M^{M}\otimes\mathcal{W}_{D}%
}\right)  _{\mathrm{id}_{M}}\,\underrightarrow{\mathrm{Ass}_{M}^{1,1}%
}\,\left(  M^{M}\otimes\mathcal{W}_{D^{2}}\right)  _{\mathrm{id}_{M}}%
\end{align*}

\item
\begin{align*}
\chi^{\ast_{21}}  & :\left(  \underset{1}{M^{M}\otimes\mathcal{W}_{D}}\right)
_{\mathrm{id}_{M}}\times\left(  \underset{2}{M^{M}\otimes\mathcal{W}_{D}%
}\right)  _{\mathrm{id}_{M}}\times\left(  \underset{3}{M^{M}\otimes
\mathcal{W}_{D}}\right)  _{\mathrm{id}_{M}}\\
& \rightarrow\left(  M^{M}\otimes\mathcal{W}_{D^{2}}\right)  _{\mathrm{id}%
_{M}}%
\end{align*}
as
\begin{align*}
& \left(  \underset{1}{M^{M}\otimes\mathcal{W}_{D}}\right)  _{\mathrm{id}_{M}%
}\times\left(  \underset{2}{M^{M}\otimes\mathcal{W}_{D}}\right)
_{\mathrm{id}_{M}}\times\left(  \underset{3}{M^{M}\otimes\mathcal{W}_{D}%
}\right)  _{\mathrm{id}_{M}}\\
& \rightarrow\left(  \underset{2}{M^{M}\otimes\mathcal{W}_{D}}\right)
_{\mathrm{id}_{M}}\times\left(  \underset{1}{M^{M}\otimes\mathcal{W}_{D}%
}\right)  _{\mathrm{id}_{M}}\,\underrightarrow{\mathrm{Ass}_{M}^{1,1}%
}\,\left(  M^{M}\otimes\mathcal{W}_{D^{2}}\right)  _{\mathrm{id}_{M}}\\
& \underrightarrow{\mathrm{id}_{M^{M}}\otimes\mathcal{W}_{\left(  d_{1}%
,d_{2}\right)  \in D^{2}\mapsto\left(  d_{2},d_{1}\right)  \in D^{2}}%
}\,\left(  M^{M}\otimes\mathcal{W}_{D^{2}}\right)  _{\mathrm{id}_{M}}%
\end{align*}

\item
\begin{align*}
\chi^{\ast_{13}}  & :\left(  \underset{1}{M^{M}\otimes\mathcal{W}_{D}}\right)
_{\mathrm{id}_{M}}\times\left(  \underset{2}{M^{M}\otimes\mathcal{W}_{D}%
}\right)  _{\mathrm{id}_{M}}\times\left(  \underset{3}{M^{M}\otimes
\mathcal{W}_{D}}\right)  _{\mathrm{id}_{M}}\\
& \rightarrow\left(  M^{M}\otimes\mathcal{W}_{D^{2}}\right)  _{\mathrm{id}%
_{M}}%
\end{align*}
as
\begin{align*}
& \left(  \underset{1}{M^{M}\otimes\mathcal{W}_{D}}\right)  _{\mathrm{id}_{M}%
}\times\left(  \underset{2}{M^{M}\otimes\mathcal{W}_{D}}\right)
_{\mathrm{id}_{M}}\times\left(  \underset{3}{M^{M}\otimes\mathcal{W}_{D}%
}\right)  _{\mathrm{id}_{M}}\\
& \rightarrow\left(  \underset{1}{M^{M}\otimes\mathcal{W}_{D}}\right)
_{\mathrm{id}_{M}}\times\left(  \underset{3}{M^{M}\otimes\mathcal{W}_{D}%
}\right)  _{\mathrm{id}_{M}}\,\underrightarrow{\mathrm{Ass}_{M}^{1,1}%
}\,\left(  M^{M}\otimes\mathcal{W}_{D^{2}}\right)  _{\mathrm{id}_{M}}%
\end{align*}

\item
\begin{align*}
\chi^{\ast_{31}}  & :\left(  \underset{1}{M^{M}\otimes\mathcal{W}_{D}}\right)
_{\mathrm{id}_{M}}\times\left(  \underset{2}{M^{M}\otimes\mathcal{W}_{D}%
}\right)  _{\mathrm{id}_{M}}\times\left(  \underset{3}{M^{M}\otimes
\mathcal{W}_{D}}\right)  _{\mathrm{id}_{M}}\\
& \rightarrow\left(  M^{M}\otimes\mathcal{W}_{D^{2}}\right)  _{\mathrm{id}%
_{M}}%
\end{align*}
as
\begin{align*}
& \left(  \underset{1}{M^{M}\otimes\mathcal{W}_{D}}\right)  _{\mathrm{id}_{M}%
}\times\left(  \underset{2}{M^{M}\otimes\mathcal{W}_{D}}\right)
_{\mathrm{id}_{M}}\times\left(  \underset{3}{M^{M}\otimes\mathcal{W}_{D}%
}\right)  _{\mathrm{id}_{M}}\\
& \rightarrow\left(  \underset{3}{M^{M}\otimes\mathcal{W}_{D}}\right)
_{\mathrm{id}_{M}}\times\left(  \underset{1}{M^{M}\otimes\mathcal{W}_{D}%
}\right)  _{\mathrm{id}_{M}}\,\underrightarrow{\mathrm{Ass}_{M}^{1,1}%
}\,\left(  M^{M}\otimes\mathcal{W}_{D^{2}}\right)  _{\mathrm{id}_{M}}\\
& \underrightarrow{\mathrm{id}_{M^{M}}\otimes\mathcal{W}_{\left(  d_{1}%
,d_{2}\right)  \in D^{2}\mapsto\left(  d_{2},d_{1}\right)  \in D^{2}}%
}\,\left(  M^{M}\otimes\mathcal{W}_{D^{2}}\right)  _{\mathrm{id}_{M}}%
\end{align*}

\item
\begin{align*}
\chi^{\ast_{23}}  & :\left(  \underset{1}{M^{M}\otimes\mathcal{W}_{D}}\right)
_{\mathrm{id}_{M}}\times\left(  \underset{2}{M^{M}\otimes\mathcal{W}_{D}%
}\right)  _{\mathrm{id}_{M}}\times\left(  \underset{3}{M^{M}\otimes
\mathcal{W}_{D}}\right)  _{\mathrm{id}_{M}}\\
& \rightarrow\left(  M^{M}\otimes\mathcal{W}_{D^{2}}\right)  _{\mathrm{id}%
_{M}}%
\end{align*}
as
\begin{align*}
& \left(  \underset{1}{M^{M}\otimes\mathcal{W}_{D}}\right)  _{\mathrm{id}_{M}%
}\times\left(  \underset{2}{M^{M}\otimes\mathcal{W}_{D}}\right)
_{\mathrm{id}_{M}}\times\left(  \underset{3}{M^{M}\otimes\mathcal{W}_{D}%
}\right)  _{\mathrm{id}_{M}}\\
& \rightarrow\left(  \underset{2}{M^{M}\otimes\mathcal{W}_{D}}\right)
_{\mathrm{id}_{M}}\times\left(  \underset{3}{M^{M}\otimes\mathcal{W}_{D}%
}\right)  _{\mathrm{id}_{M}}\,\underrightarrow{\mathrm{Ass}_{M}^{1,1}%
}\,\left(  M^{M}\otimes\mathcal{W}_{D^{2}}\right)  _{\mathrm{id}_{M}}%
\end{align*}

\item
\begin{align*}
\chi^{\ast_{32}}  & :\left(  \underset{1}{M^{M}\otimes\mathcal{W}_{D}}\right)
_{\mathrm{id}_{M}}\times\left(  \underset{2}{M^{M}\otimes\mathcal{W}_{D}%
}\right)  _{\mathrm{id}_{M}}\times\left(  \underset{3}{M^{M}\otimes
\mathcal{W}_{D}}\right)  _{\mathrm{id}_{M}}\\
& \rightarrow\left(  M^{M}\otimes\mathcal{W}_{D^{2}}\right)  _{\mathrm{id}%
_{M}}%
\end{align*}
as
\begin{align*}
& \left(  \underset{1}{M^{M}\otimes\mathcal{W}_{D}}\right)  _{\mathrm{id}_{M}%
}\times\left(  \underset{2}{M^{M}\otimes\mathcal{W}_{D}}\right)
_{\mathrm{id}_{M}}\times\left(  \underset{3}{M^{M}\otimes\mathcal{W}_{D}%
}\right)  _{\mathrm{id}_{M}}\\
& \rightarrow\left(  \underset{3}{M^{M}\otimes\mathcal{W}_{D}}\right)
_{\mathrm{id}_{M}}\times\left(  \underset{2}{M^{M}\otimes\mathcal{W}_{D}%
}\right)  _{\mathrm{id}_{M}}\,\underrightarrow{\mathrm{Ass}_{M}^{1,1}%
}\,\left(  M^{M}\otimes\mathcal{W}_{D^{2}}\right)  _{\mathrm{id}_{M}}\\
& \underrightarrow{\mathrm{id}_{M^{M}}\otimes\mathcal{W}_{\left(  d_{1}%
,d_{2}\right)  \in D^{2}\mapsto\left(  d_{2},d_{1}\right)  \in D^{2}}%
}\,\left(  M^{M}\otimes\mathcal{W}_{D^{2}}\right)  _{\mathrm{id}_{M}}%
\end{align*}

\item
\begin{align*}
\chi^{\ast_{123}}  & :\left(  \underset{1}{M^{M}\otimes\mathcal{W}_{D}%
}\right)  _{\mathrm{id}_{M}}\times\left(  \underset{2}{M^{M}\otimes
\mathcal{W}_{D}}\right)  _{\mathrm{id}_{M}}\times\left(  \underset{3}%
{M^{M}\otimes\mathcal{W}_{D}}\right)  _{\mathrm{id}_{M}}\\
& \rightarrow\left(  M^{M}\otimes\mathcal{W}_{D^{3}}\right)  _{\mathrm{id}%
_{M}}%
\end{align*}
as
\begin{align*}
\mathrm{Ass}_{M}^{1,1,1}  & :\left(  \underset{1}{M^{M}\otimes\mathcal{W}_{D}%
}\right)  _{\mathrm{id}_{M}}\times\left(  \underset{2}{M^{M}\otimes
\mathcal{W}_{D}}\right)  _{\mathrm{id}_{M}}\times\left(  \underset{3}%
{M^{M}\otimes\mathcal{W}_{D}}\right)  _{\mathrm{id}_{M}}\\
& \rightarrow\left(  M^{M}\otimes\mathcal{W}_{D^{3}}\right)  _{\mathrm{id}%
_{M}}%
\end{align*}

\item
\begin{align*}
\chi^{\ast_{132}}  & :\left(  \underset{1}{M^{M}\otimes\mathcal{W}_{D}%
}\right)  _{\mathrm{id}_{M}}\times\left(  \underset{2}{M^{M}\otimes
\mathcal{W}_{D}}\right)  _{\mathrm{id}_{M}}\times\left(  \underset{3}%
{M^{M}\otimes\mathcal{W}_{D}}\right)  _{\mathrm{id}_{M}}\\
& \rightarrow\left(  M^{M}\otimes\mathcal{W}_{D^{3}}\right)  _{\mathrm{id}%
_{M}}%
\end{align*}
as
\begin{align*}
& \left(  \underset{1}{M^{M}\otimes\mathcal{W}_{D}}\right)  _{\mathrm{id}_{M}%
}\times\left(  \underset{2}{M^{M}\otimes\mathcal{W}_{D}}\right)
_{\mathrm{id}_{M}}\times\left(  \underset{3}{M^{M}\otimes\mathcal{W}_{D}%
}\right)  _{\mathrm{id}_{M}}\\
& \rightarrow\left(  \underset{1}{M^{M}\otimes\mathcal{W}_{D}}\right)
_{\mathrm{id}_{M}}\times\left(  \underset{3}{M^{M}\otimes\mathcal{W}_{D}%
}\right)  _{\mathrm{id}_{M}}\times\left(  \underset{2}{M^{M}\otimes
\mathcal{W}_{D}}\right)  _{\mathrm{id}_{M}}\,\underrightarrow{\mathrm{Ass}%
_{M}^{1,1,1}}\\
& \left(  M^{M}\otimes\mathcal{W}_{D^{3}}\right)  _{\mathrm{id}_{M}%
}\,\underrightarrow{\mathrm{id}_{M^{M}}\otimes\mathcal{W}_{(d_{1},d_{2}%
,d_{3})\in D^{3}\mapsto(d_{1},d_{3},d_{2})\in D^{3}}}\,\left(  M^{M}%
\otimes\mathcal{W}_{D^{3}}\right)  _{\mathrm{id}_{M}}%
\end{align*}

\item
\begin{align*}
\chi^{\ast_{213}}  & :\left(  \underset{1}{M^{M}\otimes\mathcal{W}_{D}%
}\right)  _{\mathrm{id}_{M}}\times\left(  \underset{2}{M^{M}\otimes
\mathcal{W}_{D}}\right)  _{\mathrm{id}_{M}}\times\left(  \underset{3}%
{M^{M}\otimes\mathcal{W}_{D}}\right)  _{\mathrm{id}_{M}}\\
& \rightarrow\left(  M^{M}\otimes\mathcal{W}_{D^{3}}\right)  _{\mathrm{id}%
_{M}}%
\end{align*}
as
\begin{align*}
& \left(  \underset{1}{M^{M}\otimes\mathcal{W}_{D}}\right)  _{\mathrm{id}_{M}%
}\times\left(  \underset{2}{M^{M}\otimes\mathcal{W}_{D}}\right)
_{\mathrm{id}_{M}}\times\left(  \underset{3}{M^{M}\otimes\mathcal{W}_{D}%
}\right)  _{\mathrm{id}_{M}}\\
& \rightarrow\left(  \underset{2}{M^{M}\otimes\mathcal{W}_{D}}\right)
_{\mathrm{id}_{M}}\times\left(  \underset{1}{M^{M}\otimes\mathcal{W}_{D}%
}\right)  _{\mathrm{id}_{M}}\times\left(  \underset{3}{M^{M}\otimes
\mathcal{W}_{D}}\right)  _{\mathrm{id}_{M}}\,\underrightarrow{\mathrm{Ass}%
_{M}^{1,1,1}}\\
& \left(  M^{M}\otimes\mathcal{W}_{D^{3}}\right)  _{\mathrm{id}_{M}%
}\,\underrightarrow{\mathrm{id}_{M^{M}}\otimes\mathcal{W}_{(d_{1},d_{2}%
,d_{3})\in D^{3}\mapsto(d_{2},d_{1},d_{3})\in D^{3}}}\,\left(  M^{M}%
\otimes\mathcal{W}_{D^{3}}\right)  _{\mathrm{id}_{M}}%
\end{align*}

\item
\begin{align*}
\chi^{\ast_{231}}  & :\left(  \underset{1}{M^{M}\otimes\mathcal{W}_{D}%
}\right)  _{\mathrm{id}_{M}}\times\left(  \underset{2}{M^{M}\otimes
\mathcal{W}_{D}}\right)  _{\mathrm{id}_{M}}\times\left(  \underset{3}%
{M^{M}\otimes\mathcal{W}_{D}}\right)  _{\mathrm{id}_{M}}\\
& \rightarrow\left(  M^{M}\otimes\mathcal{W}_{D^{3}}\right)  _{\mathrm{id}%
_{M}}%
\end{align*}
as
\begin{align*}
& \left(  \underset{1}{M^{M}\otimes\mathcal{W}_{D}}\right)  _{\mathrm{id}_{M}%
}\times\left(  \underset{2}{M^{M}\otimes\mathcal{W}_{D}}\right)
_{\mathrm{id}_{M}}\times\left(  \underset{3}{M^{M}\otimes\mathcal{W}_{D}%
}\right)  _{\mathrm{id}_{M}}\\
& \rightarrow\left(  \underset{2}{M^{M}\otimes\mathcal{W}_{D}}\right)
_{\mathrm{id}_{M}}\times\left(  \underset{3}{M^{M}\otimes\mathcal{W}_{D}%
}\right)  _{\mathrm{id}_{M}}\times\left(  \underset{1}{M^{M}\otimes
\mathcal{W}_{D}}\right)  _{\mathrm{id}_{M}}\,\underrightarrow{\mathrm{Ass}%
_{M}^{1,1,1}}\\
& \left(  M^{M}\otimes\mathcal{W}_{D^{3}}\right)  _{\mathrm{id}_{M}%
}\,\underrightarrow{\mathrm{id}_{M^{M}}\otimes\mathcal{W}_{(d_{1},d_{2}%
,d_{3})\in D^{3}\mapsto(d_{2},d_{3},d_{1})\in D^{3}}}\,\left(  M^{M}%
\otimes\mathcal{W}_{D^{3}}\right)  _{\mathrm{id}_{M}}%
\end{align*}

\item
\begin{align*}
\chi^{\ast_{312}}  & :\left(  \underset{1}{M^{M}\otimes\mathcal{W}_{D}%
}\right)  _{\mathrm{id}_{M}}\times\left(  \underset{2}{M^{M}\otimes
\mathcal{W}_{D}}\right)  _{\mathrm{id}_{M}}\times\left(  \underset{3}%
{M^{M}\otimes\mathcal{W}_{D}}\right)  _{\mathrm{id}_{M}}\\
& \rightarrow\left(  M^{M}\otimes\mathcal{W}_{D^{3}}\right)  _{\mathrm{id}%
_{M}}%
\end{align*}
as
\begin{align*}
& \left(  \underset{1}{M^{M}\otimes\mathcal{W}_{D}}\right)  _{\mathrm{id}_{M}%
}\times\left(  \underset{2}{M^{M}\otimes\mathcal{W}_{D}}\right)
_{\mathrm{id}_{M}}\times\left(  \underset{3}{M^{M}\otimes\mathcal{W}_{D}%
}\right)  _{\mathrm{id}_{M}}\\
& \rightarrow\left(  \underset{3}{M^{M}\otimes\mathcal{W}_{D}}\right)
_{\mathrm{id}_{M}}\times\left(  \underset{1}{M^{M}\otimes\mathcal{W}_{D}%
}\right)  _{\mathrm{id}_{M}}\times\left(  \underset{2}{M^{M}\otimes
\mathcal{W}_{D}}\right)  _{\mathrm{id}_{M}}\,\underrightarrow{\mathrm{Ass}%
_{M}^{1,1,1}}\\
& \left(  M^{M}\otimes\mathcal{W}_{D^{3}}\right)  _{\mathrm{id}_{M}%
}\,\underrightarrow{\mathrm{id}_{M^{M}}\otimes\mathcal{W}_{(d_{1},d_{2}%
,d_{3})\in D^{3}\mapsto(d_{3},d_{1},d_{2})\in D^{3}}}\,\left(  M^{M}%
\otimes\mathcal{W}_{D^{3}}\right)  _{\mathrm{id}_{M}}%
\end{align*}

\item
\begin{align*}
\chi^{\ast_{321}}  & :\left(  \underset{1}{M^{M}\otimes\mathcal{W}_{D}%
}\right)  _{\mathrm{id}_{M}}\times\left(  \underset{2}{M^{M}\otimes
\mathcal{W}_{D}}\right)  _{\mathrm{id}_{M}}\times\left(  \underset{3}%
{M^{M}\otimes\mathcal{W}_{D}}\right)  _{\mathrm{id}_{M}}\\
& \rightarrow\left(  M^{M}\otimes\mathcal{W}_{D^{3}}\right)  _{\mathrm{id}%
_{M}}%
\end{align*}
as
\begin{align*}
& \left(  \underset{1}{M^{M}\otimes\mathcal{W}_{D}}\right)  _{\mathrm{id}_{M}%
}\times\left(  \underset{2}{M^{M}\otimes\mathcal{W}_{D}}\right)
_{\mathrm{id}_{M}}\times\left(  \underset{3}{M^{M}\otimes\mathcal{W}_{D}%
}\right)  _{\mathrm{id}_{M}}\\
& \rightarrow\left(  \underset{3}{M^{M}\otimes\mathcal{W}_{D}}\right)
_{\mathrm{id}_{M}}\times\left(  \underset{1}{M^{M}\otimes\mathcal{W}_{D}%
}\right)  _{\mathrm{id}_{M}}\times\left(  \underset{2}{M^{M}\otimes
\mathcal{W}_{D}}\right)  _{\mathrm{id}_{M}}\,\underrightarrow{\mathrm{Ass}%
_{M}^{1,1,1}}\\
& \left(  M^{M}\otimes\mathcal{W}_{D^{3}}\right)  _{\mathrm{id}_{M}%
}\,\underrightarrow{\mathrm{id}_{M^{M}}\otimes\mathcal{W}_{(d_{1},d_{2}%
,d_{3})\in D^{3}\mapsto(d_{3},d_{1},d_{2})\in D^{3}}}\,\left(  M^{M}%
\otimes\mathcal{W}_{D^{3}}\right)  _{\mathrm{id}_{M}}%
\end{align*}

\end{enumerate}
\end{notation}

\begin{lemma}
\label{t3.3.3}We have the following statements:

\begin{enumerate}
\item The composition of morphisms $\chi^{\ast_{321}}$ and
\[
\left(  M^{M}\otimes\mathcal{W}_{D^{3}}\right)  _{\mathrm{id}_{M}}%
\rightarrow\left(  M^{M}\otimes\mathcal{W}_{D^{3}\{(2,3)\}}\right)
_{\mathrm{id}_{M}}%
\]
is equivalent to the composition of morphisms $\chi^{\ast_{231}}$ and
\[
\left(  M^{M}\otimes\mathcal{W}_{D^{3}}\right)  _{\mathrm{id}_{M}}%
\rightarrow\left(  M^{M}\otimes\mathcal{W}_{D^{3}\{(2,3)\}}\right)
_{\mathrm{id}_{M}}%
\]
so that we have
\begin{align}
& \left(  \underset{1}{M^{M}\otimes\mathcal{W}_{D}}\right)  _{\mathrm{id}_{M}%
}\times\left(  \underset{2}{M^{M}\otimes\mathcal{W}_{D}}\right)
_{\mathrm{id}_{M}}\times\left(  \underset{3}{M^{M}\otimes\mathcal{W}_{D}%
}\right)  _{\mathrm{id}_{M}}\nonumber\\
& \underrightarrow{\left(  \chi^{\ast_{321}},\chi^{\ast_{231}}\right)
}\,\left(  M^{M}\otimes\mathcal{W}_{D^{3}}\right)  \times_{M^{M}%
\otimes\mathcal{W}_{D^{3}\{(2,3)\}}}\left(  M^{M}\otimes\mathcal{W}_{D^{3}%
}\right)  _{\mathrm{id}_{M}}\,\underrightarrow{\zeta^{\underset{1}%
{\overset{\cdot}{-}}}}\,\left(  M^{M}\otimes\mathcal{W}_{D^{2}}\right)
_{\mathrm{id}_{M}}\label{3.3.3.1}%
\end{align}
which is equivalent to
\begin{align}
& \left(  \underset{1}{M^{M}\otimes\mathcal{W}_{D}}\right)  _{\mathrm{id}_{M}%
}\times\left(  \underset{2}{M^{M}\otimes\mathcal{W}_{D}}\right)
_{\mathrm{id}_{M}}\times\left(  \underset{3}{M^{M}\otimes\mathcal{W}_{D}%
}\right)  _{\mathrm{id}_{M}}\nonumber\\
& \underrightarrow{\left(  \left(  \chi^{\ast_{32}},\chi^{\ast_{23}}\right)
,\chi^{\ast_{1}}\right)  }\nonumber\\
& \left(  \left(  M^{M}\otimes\mathcal{W}_{D^{2}}\right)  _{\mathrm{id}_{M}%
}\times_{M^{M}\otimes\mathcal{W}_{D\left(  2\right)  }}\left(  M^{M}%
\otimes\mathcal{W}_{D^{2}}\right)  _{\mathrm{id}_{M}}\right)  \times\left(
\underset{1}{M^{M}\otimes\mathcal{W}_{D}}\right)  _{\mathrm{id}_{M}%
}\nonumber\\
& \underrightarrow{\zeta^{\overset{\cdot}{-}}\times\mathrm{id}_{M^{M}%
\otimes\mathcal{W}_{D}}}\,\left(  M^{M}\otimes\mathcal{W}_{D}\right)
_{\mathrm{id}_{M}}\times\left(  \underset{1}{M^{M}\otimes\mathcal{W}_{D}%
}\right)  _{\mathrm{id}_{M}}\,\underrightarrow{\mathrm{Ass}_{M}^{1,1}%
}\,\left(  M^{M}\otimes\mathcal{W}_{D^{2}}\right)  _{\mathrm{id}_{M}%
}\nonumber\\
& \underrightarrow{\mathrm{id}_{M^{M}}\otimes\mathcal{W}_{\left(  d_{1}%
,d_{2}\right)  \in D^{2}\mapsto\left(  d_{2},d_{1}\right)  \in D^{2}}%
}\,\left(  M^{M}\otimes\mathcal{W}_{D^{2}}\right)  _{\mathrm{id}_{M}%
}\label{3.3.3.2}%
\end{align}

\item The composition of morphisms $\chi^{\ast_{132}}$ and
\[
\left(  M^{M}\otimes\mathcal{W}_{D^{3}}\right)  _{\mathrm{id}_{M}}%
\rightarrow\left(  M^{M}\otimes\mathcal{W}_{D^{3}\{(2,3)\}}\right)
_{\mathrm{id}_{M}}%
\]
is equivalent to the composition of morphisms $\chi^{\ast_{123}}$ and
\[
\left(  M^{M}\otimes\mathcal{W}_{D^{3}}\right)  _{\mathrm{id}_{M}}%
\rightarrow\left(  M^{M}\otimes\mathcal{W}_{D^{3}\{(2,3)\}}\right)
_{\mathrm{id}_{M}}%
\]
so that we have
\begin{align}
& \left(  \underset{1}{M^{M}\otimes\mathcal{W}_{D}}\right)  _{\mathrm{id}_{M}%
}\times\left(  \underset{2}{M^{M}\otimes\mathcal{W}_{D}}\right)
_{\mathrm{id}_{M}}\times\left(  \underset{3}{M^{M}\otimes\mathcal{W}_{D}%
}\right)  _{\mathrm{id}_{M}}\nonumber\\
& \underrightarrow{\left(  \chi^{\ast_{132}},\chi^{\ast_{123}}\right)
}\,\left(  M^{M}\otimes\mathcal{W}_{D^{3}}\right)  _{\mathrm{id}_{M}}%
\times_{M^{M}\otimes\mathcal{W}_{D^{3}\{(2,3)\}}}\left(  M^{M}\otimes
\mathcal{W}_{D^{3}}\right)  _{\mathrm{id}_{M}}\,\nonumber\\
& \underrightarrow{\zeta^{\underset{1}{\overset{\cdot}{-}}}}\,\left(
M^{M}\otimes\mathcal{W}_{D^{2}}\right)  _{\mathrm{id}_{M}}\label{3.3.3.3}%
\end{align}
which is equivalent to
\begin{align}
& \left(  \underset{1}{M^{M}\otimes\mathcal{W}_{D}}\right)  _{\mathrm{id}_{M}%
}\times\left(  \underset{2}{M^{M}\otimes\mathcal{W}_{D}}\right)
_{\mathrm{id}_{M}}\times\left(  \underset{3}{M^{M}\otimes\mathcal{W}_{D}%
}\right)  _{\mathrm{id}_{M}}\nonumber\\
& \underrightarrow{\left(  \chi^{\ast_{1}},\left(  \chi^{\ast_{32}},\chi
^{\ast_{23}}\right)  \right)  }\nonumber\\
& \left(  \underset{1}{M^{M}\otimes\mathcal{W}_{D}}\right)  _{\mathrm{id}_{M}%
}\times\left(  \left(  M^{M}\otimes\mathcal{W}_{D^{2}}\right)  _{\mathrm{id}%
_{M}}\times_{M^{M}\otimes\mathcal{W}_{D\left(  2\right)  }}\left(
M^{M}\otimes\mathcal{W}_{D^{2}}\right)  _{\mathrm{id}_{M}}\right) \nonumber\\
& \underrightarrow{\mathrm{id}_{M^{M}\otimes\mathcal{W}_{D}}\times
\zeta^{\overset{\cdot}{-}}}\,\left(  \underset{1}{M^{M}\otimes\mathcal{W}_{D}%
}\right)  _{\mathrm{id}_{M}}\times\left(  M^{M}\otimes\mathcal{W}_{D}\right)
_{\mathrm{id}_{M}}\,\underrightarrow{\mathrm{Ass}_{M}^{1,1}}\,\left(
M^{M}\otimes\mathcal{W}_{D^{2}}\right)  _{\mathrm{id}_{M}}\label{3.3.3.4}%
\end{align}

\item The composition of (\ref{3.3.3.2}) and
\[
\left(  M^{M}\otimes\mathcal{W}_{D^{2}}\right)  _{\mathrm{id}_{M}}%
\rightarrow\left(  M^{M}\otimes\mathcal{W}_{D\left(  2\right)  }\right)
_{\mathrm{id}_{M}}%
\]
is equivalent to the composition of (\ref{3.3.3.4}) and
\[
\left(  M^{M}\otimes\mathcal{W}_{D^{2}}\right)  _{\mathrm{id}_{M}}%
\rightarrow\left(  M^{M}\otimes\mathcal{W}_{D\left(  2\right)  }\right)
_{\mathrm{id}_{M}}%
\]
so that we have
\begin{align}
& \left(  \underset{1}{M^{M}\otimes\mathcal{W}_{D}}\right)  _{\mathrm{id}_{M}%
}\times\left(  \underset{2}{M^{M}\otimes\mathcal{W}_{D}}\right)
_{\mathrm{id}_{M}}\times\left(  \underset{3}{M^{M}\otimes\mathcal{W}_{D}%
}\right)  _{\mathrm{id}_{M}}\nonumber\\
& \underrightarrow{\left(  \left(  \chi^{\ast_{321}},\chi^{\ast_{231}}\right)
,\left(  \chi^{\ast_{132}},\chi^{\ast_{123}}\right)  \right)  }\,\nonumber\\
& \left(
\begin{array}
[c]{c}%
\left(  \underset{321}{M^{M}\otimes\mathcal{W}_{D^{3}}}\right)  _{\mathrm{id}%
_{M}}\\
\times_{M^{M}\otimes\mathcal{W}_{D^{3}\left\{  \left(  2,3\right)  \right\}
}}\\
\left(  \underset{231}{M^{M}\otimes\mathcal{W}_{D^{3}}}\right)  _{\mathrm{id}%
_{M}}%
\end{array}
\right)  \times_{M^{M}\otimes\mathcal{W}_{D\left(  2\right)  }}\left(
\begin{array}
[c]{c}%
\left(  \underset{132}{M^{M}\otimes\mathcal{W}_{D^{3}}}\right)  _{\mathrm{id}%
_{M}}\\
\times_{M^{M}\otimes\mathcal{W}_{D^{3}\left\{  \left(  2,3\right)  \right\}
}}\\
\left(  \underset{123}{M^{M}\otimes\mathcal{W}_{D^{3}}}\right)  _{\mathrm{id}%
_{M}}%
\end{array}
\right) \nonumber\\
& \underrightarrow{\zeta^{\underset{2}{\overset{\cdot}{-}}}\times
_{M^{M}\otimes\mathcal{W}_{D\left(  2\right)  }}\zeta^{\underset{2}%
{\overset{\cdot}{-}}}}\,\left(  M^{M}\otimes\mathcal{W}_{D^{2}}\right)
_{\mathrm{id}_{M}}\times_{M^{M}\otimes\mathcal{W}_{D\left(  2\right)  }%
}\left(  M^{M}\otimes\mathcal{W}_{D^{2}}\right)  _{\mathrm{id}_{M}%
}\,\underrightarrow{\zeta^{\overset{\cdot}{-}}}\,\left(  M^{M}\otimes
\mathcal{W}_{D}\right)  _{\mathrm{id}_{M}}\label{3.3.3.5}%
\end{align}
which is equivalent to
\begin{align}
& \left(  \underset{1}{M^{M}\otimes\mathcal{W}_{D}}\right)  _{\mathrm{id}_{M}%
}\times\left(  \underset{2}{M^{M}\otimes\mathcal{W}_{D}}\right)
_{\mathrm{id}_{M}}\times\left(  \underset{3}{M^{M}\otimes\mathcal{W}_{D}%
}\right)  _{\mathrm{id}_{M}}\,\underrightarrow{\mathrm{id}_{\left(
\underset{1}{M^{M}\otimes\mathcal{W}_{D}}\right)  _{\mathrm{id}_{M}}}\times
L_{M}}\nonumber\\
& \left(  \underset{1}{M^{M}\otimes\mathcal{W}_{D}}\right)  _{\mathrm{id}_{M}%
}\times\left(  M^{M}\otimes\mathcal{W}_{D}\right)  _{\mathrm{id}_{M}%
}\,\underrightarrow{L_{M}}\,\left(  M^{M}\otimes\mathcal{W}_{D}\right)
_{\mathrm{id}_{M}}\label{3.3.3.6}%
\end{align}

\end{enumerate}
\end{lemma}

\begin{proof}
The first and the second statements follow from Proposition \ref{t3.3.2}. The
last statement follows from Theorem \ref{t3.3.1}.
\end{proof}

\begin{lemma}
\label{t3.3.4}We have the following statements:

\begin{enumerate}
\item The composition of morphisms $\chi^{\ast_{132}}$ and
\[
\left(  M^{M}\otimes\mathcal{W}_{D^{3}}\right)  _{\mathrm{id}_{M}}%
\rightarrow\left(  M^{M}\otimes\mathcal{W}_{D^{3}\{(1,3)\}}\right)
_{\mathrm{id}_{M}}%
\]
in succession is equivalent to the composition of morphisms $\chi^{\ast_{312}%
}$ and
\[
\left(  M^{M}\otimes\mathcal{W}_{D^{3}}\right)  _{\mathrm{id}_{M}}%
\rightarrow\left(  M^{M}\otimes\mathcal{W}_{D^{3}\{(1,3)\}}\right)
_{\mathrm{id}_{M}}%
\]
in succession, so that we have
\begin{align}
& \left(  \underset{1}{M^{M}\otimes\mathcal{W}_{D}}\right)  _{\mathrm{id}_{M}%
}\times\left(  \underset{2}{M^{M}\otimes\mathcal{W}_{D}}\right)
_{\mathrm{id}_{M}}\times\left(  \underset{3}{M^{M}\otimes\mathcal{W}_{D}%
}\right)  _{\mathrm{id}_{M}}\nonumber\\
& \underrightarrow{\left(  \chi^{\ast_{132}},\chi^{\ast_{312}}\right)
}\,\left(  M^{M}\otimes\mathcal{W}_{D^{3}}\right)  _{\mathrm{id}_{M}}%
\times_{M^{M}\otimes\mathcal{W}_{D^{3}\{(1,3)\}}}\left(  M^{M}\otimes
\mathcal{W}_{D^{3}}\right)  _{\mathrm{id}_{M}}\,\underrightarrow
{\zeta^{\underset{2}{\overset{\cdot}{-}}}}\,\left(  M^{M}\otimes
\mathcal{W}_{D^{2}}\right)  _{\mathrm{id}_{M}}\label{3.3.4.1}%
\end{align}
which is equivalent to
\begin{align}
& \left(  \underset{1}{M^{M}\otimes\mathcal{W}_{D}}\right)  _{\mathrm{id}_{M}%
}\times\left(  \underset{2}{M^{M}\otimes\mathcal{W}_{D}}\right)
_{\mathrm{id}_{M}}\times\left(  \underset{3}{M^{M}\otimes\mathcal{W}_{D}%
}\right)  _{\mathrm{id}_{M}}\nonumber\\
& \underrightarrow{\left(  \left(  \chi^{\ast_{13}},\chi^{\ast_{31}}\right)
,\chi^{\ast_{2}}\right)  }\nonumber\\
& \left(  \left(  M^{M}\otimes\mathcal{W}_{D^{2}}\right)  _{\mathrm{id}_{M}%
}\times_{M^{M}\otimes\mathcal{W}_{D\left(  2\right)  }}\left(  M^{M}%
\otimes\mathcal{W}_{D^{2}}\right)  _{\mathrm{id}_{M}}\right)  \times\left(
\underset{2}{M^{M}\otimes\mathcal{W}_{D}}\right)  _{\mathrm{id}_{M}%
}\nonumber\\
& \underrightarrow{\zeta^{\overset{\cdot}{-}}\times\mathrm{id}_{M^{M}%
\otimes\mathcal{W}_{D}}}\,\left(  M^{M}\otimes\mathcal{W}_{D}\right)
_{\mathrm{id}_{M}}\times\left(  \underset{2}{M^{M}\otimes\mathcal{W}_{D}%
}\right)  _{\mathrm{id}_{M}}\,\underrightarrow{\mathrm{Ass}_{M}^{1,1}%
}\,\left(  M^{M}\otimes\mathcal{W}_{D^{2}}\right)  _{\mathrm{id}_{M}%
}\nonumber\\
& \underrightarrow{\mathrm{id}_{M^{M}}\otimes\mathcal{W}_{\left(  d_{1}%
,d_{2}\right)  \in D^{2}\mapsto\left(  d_{2},d_{1}\right)  \in D^{2}}%
}\,\left(  M^{M}\otimes\mathcal{W}_{D^{2}}\right)  _{\mathrm{id}_{M}%
}\label{3.3.4.2}%
\end{align}

\item The composition of morphisms $\chi^{\ast_{213}}$ and
\[
\left(  M^{M}\otimes\mathcal{W}_{D^{3}}\right)  _{\mathrm{id}_{M}}%
\rightarrow\left(  M^{M}\otimes\mathcal{W}_{D^{3}\{(1,3)\}}\right)
_{\mathrm{id}_{M}}%
\]
is equivalent to the composition of morphisms $\chi^{\ast_{231}}$ and
\[
\left(  M^{M}\otimes\mathcal{W}_{D^{3}}\right)  _{\mathrm{id}_{M}}%
\rightarrow\left(  M^{M}\otimes\mathcal{W}_{D^{3}\{(1,3)\}}\right)
_{\mathrm{id}_{M}}%
\]
in succession, so that we have
\begin{align}
& \left(  \underset{1}{M^{M}\otimes\mathcal{W}_{D}}\right)  _{\mathrm{id}_{M}%
}\times\left(  \underset{2}{M^{M}\otimes\mathcal{W}_{D}}\right)
_{\mathrm{id}_{M}}\times\left(  \underset{3}{M^{M}\otimes\mathcal{W}_{D}%
}\right)  _{\mathrm{id}_{M}}\nonumber\\
& \underrightarrow{\left(  \chi^{\ast_{213}},\chi^{\ast_{231}}\right)
}\,\left(  M^{M}\otimes\mathcal{W}_{D^{3}}\right)  _{\mathrm{id}_{M}}%
\times_{M^{M}\otimes\mathcal{W}_{D^{3}\{(2,3)\}}}\left(  M^{M}\otimes
\mathcal{W}_{D^{3}}\right)  _{\mathrm{id}_{M}}\,\underrightarrow
{\zeta^{\underset{2}{\overset{\cdot}{-}}}}\,\left(  M^{M}\otimes
\mathcal{W}_{D^{2}}\right)  _{\mathrm{id}_{M}}\label{3.3.4.3}%
\end{align}
which is equivalent to
\begin{align}
& \left(  \underset{1}{M^{M}\otimes\mathcal{W}_{D}}\right)  _{\mathrm{id}_{M}%
}\times\left(  \underset{2}{M^{M}\otimes\mathcal{W}_{D}}\right)
_{\mathrm{id}_{M}}\times\left(  \underset{3}{M^{M}\otimes\mathcal{W}_{D}%
}\right)  _{\mathrm{id}_{M}}\nonumber\\
& \underrightarrow{\left(  \chi^{\ast_{2}},\left(  \chi^{\ast_{13}},\chi
^{\ast_{31}}\right)  \right)  }\nonumber\\
& \left(  \underset{1}{M^{M}\otimes\mathcal{W}_{D}}\right)  _{\mathrm{id}_{M}%
}\times\left(  \left(  M^{M}\otimes\mathcal{W}_{D^{2}}\right)  _{\mathrm{id}%
_{M}}\times_{M^{M}\otimes\mathcal{W}_{D\left(  2\right)  }}\left(
M^{M}\otimes\mathcal{W}_{D^{2}}\right)  _{\mathrm{id}_{M}}\right) \nonumber\\
& \underrightarrow{\mathrm{id}_{M^{M}\otimes\mathcal{W}_{D}}\times
\zeta^{\overset{\cdot}{-}}}\,\left(  \underset{1}{M^{M}\otimes\mathcal{W}_{D}%
}\right)  _{\mathrm{id}_{M}}\times\left(  M^{M}\otimes\mathcal{W}_{D}\right)
_{\mathrm{id}_{M}}\,\underrightarrow{\mathrm{Ass}_{M}^{1,1}}\,\left(
M^{M}\otimes\mathcal{W}_{D^{2}}\right)  _{\mathrm{id}_{M}}\label{3.3.4.4}%
\end{align}

\item The composition of (\ref{3.3.4.2}) and
\[
\left(  M^{M}\otimes\mathcal{W}_{D^{2}}\right)  _{\mathrm{id}_{M}}%
\rightarrow\left(  M^{M}\otimes\mathcal{W}_{D\left(  2\right)  }\right)
_{\mathrm{id}_{M}}%
\]
is equivalent to the composition of (\ref{3.3.4.4}) and
\[
\left(  M^{M}\otimes\mathcal{W}_{D^{2}}\right)  _{\mathrm{id}_{M}}%
\rightarrow\left(  M^{M}\otimes\mathcal{W}_{D\left(  2\right)  }\right)
_{\mathrm{id}_{M}}%
\]
so that we have
\begin{align}
& \left(  \underset{1}{M^{M}\otimes\mathcal{W}_{D}}\right)  _{\mathrm{id}_{M}%
}\times\left(  \underset{2}{M^{M}\otimes\mathcal{W}_{D}}\right)
_{\mathrm{id}_{M}}\times\left(  \underset{3}{M^{M}\otimes\mathcal{W}_{D}%
}\right)  _{\mathrm{id}_{M}}\nonumber\\
& \underrightarrow{\left(  \left(  \chi^{\ast_{132}},\chi^{\ast_{312}}\right)
,\left(  \chi^{\ast_{213}},\chi^{\ast_{231}}\right)  \right)  }\,\nonumber\\
& \left(
\begin{array}
[c]{c}%
\left(  \underset{132}{M^{M}\otimes\mathcal{W}_{D^{3}}}\right)  _{\mathrm{id}%
_{M}}\\
\times_{M^{M}\otimes\mathcal{W}_{D^{3}\left\{  \left(  1,,3\right)  \right\}
}}\\
\left(  \underset{312}{M^{M}\otimes\mathcal{W}_{D^{3}}}\right)  _{\mathrm{id}%
_{M}}%
\end{array}
\right)  \times_{M^{M}\otimes\mathcal{W}_{D\left(  2\right)  }}\left(
\begin{array}
[c]{c}%
\left(  \underset{213}{M^{M}\otimes\mathcal{W}_{D^{3}}}\right)  _{\mathrm{id}%
_{M}}\\
\times_{M^{M}\otimes\mathcal{W}_{D^{3}\left\{  \left(  1,3\right)  \right\}
}}\\
\left(  \underset{231}{M^{M}\otimes\mathcal{W}_{D^{3}}}\right)  _{\mathrm{id}%
_{M}}%
\end{array}
\right) \nonumber\\
& \underrightarrow{\zeta^{\underset{2}{\overset{\cdot}{-}}}\times
_{M^{M}\otimes\mathcal{W}_{D\left(  2\right)  }}\zeta^{\underset{2}%
{\overset{\cdot}{-}}}}\,\left(  M^{M}\otimes\mathcal{W}_{D^{2}}\right)
_{\mathrm{id}_{M}}\times_{M^{M}\otimes\mathcal{W}_{D\left(  2\right)  }%
}\left(  M^{M}\otimes\mathcal{W}_{D^{2}}\right)  _{\mathrm{id}_{M}%
}\,\underrightarrow{\zeta^{\overset{\cdot}{-}}}\,\left(  M^{M}\otimes
\mathcal{W}_{D}\right)  _{\mathrm{id}_{M}}\label{3.3.4.5}%
\end{align}
which is equivalent to
\begin{align}
& \left(  \underset{1}{M^{M}\otimes\mathcal{W}_{D}}\right)  _{\mathrm{id}_{M}%
}\times\left(  \underset{2}{M^{M}\otimes\mathcal{W}_{D}}\right)
_{\mathrm{id}_{M}}\times\left(  \underset{3}{M^{M}\otimes\mathcal{W}_{D}%
}\right)  _{\mathrm{id}_{M}}\,\nonumber\\
& \left(  \underset{2}{M^{M}\otimes\mathcal{W}_{D}}\right)  _{\mathrm{id}_{M}%
}\times\left(  \underset{3}{M^{M}\otimes\mathcal{W}_{D}}\right)
_{\mathrm{id}_{M}}\times\left(  \underset{1}{M^{M}\otimes\mathcal{W}_{D}%
}\right)  _{\mathrm{id}_{M}}\,\underrightarrow{\mathrm{id}_{\left(
\underset{2}{M^{M}\otimes\mathcal{W}_{D}}\right)  _{\mathrm{id}_{M}}}\times
L_{M}}\nonumber\\
& \left(  \underset{2}{M^{M}\otimes\mathcal{W}_{D}}\right)  _{\mathrm{id}_{M}%
}\times\left(  M^{M}\otimes\mathcal{W}_{D}\right)  _{\mathrm{id}_{M}%
}\,\underrightarrow{L_{M}}\,\left(  M^{M}\otimes\mathcal{W}_{D}\right)
_{\mathrm{id}_{M}}\label{3.3.4.6}%
\end{align}

\end{enumerate}
\end{lemma}

\begin{proof}
The first and the second statements follow from Proposition \ref{t3.3.2}. The
last statement follows from Theorem \ref{t3.3.1}.
\end{proof}

\begin{lemma}
\label{t3.3.5}We have the following statements:

\begin{enumerate}
\item The composition of morphisms $\chi^{\ast_{213}}$ and
\[
\left(  M^{M}\otimes\mathcal{W}_{D^{3}}\right)  _{\mathrm{id}_{M}}%
\rightarrow\left(  M^{M}\otimes\mathcal{W}_{D^{3}\{(1,2)\}}\right)
_{\mathrm{id}_{M}}%
\]
is equivalent to the composition of morphisms $\chi^{\ast_{123}}$ and
\[
\left(  M^{M}\otimes\mathcal{W}_{D^{3}}\right)  _{\mathrm{id}_{M}}%
\rightarrow\left(  M^{M}\otimes\mathcal{W}_{D^{3}\{(1,2)\}}\right)
_{\mathrm{id}_{M}}%
\]
so that we have
\begin{align}
& \left(  \underset{1}{M^{M}\otimes\mathcal{W}_{D}}\right)  _{\mathrm{id}_{M}%
}\times\left(  \underset{2}{M^{M}\otimes\mathcal{W}_{D}}\right)
_{\mathrm{id}_{M}}\times\left(  \underset{3}{M^{M}\otimes\mathcal{W}_{D}%
}\right)  _{\mathrm{id}_{M}}\nonumber\\
& \underrightarrow{\left(  \chi^{\ast_{213}},\chi^{\ast_{123}}\right)
}\,\left(  M^{M}\otimes\mathcal{W}_{D^{3}}\right)  _{\mathrm{id}_{M}}%
\times_{M^{M}\otimes\mathcal{W}_{D^{3}\{(2,3)\}}}\left(  M^{M}\otimes
\mathcal{W}_{D^{3}}\right)  _{\mathrm{id}_{M}}\,\nonumber\\
& \underrightarrow{\zeta^{\underset{3}{\overset{\cdot}{-}}}}\,\left(
M^{M}\otimes\mathcal{W}_{D^{2}}\right)  _{\mathrm{id}_{M}}\label{3.3.5.1}%
\end{align}
which is equivalent to
\begin{align}
& \left(  \underset{1}{M^{M}\otimes\mathcal{W}_{D}}\right)  _{\mathrm{id}_{M}%
}\times\left(  \underset{2}{M^{M}\otimes\mathcal{W}_{D}}\right)
_{\mathrm{id}_{M}}\times\left(  \underset{3}{M^{M}\otimes\mathcal{W}_{D}%
}\right)  _{\mathrm{id}_{M}}\nonumber\\
& \underrightarrow{\left(  \left(  \chi^{\ast_{31}},\chi^{\ast_{12}}\right)
,\chi^{\ast_{3}}\right)  }\nonumber\\
& \left(  \left(  M^{M}\otimes\mathcal{W}_{D^{2}}\right)  _{\mathrm{id}_{M}%
}\times_{M^{M}\otimes\mathcal{W}_{D\left(  2\right)  }}\left(  M^{M}%
\otimes\mathcal{W}_{D^{2}}\right)  _{\mathrm{id}_{M}}\right)  \times\left(
\underset{1}{M^{M}\otimes\mathcal{W}_{D}}\right)  _{\mathrm{id}_{M}%
}\nonumber\\
& \underrightarrow{\zeta^{\overset{\cdot}{-}}\times\mathrm{id}_{M^{M}%
\otimes\mathcal{W}_{D}}}\,\left(  M^{M}\otimes\mathcal{W}_{D}\right)
_{\mathrm{id}_{M}}\times\left(  \underset{1}{M^{M}\otimes\mathcal{W}_{D}%
}\right)  _{\mathrm{id}_{M}}\,\underrightarrow{\mathrm{Ass}_{M}^{1,1}%
}\,\left(  M^{M}\otimes\mathcal{W}_{D^{2}}\right)  _{\mathrm{id}_{M}%
}\nonumber\\
& \underrightarrow{\mathrm{id}_{M^{M}}\otimes\mathcal{W}_{\left(  d_{1}%
,d_{2}\right)  \in D^{2}\mapsto\left(  d_{2},d_{1}\right)  \in D^{2}}%
}\,\left(  M^{M}\otimes\mathcal{W}_{D^{2}}\right)  _{\mathrm{id}_{M}%
}\label{3.3.5.2}%
\end{align}

\item The composition of morphisms $\chi^{\ast_{321}}$ and
\[
\left(  M^{M}\otimes\mathcal{W}_{D^{3}}\right)  _{\mathrm{id}_{M}}%
\rightarrow\left(  M^{M}\otimes\mathcal{W}_{D^{3}\{(2,3)\}}\right)
_{\mathrm{id}_{M}}%
\]
is equivalent to the composition of morphisms $\chi^{\ast_{312}}$ and
\[
\left(  M^{M}\otimes\mathcal{W}_{D^{3}}\right)  _{\mathrm{id}_{M}}%
\rightarrow\left(  M^{M}\otimes\mathcal{W}_{D^{3}\{(2,3)\}}\right)
_{\mathrm{id}_{M}}%
\]
so that we have
\begin{align}
& \left(  \underset{1}{M^{M}\otimes\mathcal{W}_{D}}\right)  _{\mathrm{id}_{M}%
}\times\left(  \underset{2}{M^{M}\otimes\mathcal{W}_{D}}\right)
_{\mathrm{id}_{M}}\times\left(  \underset{3}{M^{M}\otimes\mathcal{W}_{D}%
}\right)  _{\mathrm{id}_{M}}\nonumber\\
& \underrightarrow{\left(  \chi^{\ast_{321}},\chi^{\ast_{312}}\right)
}\,\left(  M^{M}\otimes\mathcal{W}_{D^{3}}\right)  _{\mathrm{id}_{M}}%
\times_{M^{M}\otimes\mathcal{W}_{D^{3}\{(2,3)\}}}\left(  M^{M}\otimes
\mathcal{W}_{D^{3}}\right)  _{\mathrm{id}_{M}}\,\underrightarrow
{\zeta^{\underset{3}{\overset{\cdot}{-}}}}\,\left(  M^{M}\otimes
\mathcal{W}_{D^{2}}\right)  _{\mathrm{id}_{M}}\label{3.3.5.3}%
\end{align}
which is equivalent to
\begin{align}
& \left(  \underset{1}{M^{M}\otimes\mathcal{W}_{D}}\right)  _{\mathrm{id}_{M}%
}\times\left(  \underset{2}{M^{M}\otimes\mathcal{W}_{D}}\right)
_{\mathrm{id}_{M}}\times\left(  \underset{3}{M^{M}\otimes\mathcal{W}_{D}%
}\right)  _{\mathrm{id}_{M}}\nonumber\\
& \underrightarrow{\left(  \chi^{\ast_{3}},\left(  \chi^{\ast_{21}},\chi
^{\ast_{12}}\right)  \right)  }\nonumber\\
& \left(  \underset{3}{M^{M}\otimes\mathcal{W}_{D}}\right)  _{\mathrm{id}_{M}%
}\times\left(  \left(  M^{M}\otimes\mathcal{W}_{D^{2}}\right)  _{\mathrm{id}%
_{M}}\times_{M^{M}\otimes\mathcal{W}_{D\left(  2\right)  }}\left(
M^{M}\otimes\mathcal{W}_{D^{2}}\right)  _{\mathrm{id}_{M}}\right) \nonumber\\
& \underrightarrow{\mathrm{id}_{M^{M}\otimes\mathcal{W}_{D}}\times
\zeta^{\overset{\cdot}{-}}}\,\left(  \underset{3}{M^{M}\otimes\mathcal{W}_{D}%
}\right)  _{\mathrm{id}_{M}}\times\left(  M^{M}\otimes\mathcal{W}_{D}\right)
\,\underrightarrow{\mathrm{Ass}_{M}^{1,1}}\,\left(  M^{M}\otimes
\mathcal{W}_{D^{2}}\right)  _{\mathrm{id}_{M}}\label{3.3.5.4}%
\end{align}

\item The composition of (\ref{3.3.5.2}) and
\[
\left(  M^{M}\otimes\mathcal{W}_{D^{2}}\right)  _{\mathrm{id}_{M}}%
\rightarrow\left(  M^{M}\otimes\mathcal{W}_{D\left(  2\right)  }\right)
_{\mathrm{id}_{M}}%
\]
is equivalent to the composition of (\ref{3.3.5.4}) and
\[
\left(  M^{M}\otimes\mathcal{W}_{D^{2}}\right)  _{\mathrm{id}_{M}}%
\rightarrow\left(  M^{M}\otimes\mathcal{W}_{D\left(  2\right)  }\right)
_{\mathrm{id}_{M}}%
\]
in succession, so that we have
\begin{align}
& \left(  \underset{1}{M^{M}\otimes\mathcal{W}_{D}}\right)  _{\mathrm{id}_{M}%
}\times\left(  \underset{2}{M^{M}\otimes\mathcal{W}_{D}}\right)
_{\mathrm{id}_{M}}\times\left(  \underset{3}{M^{M}\otimes\mathcal{W}_{D}%
}\right)  _{\mathrm{id}_{M}}\nonumber\\
& \underrightarrow{\left(  \left(  \chi^{\ast_{213}},\chi^{\ast_{123}}\right)
,\left(  \chi^{\ast_{321}},\chi^{\ast_{312}}\right)  \right)  }\,\nonumber\\
& \left(
\begin{array}
[c]{c}%
\left(  \underset{213}{M^{M}\otimes\mathcal{W}_{D^{3}}}\right)  _{\mathrm{id}%
_{M}}\\
\times_{M^{M}\otimes\mathcal{W}_{D^{3}\left\{  \left(  1,,3\right)  \right\}
}}\\
\left(  \underset{_{123}}{M^{M}\otimes\mathcal{W}_{D^{3}}}\right)
_{\mathrm{id}_{M}}%
\end{array}
\right)  \times_{M^{M}\otimes\mathcal{W}_{D\left(  2\right)  }}\left(
\begin{array}
[c]{c}%
\left(  \underset{321}{M^{M}\otimes\mathcal{W}_{D^{3}}}\right)  _{\mathrm{id}%
_{M}}\\
\times_{M^{M}\otimes\mathcal{W}_{D^{3}\left\{  \left(  1,3\right)  \right\}
}}\\
\left(  \underset{312}{M^{M}\otimes\mathcal{W}_{D^{3}}}\right)  _{\mathrm{id}%
_{M}}%
\end{array}
\right) \nonumber\\
& \underrightarrow{\zeta^{\underset{3}{\overset{\cdot}{-}}}\times
_{M^{M}\otimes\mathcal{W}_{D\left(  2\right)  }}\zeta^{\underset{3}%
{\overset{\cdot}{-}}}}\,\left(  M^{M}\otimes\mathcal{W}_{D^{2}}\right)
_{\mathrm{id}_{M}}\times_{M^{M}\otimes\mathcal{W}_{D\left(  2\right)  }%
}\left(  M^{M}\otimes\mathcal{W}_{D^{2}}\right)  _{\mathrm{id}_{M}%
}\,\underrightarrow{\zeta^{\overset{\cdot}{-}}}\,\left(  M^{M}\otimes
\mathcal{W}_{D}\right)  _{\mathrm{id}_{M}}\label{3.3.5.5}%
\end{align}
which is equivalent to
\begin{align}
& \left(  \underset{1}{M^{M}\otimes\mathcal{W}_{D}}\right)  _{\mathrm{id}_{M}%
}\times\left(  \underset{2}{M^{M}\otimes\mathcal{W}_{D}}\right)
_{\mathrm{id}_{M}}\times\left(  \underset{3}{M^{M}\otimes\mathcal{W}_{D}%
}\right)  _{\mathrm{id}_{M}}\nonumber\\
& \left(  \underset{3}{M^{M}\otimes\mathcal{W}_{D}}\right)  _{\mathrm{id}_{M}%
}\times\left(  \underset{1}{M^{M}\otimes\mathcal{W}_{D}}\right)
_{\mathrm{id}_{M}}\times\left(  \underset{2}{M^{M}\otimes\mathcal{W}_{D}%
}\right)  _{\mathrm{id}_{M}}\,\underrightarrow{\mathrm{id}_{\left(
\underset{3}{M^{M}\otimes\mathcal{W}_{D}}\right)  _{\mathrm{id}_{M}}}\times
L_{M}}\nonumber\\
& \left(  \underset{3}{M^{M}\otimes\mathcal{W}_{D}}\right)  _{\mathrm{id}_{M}%
}\times\left(  M^{M}\otimes\mathcal{W}_{D}\right)  _{\mathrm{id}_{M}%
}\,\underrightarrow{L_{M}}\,\left(  M^{M}\otimes\mathcal{W}_{D}\right)
_{\mathrm{id}_{M}}\label{3.3.5.6}%
\end{align}

\end{enumerate}
\end{lemma}

\begin{proof}
The first and the second statements follow from Proposition \ref{t3.3.2}. The
last statement follows from Theorem \ref{t3.3.1}.
\end{proof}

\begin{theorem}
\label{t3.3.6}(\underline{The conventional Jacobi Identity}) We have the
following two statements:

\begin{enumerate}
\item The three morphisms (\ref{3.3.3.5}), (\ref{3.3.4.5}) and (\ref{3.3.5.5})
sum up only to vanish.

\item The three morphisms (\ref{3.3.3.6}), (\ref{3.3.4.6}) and (\ref{3.3.5.6})
sum up only to vanish.
\end{enumerate}
\end{theorem}

\begin{proof}
We apply the general Jacobi identity to the morphism
\begin{align*}
& \left(  \underset{1}{M^{M}\otimes\mathcal{W}_{D}}\right)  _{\mathrm{id}_{M}%
}\times\left(  \underset{2}{M^{M}\otimes\mathcal{W}_{D}}\right)
_{\mathrm{id}_{M}}\times\left(  \underset{3}{M^{M}\otimes\mathcal{W}_{D}%
}\right)  _{\mathrm{id}_{M}}\\
& \underrightarrow{\left[
\begin{array}
[c]{ccc}
&
\begin{array}
[c]{c}%
\left(
\begin{array}
[c]{c}%
\chi^{\ast_{321}}\\
\times_{M\otimes\mathcal{W}_{D^{3}\{(2,3)\}}}\\
\chi^{\ast_{231}}%
\end{array}
\right) \\
\times_{M\otimes\mathcal{W}_{D(2)}}\\
\left(
\begin{array}
[c]{c}%
\chi^{\ast_{132}}\\
\times_{M\otimes\mathcal{W}_{D^{3}\{(2,3)\}}}\\
\chi^{\ast_{123}}%
\end{array}
\right)
\end{array}
& \\%
\begin{array}
[c]{c}%
\times_{M\otimes\mathcal{W}_{D^{3}\oplus D^{3}}}\\
\,\\
\,
\end{array}
&  &
\begin{array}
[c]{c}%
\times_{M\otimes\mathcal{W}_{D^{3}\oplus D^{3}}}\\
\,\\
\,
\end{array}
\\%
\begin{array}
[c]{c}%
\left(
\begin{array}
[c]{c}%
\chi^{\ast_{132}}\\
\times_{M\otimes\mathcal{W}_{D^{3}\{(1,3)\}}}\\
\chi^{\ast_{312}}%
\end{array}
\right) \\
\times_{M\otimes\mathcal{W}_{D(2)}}\\
\left(
\begin{array}
[c]{c}%
\chi^{\ast_{213}}\\
\times_{M\otimes\mathcal{W}_{D^{3}\{(1,3)\}}}\\
\chi^{\ast_{231}}%
\end{array}
\right)
\end{array}
& \times_{M\otimes\mathcal{W}_{D^{3}\oplus D^{3}}} &
\begin{array}
[c]{c}%
\left(
\begin{array}
[c]{c}%
\chi^{\ast_{213}}\\
\times_{M\otimes\mathcal{W}_{D^{3}\{(1,2)\}}}\\
\chi^{\ast_{123}}%
\end{array}
\right) \\
\times_{M\otimes\mathcal{W}_{D(2)}}\\
\left(
\begin{array}
[c]{c}%
\chi^{\ast_{321}}\\
\times_{M\otimes\mathcal{W}_{D^{3}\{(1,2)\}}}\\
\chi^{\ast_{312}}%
\end{array}
\right)
\end{array}
\end{array}
\right]  }\\
& \left[
\begin{array}
[c]{ccc}
&
\begin{array}
[c]{c}%
\left(
\begin{array}
[c]{c}%
\underset{321}{\left(  M\otimes\mathcal{W}_{D^{3}}\right)  }\\
\times_{M\otimes\mathcal{W}_{D^{3}\{(2,3)\}}}\\
\underset{231}{\left(  M\otimes\mathcal{W}_{D^{3}}\right)  }%
\end{array}
\right) \\
\times_{M\otimes\mathcal{W}_{D(2)}}\\
\left(
\begin{array}
[c]{c}%
\underset{132}{\left(  M\otimes\mathcal{W}_{D^{3}}\right)  }\\
\times_{M\otimes\mathcal{W}_{D^{3}\{(2,3)\}}}\\
\underset{123}{\left(  M\otimes\mathcal{W}_{D^{3}}\right)  }%
\end{array}
\right)
\end{array}
& \\%
\begin{array}
[c]{c}%
\times_{M\otimes\mathcal{W}_{D^{3}\oplus D^{3}}}\\
\,\\
\,
\end{array}
&  &
\begin{array}
[c]{c}%
\times_{M\otimes\mathcal{W}_{D^{3}\oplus D^{3}}}\\
\,\\
\,
\end{array}
\\%
\begin{array}
[c]{c}%
\left(
\begin{array}
[c]{c}%
\underset{132}{\left(  M\otimes\mathcal{W}_{D^{3}}\right)  }\\
\times_{M\otimes\mathcal{W}_{D^{3}\{(1,3)\}}}\\
\underset{312}{\left(  M\otimes\mathcal{W}_{D^{3}}\right)  }%
\end{array}
\right) \\
\times_{M\otimes\mathcal{W}_{D(2)}}\\
\left(
\begin{array}
[c]{c}%
\underset{213}{\left(  M\otimes\mathcal{W}_{D^{3}}\right)  }\\
\times_{M\otimes\mathcal{W}_{D^{3}\{(1,3)\}}}\\
\underset{231}{\left(  M\otimes\mathcal{W}_{D^{3}}\right)  }%
\end{array}
\right)
\end{array}
& \times_{M\otimes\mathcal{W}_{D^{3}\oplus D^{3}}} &
\begin{array}
[c]{c}%
\left(
\begin{array}
[c]{c}%
\underset{213}{\left(  M\otimes\mathcal{W}_{D^{3}}\right)  }\\
\times_{M\otimes\mathcal{W}_{D^{3}\{(1,2)\}}}\\
\underset{123}{\left(  M\otimes\mathcal{W}_{D^{3}}\right)  }%
\end{array}
\right) \\
\times_{M\otimes\mathcal{W}_{D(2)}}\\
\left(
\begin{array}
[c]{c}%
\underset{321}{\left(  M\otimes\mathcal{W}_{D^{3}}\right)  }\\
\times_{M\otimes\mathcal{W}_{D^{3}\{(1,2)\}}}\\
\underset{312}{\left(  M\otimes\mathcal{W}_{D^{3}}\right)  }%
\end{array}
\right)
\end{array}
\end{array}
\right]
\end{align*}
so as to obtain the first statement. The second statement follows directly
from the first by Lemmas \ref{t3.3.3}, \ref{t3.3.4}\ and \ref{t3.3.5}.
\end{proof}

\end{document}